\documentclass[12pt]{amsart}

\pdfoutput=1

\usepackage{latexsym}
\usepackage[centertags]{amsmath}
\usepackage{amsfonts}
\usepackage{amssymb}
\usepackage{amsthm}
\usepackage{newlfont}
\usepackage{graphics}
\usepackage{color}

\usepackage[usenames,dvipsnames]{xcolor}

\usepackage[demo]{graphicx} 
\usepackage{wrapfig,floatrow}
\usepackage{lipsum} 
\usepackage[font=small,labelfont=bf,margin=5mm]{caption}

\RequirePackage{xparse, graphicx, caption, picins}
\DeclareDocumentCommand \addpic{O{0.4\textwidth} m g}{\parpic[r]{%
\begin{minipage}{#1}
    \includegraphics[width=\textwidth]{#2}%
    \IfNoValueTF{#3}{}{\captionof{figure}{\footnotesize #3}}
\end{minipage}
}}

\usepackage{amsmath}
\usepackage{amssymb}
\usepackage{amscd}
\usepackage{latexsym}
\usepackage{amsthm}
\usepackage{mathrsfs}
\usepackage{color}
\newtheorem{thm}{Theorem}[section]
\newtheorem*{thm*}{Theorem}
\newtheorem{cor}[thm]{Corollary}
\newtheorem*{cor*}{Corollary}
\newtheorem{lem}[thm]{Lemma}
\newtheorem*{lem*}{Lemma}
\newtheorem{prop}[thm]{Proposition}
\newtheorem*{prop*}{Proposition}

\theoremstyle{definition}
\newtheorem{defn}{Definition}[section]
\newtheorem*{defn*}{Definition}
\newtheorem*{conjecture*}{Conjecture}
\newtheorem{conjecture}{Conjecture}[section]

\theoremstyle{remark}
\newtheorem{rem}{Remark}[section]
\newtheorem*{rem*}{Remark}
\newtheorem{example}{Example}[section]

\newtheorem*{problem*}{Problem}

\newcommand{\Q}{\mathbb Q}
\newcommand{\N}{\mathbb N}

\newcommand{\Z}{\mathbb Z}

\renewcommand{\AA}{\mathbb A}

\newcommand{\DD}{\mathbb D}

\newcommand{\eps}{\epsilon}

\DeclareMathOperator{\Id}{Id}
\DeclareMathOperator{\touch}{touch}

\DeclareMathOperator{\Dinv}{Dinv}
\DeclareMathOperator{\dinv}{dinv}

\DeclareMathOperator{\pExp}{Exp}

\DeclareMathOperator{\Sym}{Sym}
\DeclareMathOperator{\inv}{inv}
\DeclareMathOperator{\Inv}{Inv}
\DeclareMathOperator{\area}{area}
\DeclareMathOperator{\Area}{Area}

\DeclareMathOperator{\bounce}{bounce}

\renewcommand{\Im}{\mathrm{Im}}

\usepackage[style=alphabetic,isbn=false,url=false,maxnames=5,firstinits=true]{biblatex}

\usepackage{amsfonts}
\usepackage{enumitem}
\usepackage{tikz}
\usepackage{tikz-cd}
\usepackage{mathabx,epsfig}
\usetikzlibrary{arrows, matrix}
\renewbibmacro{in:}{}

\def\acts{\mathrel{\reflectbox{$\righttoleftarrow$}}}

\title[Lecture Notes on Shuffle Conjecture]{Lecture notes on the Carlsson-Mellit proof of the shuffle conjecture}
\author{James Haglund$^1$ and Guoce Xin$^2$}

\address{ $^1$Department of Mathematics, University of Pennsylvania, Philadelphia, PA 19104
 \\
$^2$School of Mathematical Sciences, Capital Normal University,
Beijing 100048, PR China}

\email{$^1$\texttt{jhaglund@math.upenn.edu}\ \  \& $^2$\texttt{guoce.xin@gmail.com}}

\date{May 22, 2017}
\begin{document}

\begin{abstract}
  This note is based on the original proof of the shuffle conjecture by Carlsson and Mellit (arXiv:1508.06239, version 2), which seems to be too concise for the combinatorial community.
James Haglund spent a semester to check through the proof line by line and took explicit note. Guoce Xin was asked by Adriano Garsia to elaborate the proof. The whole paper is thoroughly checked and filled in with details. \textbf{Section 1 and the abstract of the original paper is included. Section 2 is rewritten for clarity. Other sections are filled with detailed proof and computation.} We hope this note is accessible to graduate students.
\end{abstract}

\maketitle

\vspace{-5mm}
\tableofcontents
\section*{Original Abstract}
We present a proof of the compositional shuffle conjecture
\cite{haglund2012compositional},
which generalizes the famous shuffle conjecture
for the character of the diagonal coinvariant algebra
\cite{haglund2005diagcoinv}.
We first formulate the
combinatorial side of the conjecture in terms
of certain operators on a graded vector space $V_*$
whose degree zero part is the ring of symmetric
functions $\Sym[X]$ over $\mathbb{Q}(q,t)$.
We then extend these operators to
an action of an algebra $\tilde{\AA}$ acting on this space,
and interpret the right generalization
of the $\nabla$ using an involution of the algebra
which is antilinear with
respect to the conjugation $(q,t)\mapsto (q^{-1},t^{-1})$.

\section{Original Introduction}

The shuffle conjecture
of Haglund, Haiman, Loehr, Remmel, and Ulyanov
\cite{haglund2005diagcoinv} predicts a combinatorial
formula for the Frobenius
character $\mathcal{F}_{R_n}(X;q,t)$ of the diagonal coinvariant
algebra $R_n$ in $n$ pairs of variables, which is a symmetric function
in infinitely many variables with coefficients in
 $\mathbb{Z}_{\geq 0}[q,t]$.
By a result of Haiman \cite{Hai02},
the Frobenius character is given explicitly by
\[\mathcal{F}_{R_n}(X;q,t)=(-1)^n \nabla e_n[X],\]
where up to a sign convention,
$\nabla$ is the operator which is diagonal
in the modified Macdonald basis defined in
\cite{Bergeron99identitiesand}.
The original shuffle conjecture states
\begin{equation}
\label{shuffconj}
(-1)^n \nabla e_n[X] =
\sum_{\pi} \sum_{w \in \mathcal{WP}_\pi} t^{\area(\pi)} q^{\dinv(\pi,w)} x_w.
\end{equation}
Here $\pi$ is a Dyck path of length $n$, and $w$ is some
extra data called a ``word parking function''
depending on $\pi$. The functions $(\area,\dinv)$
are statistics associated to a Dyck path and a parking function, and $x_w$
is a monomial in the variables $x$.
They proved that this sum,
denoted $D_n(X;q,t)$, is symmetric in the
$x$ variables and so does at least define a symmetric
function.
They furthermore showed that it included many previous
conjectures and results about the $q,t$-Catalan numbers,
and other special cases
\cite{garsia1996remarkable,
garsia2002catalan,haglund2003conjectured,
egge2003sch,haglund2004sch}.
Remarkably, $D_n(X;q,t)$ had not even been
proven to be symmetric in the $q,t$ variables until now,
even though the symmetry of $\mathcal{F}_{R_n}(X;q,t)$ is obvious.
For a thorough introduction to this topic, see
Haglund's book \cite{haglund2008catalan}.


In \cite{haglund2012compositional}
Haglund, Morse, and Zabrocki conjectured
a refinement of the original conjecture which partitions
$D_n(X;q,t)$ by specifying the points where the Dyck
path touches the diagonal called the ``compositional
shuffle conjecture.''
The refined conjecture states
\begin{equation}
\label{compshuff}
\nabla \left(C_{\alpha}[X;q]\right)=
\sum_{\touch(\pi)=\alpha} \sum_{w \in \mathcal{WP}_\pi} t^{\area(\pi)}q^{\dinv(\pi,w)} x_w.
\end{equation}
Here $\alpha$ is a composition, i.e. a finite list of positive
integers specifying the gaps between the touch points of $\pi$.
The function $C_{\alpha}[X;q]$ is defined below as a
composition of creation operators for Hall-Littlewood polynomials
in the variable $1/q$. They proved that
\[\sum_{|\alpha|=n} C_\alpha[X;q]=(-1)^n e_n[X],\]
implying that \eqref{compshuff} does indeed generalize
\eqref{shuffconj}.
The right hand side of \eqref{compshuff} will be denoted
by $D_{\alpha}(X;q,t)$.
A desirable approach to proving
\eqref{compshuff} would be to determine a recursive formula
for $D_{\alpha}(X;q,t)$, and interpret the result in terms of
some commutation relations for $\nabla$. Indeed, this
approach has been applied in some important special cases, see
\cite{garsia2002catalan,hicks2012sharpening}.
Unfortunately, no such recursion is known in the general
case, and so an even more refined function is needed.

In this paper, we will construct the desired refinement
as an element of a larger vector space $V_k$
of symmetric functions
over $\Q(q,t)$ with $k$ additional variables $y_i$ adjoined,
where $k$ is the length of the composition $\alpha$,
\[N_{\alpha} \in V_k=\Sym[X][y_1,...,y_k].\]
In our first result, (Theorem \ref{thm:recN}),
we will explain how to recover $D_\alpha(X;q,t)$ from $N_\alpha$.
and prove that $N_\alpha$ satisfies
a recursion that completely determines it.
We then define a pair of algebras $\mathbb{A}$
and $\mathbb{A}^*$
which are isomorphic by an antilinear isomorphism
with respect to the conjugation
$(q,t)\rightarrow (q^{-1},t^{-1})$, as well as an explicit action
of each
on the direct sum $V_*=\bigoplus_{k\geq 0} V_k$.
We will then prove that there is an
antilinear involution $N$ on $V_*$
which intertwines the two actions (Theorem \ref{mainthm}),
and represents an involutive automorphism on
a larger algebra $\AA,\AA^*\subset \tilde{\AA}$.
This turns out to be the essential fact that relates the $N_\alpha$
to $\nabla$.

The compositional shuffle conjecture (Theorem \ref{shuffthm}),
then follows as a simple corollary from the following properties:
\begin{enumerate}
\item There is a surjection coming from $\mathbb{A},\mathbb{A}^*$
\[d_-^{k} : V_k \rightarrow V_0=\Sym[X]\]
which maps a monomial $y_\alpha$ in the $y$ variables
to an element $B_{\alpha}[X;q]$ which is similar to
$C_\alpha[X;q]$, and maps $N_\alpha$ to $D_{\alpha}(X;q,t)$, up to a sign.
\item The involution $N$
commutes with $d_-$, and maps $y_\alpha$ to $N_\alpha$. 
\item The restriction of $N$ to $V_0=\Sym[X]$
agrees with $\nabla$ composed with
a conjugation map which essentially exchanges the
$B_{\alpha}[X;q]$ and $C_{\alpha}[X;q]$.
\end{enumerate}
It then becomes clear that these properties imply \eqref{compshuff}.

While the compositional shuffle conjecture
is clearly our main application,
the shuffle conjecture has been further generalized in
several remarkable directions such as the
rational compositional shuffle conjecture, and relationships to
knot invariants, double affine Hecke algebras,
and the cohomology of the affine Springer fibers, see
\cite{bergeron2014rational,
gorsky2014torus,gorsky2015refined,negut2013shuffle,hikita2014affine,
shiff2011elliptic, shiff2013elliptic}.
We hope that future applications to
these fascinating topics will be forthcoming.

\subsection{Acknowledgments}
The authors would like to thank Fran\c{c}ois Bergeron, Adriano Garsia,
Mark Haiman, Jim Haglund, Fernando Rodriguez-Villegas and Guoce Xin
for many valuable discussions on this and related topics.
The authors also acknowledge
the International Center for Theoretical Physics, Trieste, Italy,
at which most of the research for this paper was performed.
Erik Carlsson was also supported by the
Center for Mathematical Sciences and Applications at Harvard University
during some of this period, which he gratefully acknowledges.

\def\Q{\mathbb{Q}}

\section{The Compositional shuffle conjecture}

\subsection{Plethystic operators}
Here we only introduce the basic concepts of $\lambda$-rings. A $\lambda$-ring is a ring $R$ with a family of ring endomorphisms $(p_i)_{i\in \Z_{>0}}$ satisfying
\[p_1[x]=x,\quad p_m[p_n[x]]=p_{mn}[x],\quad (x\in R,\quad m,n\in\Z_{>0}).\]

There are special elements $k\in R$ satisfying $p_i[k]=k$ for all $i \in \Z_{>0}$. Such elements form a subring $K\supseteq \Z$, called the base ring. Here we give a simple proof. Firstly, since $p_i$ are ring endomorphisms, $p_i[0]=0$ and $p_i[1]=1$, so that $0,1\in K$.
Assume $k,k'\in K$. Then i) $0=p_i[k+(-k)]=p_i[k]+p_i[-k]$. This implies $p_i[-k]=-p_i[k]=-k$ and hence $-k\in K$; ii)
$p_i[k+k']=p_i[k]+p_i[k']=k+k'$. This implies that $k+k'\in K$; iii) similarly $p_i[kk']=p_i[k]p_i[k']=kk'$. This implies that $kk'\in K$.
It follows that $K$ is a subring of $R$. It contains $\Z$ because $0,1\in K$.

Elements in $K$ will be simply called numbers.

Similarly the set of elements $x$ satisfying $p_i[x]=x^i$ for all $i$ forms a monoid. Such elements will be called monomials.

For example, $\Q(q,t)$ is a $\lambda$-ring if we use the endomorphism generated by the rules $p_i[k]=k, \ k\in \Q$ and $p_i[q]=q^i, p_i[t]=t^i$.

Unless stated otherwise, the endomorphisms are defined by $p_n(x)=x^n$ for each generator $x$, and every variable in this paper is considered a generator. The ring of symmetric functions over the $\lambda$-ring $\Q(q,t)$
is a free $\lambda$-ring with
generator in $X=x_1+x_2+\cdots$, and will be denoted $\Sym[X]$. To be more precise $p_i$ are ring endomorphisms generated by the rules
$p_i[k]=k$ for $k\in \Q$, $p_i[y]=y^i$ for $y=q,t,x_1,x_2,\dots$.

We will employ the standard notation used for
plethystic substitution defined as follows:
given an element $F\in \Sym[X]$ and $A$ in some $\lambda$-ring $R$,
the plethystic substitution $F[A]$ is the image of the homomorphism
from $\Sym[X]\rightarrow R$ defined by replacing $p_n$ by $p_n(A)$.
This is well-defined since $p_\lambda=p_{\lambda_1}\cdots p_{\lambda_r}$ forms
a basis of $\Sym[X]$ when $\lambda$ ranges over all partitions.
For instance, we would have
\[p_1 p_2[X/(1-q)]=p_1[X] p_2[X](1-q)^{-1}(1-q^2)^{-1}.\]
See \cite{haiman2001polygraphs} for a reference. Note that $X$ is identified with $p_1[X]=x_1+x_2+\cdots$ in plethystic substitution.

We emphasize that normal substitution does not commute with the plethystic substitution. For example,
$$ p_2[qX] \Big|_{q=2}=q^2 p_2[X] \Big|_{q=2}=4 p_2[X],$$
while $p_2[qX \Big|_{q=2}] =p_2[2X]=2p_2[X].$

In order to make normal substitution within plethystic substitution, we need to introduce variables $\epsilon_k$ for $k\in K$, which is treated as a variable inside bracket and $k$ outside of bracket. In this way we will have $p_i[\epsilon_k]=\epsilon_k^i =k^i$.
In particular $\epsilon_1=1$ always holds true, and we may set $\epsilon=\epsilon_{-1}$ for short. (This paragraph is irrelevant to the proof).


The following well-known result will be frequently used.
\begin{lem}\label{l-h1-u}
  If $u$ and $v$ are monomials, then we have
\begin{align}
  h_n[(1-u)v]&=\left\{
                 \begin{array}{ll}
                   1, & \mbox{ if  } n=0; \\
                   (1-u)v^n, & \mbox{ if } n\ge 1.
                 \end{array}
               \right.
\end{align}
\end{lem}
\begin{proof} We have
\begin{align*}
  h_n[(1-u)v] &=\sum_{a=0}^n h_a[v] h_{n-a} [-uv] \\
                &= \sum_{a=0}^n h_a[v](-1)^{n-a} e_{n-a}[uv] \\
(\text{when } n\ge 1)                &=v^n -v^{n-1}\cdot uv =(1-u) v^n.
\end{align*}

\end{proof}

The modified Macdonald polynomials \cite{garsia1998explicit} will be
denoted
\[H_\mu=t^{n(\mu)} J_{\mu}[X/(1-t^{-1});q,t^{-1}] \in\Sym[X]\]
where $J_\mu$ is the integral form of the Macdonald polynomial
\cite{Mac}, and
\[n(\mu)=\sum_{i} (i-1)\mu_i.\]
The operator $\nabla:\Sym[X]\to\Sym[X]$ is defined by
\begin{equation}
\label{nabdef}
\nabla H_\mu = H_\mu[-1] H_\mu = (-1)^{|\mu|} q^{n(\mu')} t^{n(\mu)} H_\mu.
\end{equation}
Note that our definition differs from the usual one from
\cite{Bergeron99identitiesand}
by the sign $(-1)^{|\mu|}$.
We also have the sequences of operators
$B_r, C_r:\Sym[X]\rightarrow\Sym[X]$ given by the following formulas:
$$
(B_r F)[X] = F[X-(q-1)z^{-1}] \pExp[-z X]\big|_{z^r},
$$
$$
(C_r F)[X] = -q^{1-r} F[X+(q^{-1}-1)z^{-1}] \pExp[z X]\big|_{z^r},
$$
where $\pExp[X]=\sum_{n=0}^\infty h_n[X]$ is the plethystic exponential and $\vert_{z^r}$ denotes the operation of taking the coefficient of $z^r$ of a Laurent power series. Our definition again
differs from the one in \cite{haglund2012compositional} by a factor $(-1)^r$.
For any composition $\alpha$ of length $l$, let $C_\alpha$ denote the composition $C_{\alpha_1}\cdots
C_{\alpha_l}$, and similarly for $B_\alpha$.

Finally we denote by $x \mapsto \bar{x}$ the involutive
automorphism of $\Q(q,t)$ obtained by sending $q,t$ to $q^{-1},t^{-1}$. We denote by $\omega$ the $\lambda$-ring automorphism of $\Sym[X]$ obtained by sending $X$ to $-X$ and by $\bar\omega$ its composition with $\bar{*}$, i.e.
\[
(\omega F)[X] = F[-X],\quad (\bar\omega F)[X] = \bar F[-X].
\]

\newcommand\ce[1]{\fbox{$#1$}}
\subsection{Parking functions}
We now recall the combinatorial background to state the Shuffle
conjecture, for which we refer to Haglund's book \cite{haglund2008catalan}.
We consider the infinite grid on the top right quadrant of the plane. Its intersection points are denoted as $(i,j)$ with $i,j\in\Z$.
The $(i,j)$-cell, denoted $\ce{(i,j)}$, is the cell of the grid whose top right corner has coordinate $(i,j)$. Thus $i=1,2,\ldots$ indexes the columns and $j=1,2,\ldots$ indexes the rows.
Let $\DD$ be the set of Dyck paths of all lengths. A Dyck path of length $n$ is a grid path from $(0,0)$ to $(n,n)$ consisting of North and East steps that stays weakly
above the main diagonal $i=j$. For $\pi\in\DD$ we denote by $|\pi|$ its
length $n$ and let
\[\area(\pi):=\#\Area(\pi),\quad
\Area(\pi):=\left\{(i,j): i<j,\ \ce{(i,j)} \mbox{ is under } \pi\right\}.\]
This is the set of cells between the path and the diagonal.
Let $a_j$ denote the number of cells $(i,j) \in\Area(\pi)$ in row $j$. The \emph{area sequence} is the sequence $a(\pi)=(a_1,a_2,\ldots,a_n)$ and we have $\area(\pi)=\sum_{j=1}^n a_n$.

Let $\ce{(x_1,1)}, \ce{(x_2,2)},\ldots, \ce{(x_n, n)}$ be the cells immediately to the right of the North steps of $\pi$. (These are the places we will put cars $w_i$ for parking functions.)

The sequence $x(\pi)=(x_1,x_2,\ldots,x_n)$ is called the \emph{coarea sequence}\footnote{More precisely $(x_1-1,x_2-1,\ldots,x_n-1)$ is the coarea sequence, so that coarea and area adds to max area $n(n-1)/2$.} and we have $a_j+x_j=j$ for all $j$.

We have the $\dinv$ statistic and the $\Dinv$ set defined by
\[\dinv(\pi):=\#\Dinv(\pi),\quad \Dinv(\pi):= \Dinv^0(\pi) \cup \Dinv^1(\pi),\]
where $\Dinv^0(\pi)$ is the primary dinv and $\Dinv^1(\pi)$ is the secondary dinv given by
\begin{align*}
\Dinv^0(\pi) &=\left\{(j,j'):1\leq j<j'\leq n,\ a_j=a_{j'}\right\},\\
\Dinv^1(\pi)&=\left\{(j,j'):1\leq j'<j\leq n,\ a_{j'}=a_j+1\right\}.
\end{align*}
For $(j,j')\in\Dinv(\pi)$ we say that $\ce{(x_j, j)}$ \emph{attacks} $\ce{(x_{j'}, j')}$.


For any $\pi$, the set of \emph{word parking functions}
associated to $\pi$ is defined by
\[\mathcal{WP}_\pi:=\left\{{w} \in \Z_{>0}^n :
w_{j}>w_{j+1} \;\mbox{whenever}\; x_{j}(\pi)=x_{j+1}(\pi)\right\}.\]
In other words, the elements of $\mathcal{WP}_\pi$
are $n$-tuples $w$ of positive integers which, when
written from bottom to top to the right of each North step,
are strictly decreasing on cells such that one is on top
of the other. For any $w$, let
\[\dinv(\pi,w):=\#\Dinv(\pi,w),\quad
\Dinv(\pi,w):=\left\{(j,j') \in \Dinv(\pi):w_j>w_{j'}\right\}.\]

We note that both of these conditions differ from the usual notation
in which parking functions are expected to increase rather than
decrease, and in which the inequalities are reversed in the definition
of $\dinv$. This corresponds to choosing the opposite total ordering on
$\Z_{>0}$ everywhere, which does not affect the final answer,
and is more convenient for the purposes of this paper.

Let us call $\alpha=(\alpha_1,...,\alpha_k)=\touch(\pi)$
the \emph{touch composition} of $\pi$ if $\alpha_1,...,\alpha_k$ are the
lengths of the gaps between the points where $\pi$ touches the
main diagonal starting at the lower left. Equivalently, $\sum_{i=1}^k \alpha_i=n$ and the numbers $1$, $1+\alpha_1$, $1+\alpha_1+\alpha_2$, \ldots $1+\alpha_1+\cdots+\alpha_{k-1}$ are the positions of the $0$'s in the area sequence $a(\pi)$.

\begin{example}
\label{dyckex}
Let $\pi$ be the Dyck path of length $8$ described in Figure \ref{fig:PF-concepts1}. 
Then we have
\[\Area(\pi)=\left\{
(2,3),(2,4),(3,4),(3,5),(3,6),(4,5),(4,6),(5,6),(7,8)\right\},\]
\[\Dinv(\pi)=\left\{(1,2),(1,7),(2,7),(3,8),(4,5)\right\} \cup
\left\{(7,3),(8,4),(8,5)\right\},\]
\[\touch(\pi)=(1,5,2),\quad a(\pi)=(0,0,1,2,2,3,0,1),\]
\[x(\pi)=(1,2,2,2,3,3,7,7)\]
whence $\area(\pi)=9$, $\dinv(\pi)=5+3=8$.
The labels shown above correspond to the vector
$w=(9,5,2,1,5,2,3,2)$, which we can see is an element of
$\mathcal{WP}_\pi$ because we have $5>2>1$, $5>2$, $3>2$.
We then have
\[\Dinv(\pi,w)=\left\{(1,2),(1,7),(2,7)\right\}\cup
\left\{(7,3),(8,4)\right\},\]
giving $\dinv(\pi,w)=5$.

\begin{figure}[ht]
$$\includegraphics[width=8cm]{
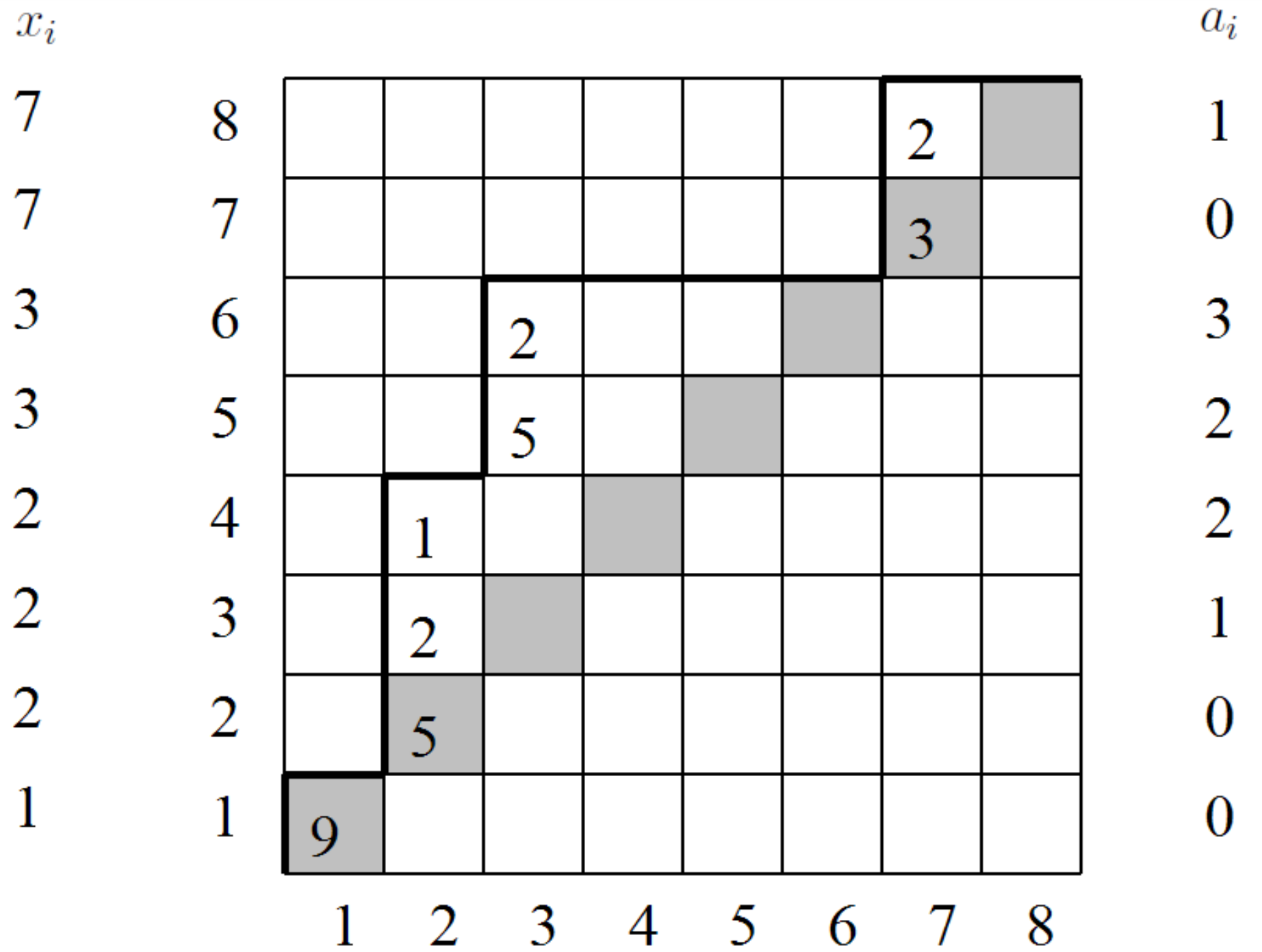} $$
\caption{A parking function of size 8.}
\label{fig:PF-concepts1}
\end{figure}
\end{example}

\subsection{The shuffle conjectures}
Let $X=\{x_1,x_2,...\}$ be any infinite set of variables. Denote by $x_w=x_{w_1}\cdots x_{w_n}$.
In this notation, the original shuffle conjecture \cite{haglund2005diagcoinv}
states
\begin{conjecture}[\cite{haglund2005diagcoinv}] We have
\[
(-1)^n \nabla e_n=
\sum_{|\pi|=n} t^{\area(\pi)}\sum_{w\in \mathcal{WP}_\pi}
q^{\dinv(\pi,w)} x_{w}.
\]
In particular, the right hand side is symmetric in the $x_i$,
and in $q,t$.
\end{conjecture}
The stronger compositional shuffle conjecture \cite{haglund2012compositional}
states
\begin{conjecture}[\cite{haglund2012compositional}]
\label{conj-refined-shuffle} For any composition $\alpha$,
we have
\begin{equation} \label{compshuffeq}
(-1)^n \nabla C_{\alpha}(1)=
\sum_{\touch(\pi)=\alpha} t^{\area(\pi)}\sum_{w\in \mathcal{WP}_\pi}
q^{\dinv(\pi,w)} x_w.
\end{equation}
\end{conjecture}

The inner sum of \eqref{compshuffeq} is a symmetric function in the $x$'s for each Dyck path $\pi$. This symmetry is known as the consequence of a
general result of the symmetry of the well-known LLT-polynomials: it is an LLT product of vertical strips. See Chapter 6 (especially Remark~6.5) of \cite{haglund2008catalan} for more information on LLT polynomials. If we release the sum to ranges over all words, that is, let
\begin{align}\label{e-bar-chi}
  \bar\chi(\pi)= \sum_{w\in \Z_{>0}^n}
q^{\dinv(\pi,w)} x_w,
\end{align}
then $\bar \chi(\pi)$ can be seen to be an LLT product of single cells (called unicellular LLT-polynomials) and is hence symmetric.
We will see under the $\zeta$ map, $\bar\chi(\pi)$ is naturally transformed to $\chi(\pi')$.

\subsection{The $\zeta$ Map: From $(\area,\dinv)$ to $(\bounce,\area')$}
Our next task is to prove an equivalent version of Conjecture \ref{conj-refined-shuffle}
by using the bijection $\zeta$ in \cite{haglund2008catalan} which is known to send
the statistics $(\area,\dinv)$ to another statistics $(\bounce,\area')$.

We shall use a different description of the $\zeta$ map that comes naturally from analysis of the attack relations.
An important property of this construction is that it has a natural lift from Dyck paths to parking functions.
Moreover, many known results are easier to prove under this model.

In order to define the map $\zeta$ from $\DD_n$ to itself we need the some concepts.
Given a Dyck path $\pi \in \DD_n$, we put
the (diagonal) \emph{reading
order} labels $\sigma_j$ into the cells $\ce{(x_j, j)}$ (where we put cars for parking functions) such that:
the labels are increasing when read them by diagonals from bottom to top, and from left to right in each diagonal.
The labels will always be $1,2,\dots, n$ (if not specified) and $\sigma=\sigma_1\cdots \sigma_n \in S_n$ is called the reading order permutation of $\pi$.
To be more precise, $\sigma_i<\sigma_j$ if $a_i<a_j$ or $a_i=a_j$ and $i<j$.
For instance, for the path $\pi$ in the left picture of Figure \ref{fig:PF-concepts2}, we have put reading order labels in the cells $\ce{(x_j,j)}$.
Reading these labels from bottom to top gives the reading order permutation $\sigma=1   2   4   6   7  8   3   5 $. This is the unique permutation satisfying $\dinv(\pi,\sigma)=0$.

The map $\zeta: \pi \to \pi'$ is defined by the formula:
\[\Area(\pi')=\sigma(\Dinv(\pi))=
\left\{(\sigma_i,\sigma_{j}):(i,j)\in \Dinv(\pi)\right\}.\]
In words we can construct $\pi'$ as follows. We will simply say the reading label $i$ attacks the reading label $j$ if the corresponding cells attacks each other. Observe that for each $j=1,\ldots,n$ the cell $(x_j,j)$ with label $\sigma_j$ attacks all the subsequent cells (with labels $\sigma_j+1,\sigma_j+2,\dots$) that are either in the same diagonal and to the right of it or in the next diagonal and to the left of the cell $(x_j, j+1)$ on top of it. It follows that
if reading label $i$ attacks $j$ then it also holds that $i+1$ attacks $j$ and $i$ attacks $j-1$. This implies that these attack relations forms the area cells of a unique Dyck path $\pi'$. Figure
\ref{fig:PF-concepts2} illustrates an example of the $\zeta$ map.
\begin{figure}
$$\includegraphics[height=6cm]{
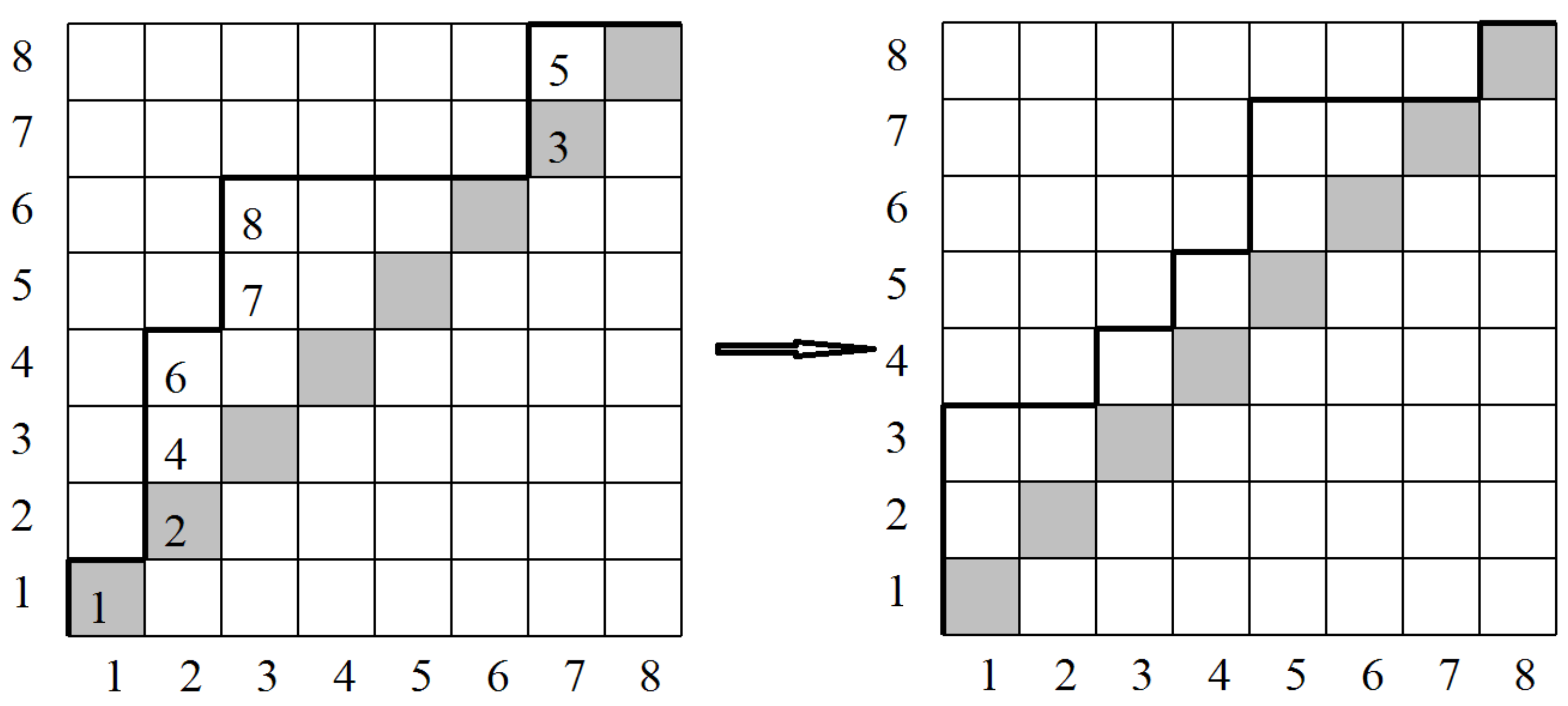} $$
\caption{The left picture illustrate the diagonal reading order of $\pi$. The attack relations of $\pi$ are: 1 attacks 2 and 3; 2 attacks 3;
 3 attacks $4$; 4 attacks 5; 5 attacks 6 and 7;  6 attacks 7. They become the cells in $\Area(\pi')$.}
\label{fig:PF-concepts2}
\end{figure}

On the left picture, we have labelled the cells $\ce{(x_j,j)}$ using their reading order, from $1$ to $n$.
The reading label $1$ attacks $2$ and $3$, but not $4$ or later; the reading label $2$ attacks $3$ but not $4$; and so on. In summary, the attack relations are given by
$$\{(1,2), (1,3), (2,3), (3,4), (4,5), (5,6), (5,7), (6,7) .\}   $$
This is $\Area(\pi')$ for the Dyck path $\pi'$ in the right picture.

It is clear from the construction that
$\dinv(\pi)=\area(\pi')$.
We will explain in the next subsection that: i) the bijectivity of $\zeta$, ii)
the property $\area(\pi)=\bounce(\pi')$, and iii) define $\touch'(\pi')=\touch(\pi)$ directly using $\pi'$.

With the help of the reading order permutation $\sigma$, the $\zeta$ map naturally extends for word parking functions:
From any pair $(\pi,w)$ with $\pi\in\DD$, $w\in\mathcal{WP}_\pi$, define $\zeta(\pi,w)=(\pi',w')$ by
$\pi'=\zeta(\pi)\in \DD$ and $w'=\sigma(w)=w_{\sigma_1}\cdots w_{\sigma_n} \in \mathcal{WP'}_{\pi'}$,
where $\mathcal{WP'}_{\pi'}$ will be defined after analyzing some statistics.

The translation of the $\dinv$ statistic is straightforward.
For any $w'\in \Z_{>0}^n$, let
\[\inv(\pi',w'):=\#\Inv(\pi',w'),\quad
\Inv(\pi',w'):=\left\{(i,j) \in \Area(\pi'),\ w'_i>w'_j\right\},\]
so that
\[\Inv(\pi',w')=\sigma\left(\Dinv(\pi,w)\right),\quad
w'_{\sigma_i}=w_{i}.\]

\begin{figure}[ht]
$$\includegraphics[width=12cm]{
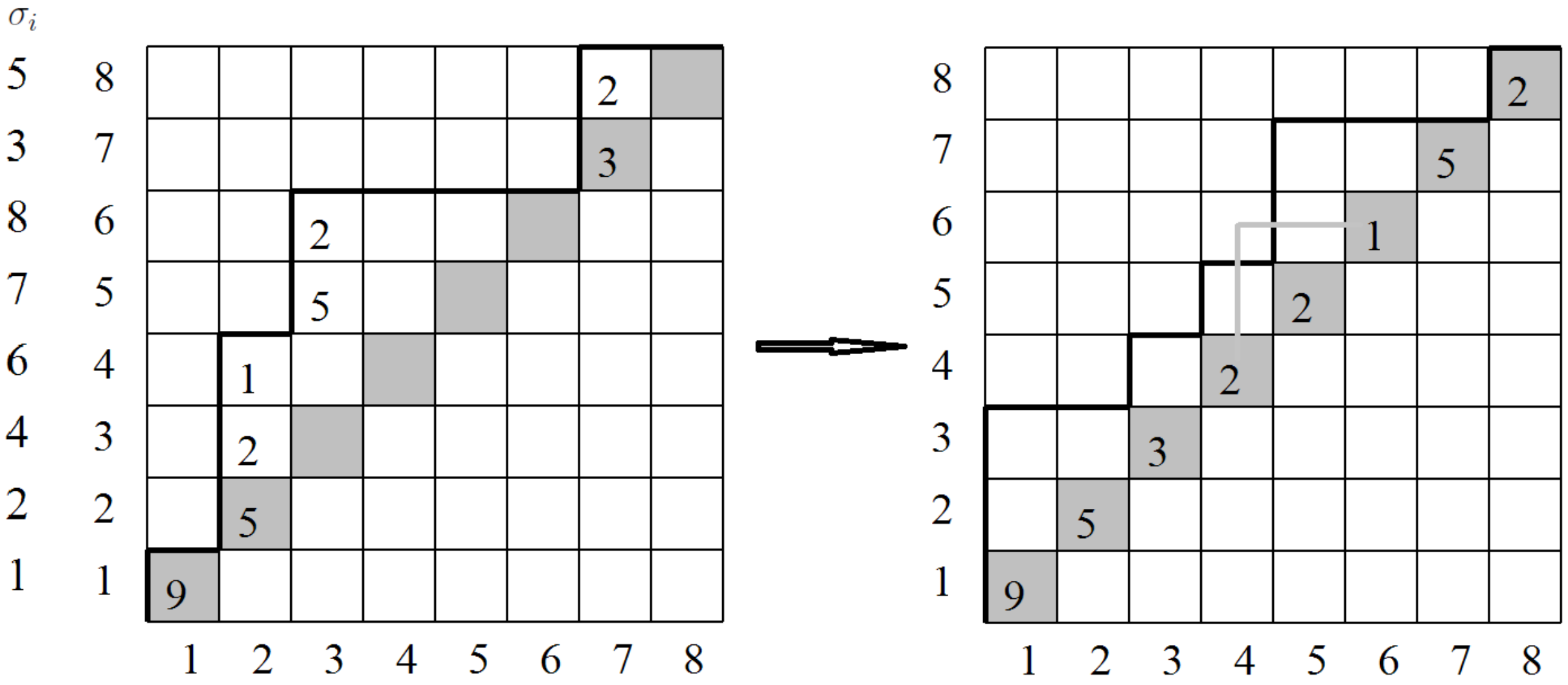} $$
\caption{An example of the $\zeta$ map on parking function.}
\label{fig:zetaw}
\end{figure}

Figure \ref{fig:zetaw} illustrates the extension of the $\zeta$ map for parking functions. Now in the left picture we have put ``cars" $w_i$ in and moved
the reading order permutation to the left. For the right picture we have put $w_i'$ into the diagonal cell $(i,i)$. The geometric meaning for a pair of cars to contribute to $\Inv(\pi',w')$ is: the cell in the column of the lower car and in the row of the higher car is under $\pi'$, and the lower car is greater than the higher car.
For instance the $4$-th car and the $6$-th car will not contribute even if they form an inversion. Indeed we have
\[\Inv(\pi',w')=\{(1,2),(1,3),(2,3),(3,4),(5,6)\}.\]
These pairs of cars form an inversion and they determine a cell under $\pi'$.
In particular, $\inv(\pi',w')=\dinv(\pi,w)=5$.

Finally, we reconstruct the word parking function condition.
A cell $(i,j)$ is called a {\em corner} of $\pi'$
if it is above the path, but both its Southern and Eastern neighbors
are below the path. Denote the set of corners by $c(\pi')$.
For instance, from example in Figure \ref{fig:zetaw},  we have
$c(\pi')=\{(2,4),(3,5),(4,6),(7,8)\}$. We conclude:
\begin{prop}
  We have
\[
c(\pi') = \{(\sigma_r, \sigma_{r+1}):\;1\leq r<n,\;x_r(\pi)=x_{r+1}(\pi)\}.
\]
In words, a corner $(i,j)$ in $\pi'$ corresponds to the cell with reading label $j$ lying on top of the cell
with reading label $i$ in $\pi$.
\end{prop}
\begin{proof}
On one hand, if reading label $j$ is on top of reading label $i$, then clearly $i$ attacks $j-1$,
$i+1$ attack $j$, and $i$ does not attack $j$. This shows that the cell $(i,j)$ is a corner of $\pi'$.

On the other hand, if $(i,j)$ is a corner of $\pi'$ then reading label $i+1$ (in $\pi$) attacks $j$,
$i$ attacks $j-1$ but not $j$. The latter condition forces label $j$ being in the cell on top of $i$ if that cell is under $\pi$.
So assume label $i$ is on diagonal $a$\footnote{A cell $(r,s)$ is said to be on diagonal $s-r$.} and the cell on top of $i$ is above $\pi$. Case 1: $i+1$ is to the right of $i$ and
on the same diagonal $a$. Observe that all labels to the right of $i$ and to the left of $i+1$ must lie below diagonal $a$.
Then there can be no $j$ attacked by $i+1$ and not attacked by $i$. A contradiction. Case 2: $i+1$ is the leftmost label on diagonal $a+1$ and it is to the left of $i$. Observe that all labels to the right of $i$ must lie below diagonal $a$, and that all labels to the left of $i+1$ lie weakly below diagonal $a$. Then there can be no $j$ attacked by $i+1$ and not attacked by $i$. A contradiction. This proves the proposition.
\end{proof}

We therefore define
\begin{equation}
\label{WP1}
\mathcal{WP}'_{\pi'}:=\left\{w' \in \Z_{>0}^n :
w'_i>w'_j \mbox{ for } (i,j) \in c(\pi')\right\},
\end{equation}
so that the condition $w\in \mathcal{WP}_\pi$ is equivalent to $w'\in \mathcal{WP}'_{\pi'}$.

Putting these together, we have
\begin{prop}
For any composition $\alpha$ we have
\begin{equation}\label{eq:Dalpha}
D_{\alpha}(q,t)=
\sum_{\touch'(\pi)=\alpha} t^{\bounce(\pi)}\sum_{w \in \mathcal{WP}'_\pi}
q^{\inv(\pi,w)}
\end{equation}
where $D_\alpha(q,t)$ is the right hand side of \eqref{compshuffeq}.
\end{prop}

\begin{rem}
  Readers familiar with the sweep map may find this bijection is different from the sweep map: here the sweep map is to sort the steps of $\pi$ according to their rank of starting points. The sweep order is from bottom to top, and on each diagonal, from right to left. Indeed this bijection is a variation of the sweep: we need first transpose $\pi$ and then apply the sweep map. See Figure \ref{fig:SweepMap}.
\begin{figure}[ht]
$$\includegraphics[height=4cm]{
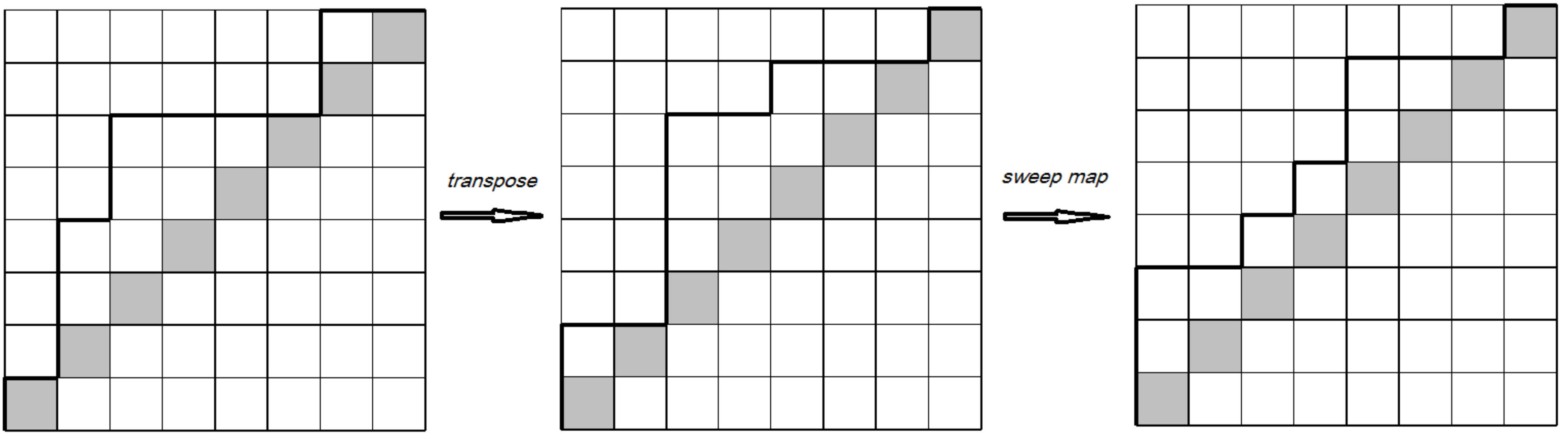} $$
\caption{An alternative way for the bijection.}
\label{fig:SweepMap}
\end{figure}
\end{rem}

\subsection{The Bounce Statistic and the Bijectivity of the $\zeta$ Map}

For any path $\pi'$, we obtain a new Dyck path called the ``bounce path'' as
follows: start at the origin $(0,0)$, and begin moving North until
contact is made with the first East step of $\pi$. Then start moving East until
contacting the diagonal. Then move North until contacting the path again,
and so on. Note that contacting the path means running into the left
endpoint of an East step, but passing by the rightmost endpoint does
not count, as illustrated in the right picture of Figure \ref{fig:bounce}. The bounce path splits the main diagonal into
the \emph{bounce blocks}. We number the bounce blocks starting from $0$ and define the
\emph{bounce sequence} $b(\pi)=(b_1, b_2, \ldots, b_n)$ in such a way that for any $i$ the cell $(i,i)$ belongs to the $b_i$-th block.
We then define
\[\bounce(\pi'):=\sum_{i=1}^n b_i.\]
In our running example we have
We have
\[b(\pi')=(0,0,0,1,1,2,2,3),\quad \bounce(\pi')=9=\area(\pi).\]
Indeed, we have
\begin{figure}[ht]
$$\includegraphics[height=5cm]{
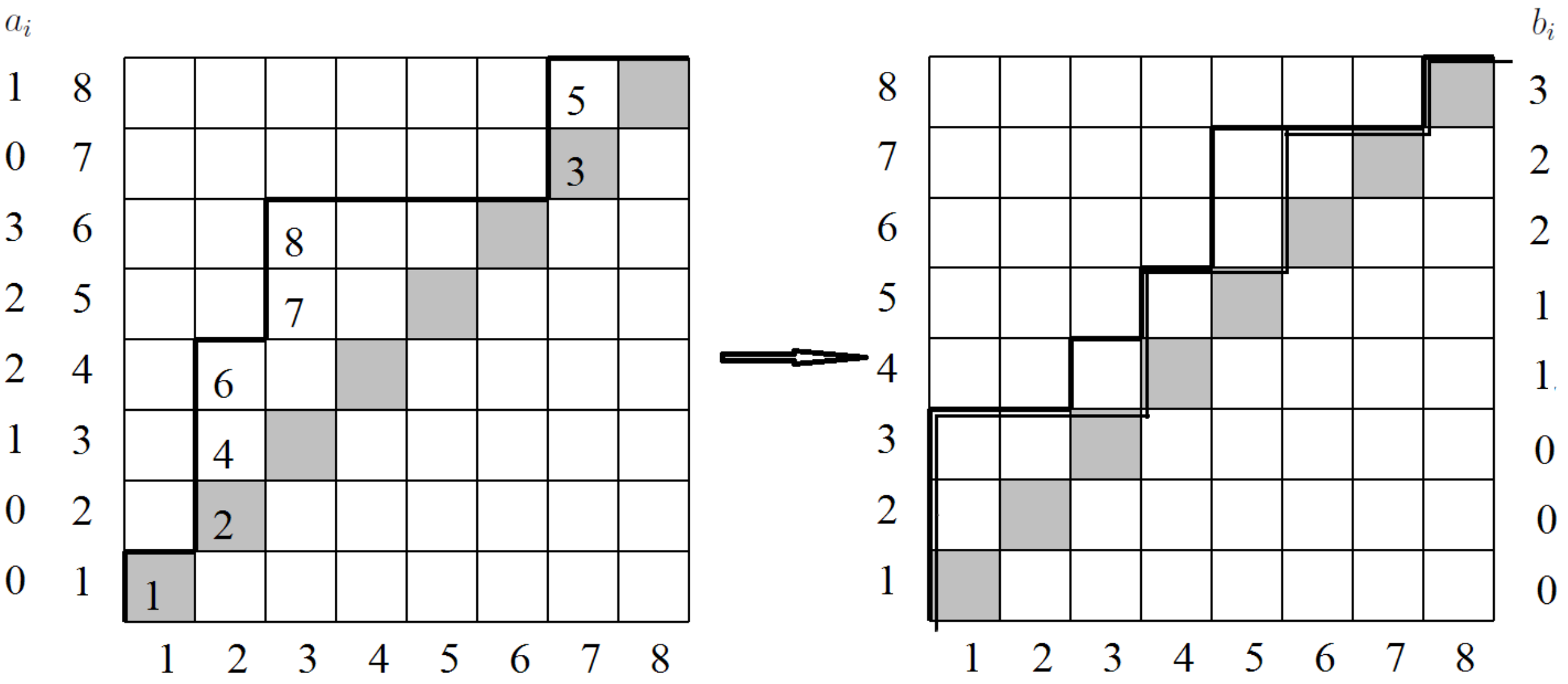} $$
\caption{The area sequence on the left picture maps to the bounce sequence on the right picture.}
\label{fig:bounce}
\end{figure}

\begin{prop}\cite{haglund2008catalan}
Suppose $\zeta(\pi)=\pi'$, $\sigma$ is the reading permutation of $\pi$, $a_i$ is the area sequence of $\pi$ and $b_i$ is the bounce sequence of $\pi'$. Then $b_{\sigma_i} = a_i$ for all $i$.
\end{prop}
\begin{proof}
Assume the area sequence $(a_1,\dots, a_n)$ of $\pi$ consists of $\beta_1$ 0's, $\beta_2$ 1's, and so on.
Then $f_j=1+\beta_1+\cdots +\beta_{j}$ is the left most reading label on diagonal $j$ for $j\ge 0$. Observe
that $f_{j+1}$ is weakly to the right of $f_{j}$ for all $j\ge 0$. This is because the Dyck path starts at $(0,0)$,
and consecutively reaches higher diagonal. It follows that
$1=f_0$ attacks $2,\dots, f_{1}-1$, and hence in $\pi'$ we reaches the first peak at $(0,f_{1}-1)$ and bounce to the diagonal point
$(f_1-1, f_1-1)$. Similarly $f_{1}$ attacks $f_1+1,\dots, f_2-1$ so that $(f_1-1,f_2-1)$ is a peak of $\pi'$ and bounce to the
diagonal point $(f_2-1,f_2-1)$. Continue this way we see that the bounce blocks are of sizes $(f_1-1,f_2-f_1,\dots)=(\beta_1,\beta_2,\dots)$.
This completes the proof.
\end{proof}

\begin{rem}
From the proof we see that $\pi'$ pass through the peaks of its bounce path, say $p_1,p_2,\dots$. It is easy to see that the steps
from $p_i$ to $p_{i+1}$ is uniquely determined by the (natural path) order of the reading labels on diagonal $i-1$ and $i$, which is given by
the area sequence $(a_1,\dots, a_n)$ restricted to $i-1$ and $i$. This leads to the original definition of $\zeta$ in \cite[Page ~50]{haglund2008catalan}: Place a pen at
the second peak $(\beta_1,\beta_1+\beta_2)$. Start at the end of the area sequence and travel left. Trace a south step with your pen
 whenever encounter a $1$ and trace a west step whenever encounter a $0$. Similarly draw the path from $p_i$ to $p_{i-1}$ for all $i$. It is not hard to see that the resulting path is the same as our $\pi'$.
\end{rem}

Next we introduce a sequence of injections $\psi_r: \DD_{n-1} \to \DD_{n}$ and their left inverse $\bar\psi: \DD_n \to \DD_{n-1}$ for $n\ge 0$ as follows.
Let $\alpha=(\alpha_1,\dots, \alpha_l)$ be a composition. We denote by $\DD_\alpha =\{\pi \in \DD_n: \touch(\pi)=\alpha\}.$
Suppose $\pi \in \DD_\alpha$. i) For any $r$ with $0\le r \le l$, split $\pi=\pi_2 \pi_1$ at its $r$-th touch point and define
$\psi_r(\pi)=N\pi_1 E \pi_2$. Note that we count $(0,0)$ as the $0$-th touch point, so that $|\pi_2|=\alpha_1+\cdots+\alpha_{r}$;
ii) For the same $\pi$ we split it as $N\pi_1E\pi_2$ so that $\pi_2$ starts at its first touch point $(\alpha_1,\alpha_1)$, and define $\bar\psi \pi=\pi_2\pi_1$. Note that $\pi_1$ and $\pi_2$ are both (possibly empty) Dyck paths.
See Figure \ref{fig:psi} for an example.
In Figure \ref{fig:psi}, we have put the reading order labels for both $\pi$ and $\bar \psi(\pi)$, and we have decreased the reading labels by $1$ in the left picture for comparison.
\begin{figure}[ht]
$$\includegraphics[height=5cm]{
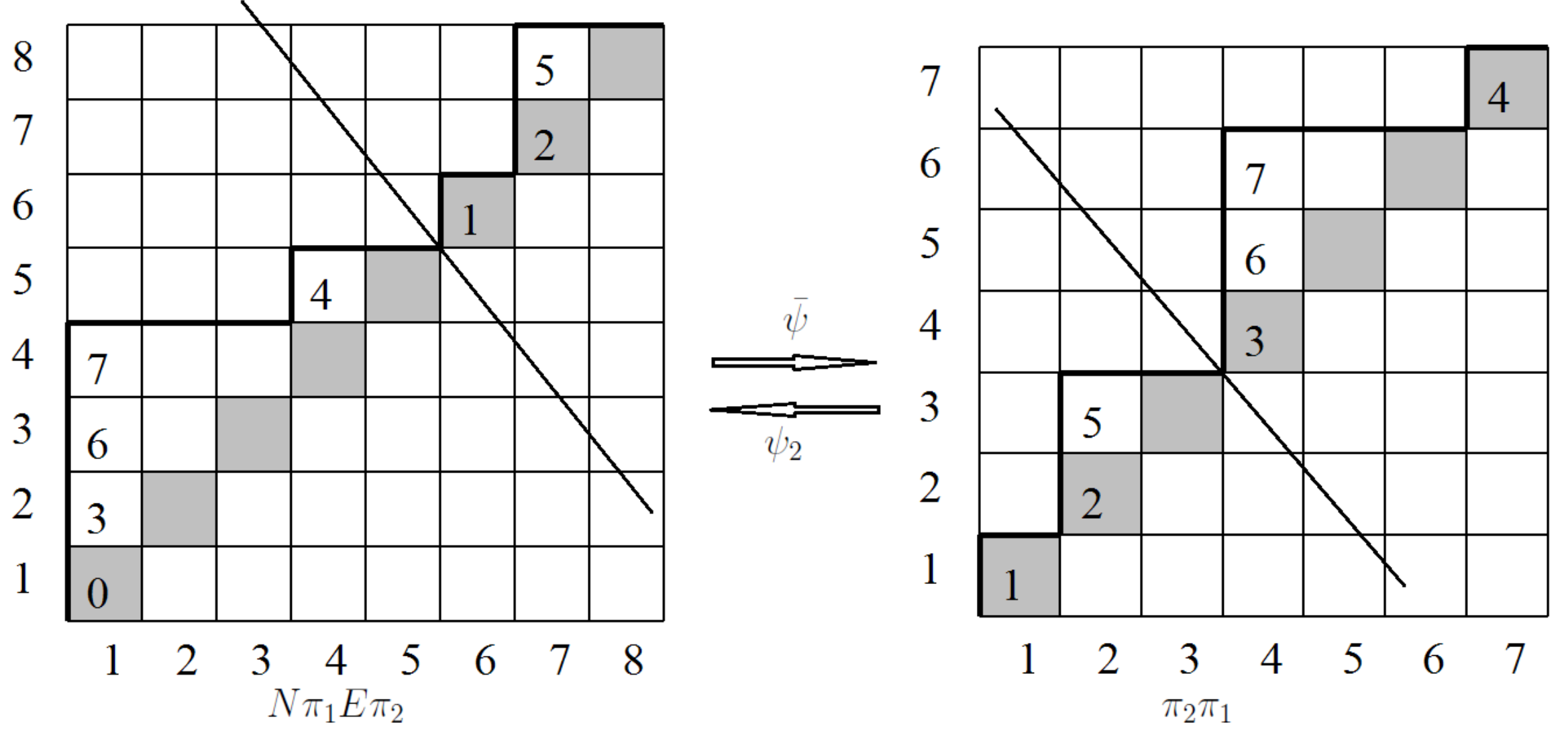} $$
\caption{A Dyck path $\pi$ and its image $\bar \psi(\pi)$.}
\label{fig:psi}
\end{figure}

From this example, the following properties of $\bar \psi$ are obvious.
\begin{prop} \label{p-psi}
Suppose $\alpha=(\alpha_1,\dots, \alpha_l)$ is a composition of $n\ge 1$. Then we have the following properties.
\begin{enumerate}
  \item If $\pi \in \DD_\alpha$ then $\touch(\bar \psi(\pi))=(\alpha_2,\dots, \alpha_l, \beta)$ for some composition
$\beta$ of $\alpha_1-1$.

\item The restriction of $\bar \psi$ on $\DD_\alpha$ is a bijection from $\DD_\alpha$ to the disjoint union $\bigsqcup_{\beta \models \alpha_1-1} \DD_{\alpha_2,\dots, \alpha_l,\beta}$. Its left inverse map is $\psi_{l-1}$.

\item $\area(\bar \psi(\pi))=\area(\pi)-\alpha_1+1$.

\item The attack relations of $\pi$ is the union of the attack relations of $\psi(\pi)$ and $\{(0,1),(0,2),\dots,(0,l-1)\}$.
\end{enumerate}
\end{prop}

We also have the following properties for $\psi_r$.
\begin{prop}\label{p-psi-r}
Suppose $\alpha=(\alpha_1,\dots, \alpha_l)$ is a composition of $n\ge 0$. Then we have the following properties.
\begin{enumerate}
\item If $\pi \in \DD_\alpha$ then $\touch(\psi_r(\pi))=(1+\alpha_{r+1}+\cdots+\alpha_l, \alpha_1,\alpha_2,\dots, \alpha_r)$.

\item $\area(\psi_r(\pi))=\area(\pi)+\alpha_{r+1}+\cdots+\alpha_l$ with $\area(\psi_l(\pi))=\area(\pi)$ and $\area(\psi_0(\pi))=\area(\pi)+|\alpha|.$

\item The attack relations of $\psi_r(\pi)$ is the union of the attack relations of $\pi$ and $\{(0,1),(0,2),\dots,(0,r)\}$.
\end{enumerate}
\end{prop}

Property (iii) implies that the $\zeta$ map images of $\pi$ and that of $\psi_r(\pi)$ are very similar since the $\zeta$ map sends attack relations to area cells. See Figure \ref{fig:psizeta}.
\begin{figure}[ht]
$$\includegraphics[height=5cm]{
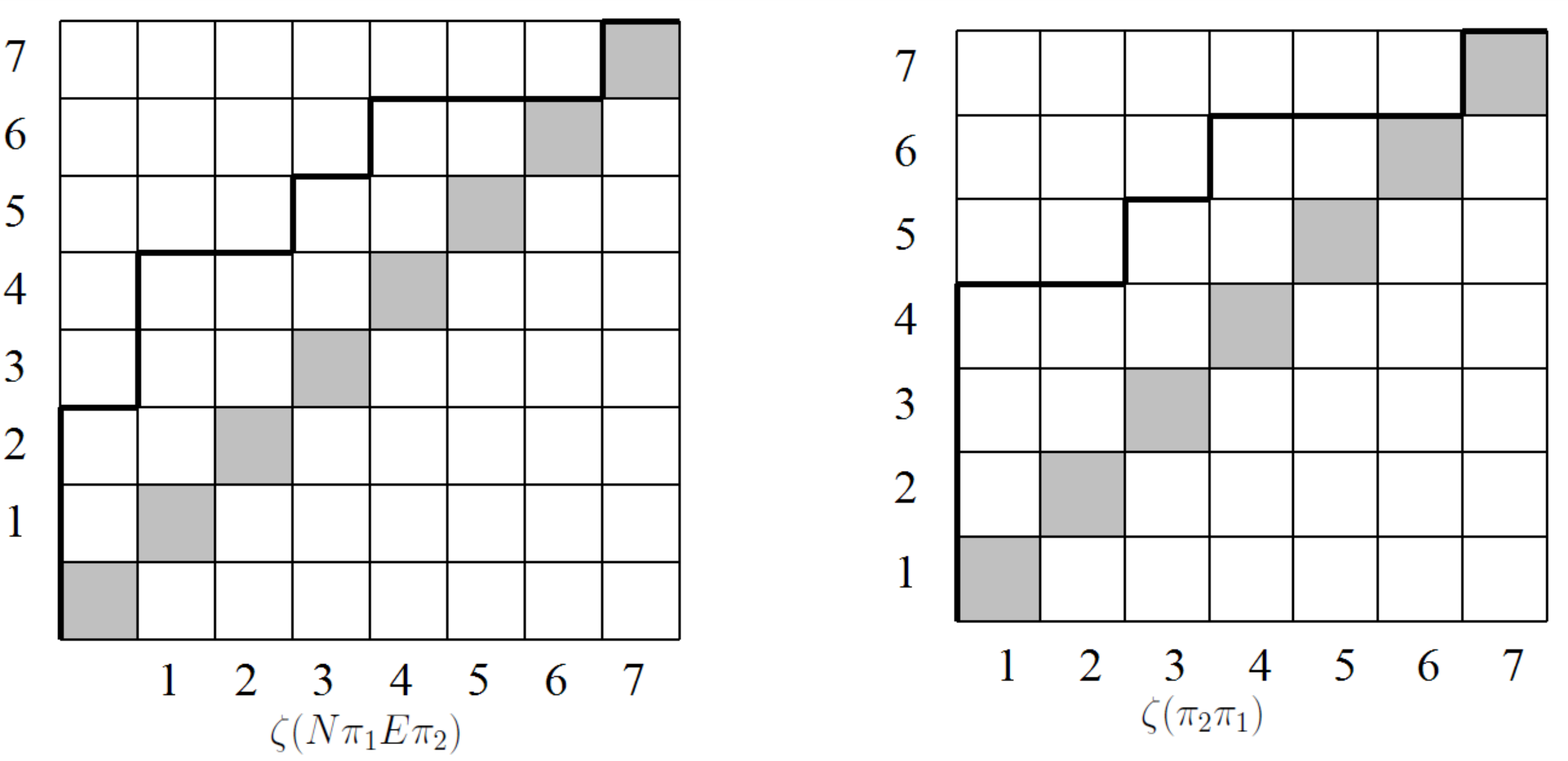} $$
\caption{The $\zeta$ map images of the two Dyck paths in Figure \ref{fig:psi}.}
\label{fig:psizeta}
\end{figure}

Let $\pi' \in \DD_n$ with $n\ge 1$. Then we can uniquely write
$\pi'=N^lE\tilde{\pi}'$, i.e., $\pi'$ starts with $l$ North steps
from the origin, followed by an East step, followed by
the steps of $\tilde{\pi}'$.
Define $\gamma_r (\pi')=N^{r+1} E N^{l-r} \tilde{\pi}'$. This corresponds to adding area cells $(0,1),\dots, (0,r)$.
Clearly $|\gamma_{ r}(\pi')| = |\pi'|+1$.
Define $\bar\gamma(\pi')=N^{l-1}\tilde{\pi}'$. This corresponds to removing the area cells in the first column.

\begin{cor}\label{c-psi-gamma}
If $\pi'=\zeta(\pi)$ can be written as $N^l E \tilde{\pi}'$, then we have
\begin{align}
  \zeta\circ \psi_r (\pi) &= \gamma_{r} (\pi')=\gamma_r \circ \zeta (\pi), \qquad 0\le r \le l; \label{e-psi-gamma}\\
\zeta\circ \bar \psi (\pi) &= \bar \gamma (\pi')=\bar \gamma \circ \zeta (\pi). \label{e-psibar-gamma}
\end{align}
\end{cor}
\begin{proof}
These are direct consequence of Propositions \ref{p-psi} and \ref{p-psi-r} and the fact that the $\zeta$ map takes attack relations to area cells.
\end{proof}

\begin{thm}
  The $\zeta$ map is invertible.
\end{thm}
\begin{proof}
We need only prove the injectivity of the zeta map $\zeta$.
Given $\pi'\in \DD_n$ as in Corollary \eqref{c-psi-gamma} with a pre-image $\pi$. Then $l$ is known in advance.
By Equation \eqref{e-psibar-gamma} and Proposition \ref{p-psi}(ii), we must have
$$ \pi=\psi_{l-1}\circ \zeta^{-1} \circ \bar \gamma (\pi'),$$
where $\zeta^{-1}$ is well-defined by induction on the length of $\pi$. Indeed, this formula gives a recursive construction of $\pi$:
We first construct $\bar \gamma(\pi')$ which is in $\DD_{n-1}$. Thus by induction we can recursively construct $\zeta^{-1}\circ \bar \gamma(\pi')$. Finally applying $\psi_{l-1}$ gives $\pi$, which has $l$ touch points.
\end{proof}

Define numbers $t_i$ by
\[t_r:=\bounce \gamma_r(\pi')=\bounce \left(N^{r+1}EN^{l-r}E \tilde{\pi}\right),\quad
0\leq r \leq l.\]
Set $\touch'(\pi')=(t_0-t_1,t_1-t_2,\dots, t_{l-1}-t_l)$.
\begin{prop}
For every Dyck path $\pi$ and $\pi'=\zeta(\pi)$, we have
\[\touch'(\pi')=\touch(\pi).\]
\end{prop}
\begin{proof}
By Proposition \ref{p-psi-r} and Corollary \ref{c-psi-gamma} and the fact $\bounce(\zeta(\pi))=\area(\pi)$.
\end{proof}

Now we summarize as follows in a compact way:
\begin{align*}
  \psi_r &= \zeta^{-1} \circ \gamma_r \circ \zeta, \qquad \bar \psi =\zeta^{-1} \circ \bar \gamma \circ \zeta\\
  \bar\gamma \circ \gamma_r &=id, \qquad \gamma_r \circ \gamma (\pi) = \pi \quad \text{if } l(\touch'(\pi))=r+1 \\
  \bar\psi \circ \psi_r &=id, \qquad \psi_r \circ \psi (\pi) = \pi \quad \text{if } l(\touch(\pi))=r+1 \\
\end{align*}

Denote by $\DD'_{\alpha}=\{\pi': \touch'(\pi')=\alpha\}=\zeta(\DD_\alpha)$.
\begin{prop}
\label{p-D'-rec}
Let $\alpha$ be a composition of length $l$ and $a$ be a positive integer. Then we have the following recursion.
\begin{align}
  \label{e-D'-rec}
\DD'_{a,\alpha}=  \bigsqcup_{\beta \models a-1} \gamma_l(\DD'_{\alpha\beta}).
\end{align}
\end{prop}
\begin{proof}
 By Proposition \ref{p-psi}, we have the identity
$$ \bar \psi( \DD_{a,\alpha}) = \bigsqcup_{\beta \models a-1} \DD_{\alpha\beta}.$$
Applying $\zeta$ to this equality and using \eqref{e-psi-gamma} gives
$$ \bar\gamma (\DD'_{a,\alpha}) = \bigsqcup_{\beta \models a-1} \DD'_{\alpha\beta}.$$
Finally applying $\gamma_l$ gives equation \eqref{e-D'-rec}.
\end{proof}

\section{Characteristic functions of Dyck paths}

\subsection{Simple characteristic function}
We are going to study the summand in $D_{\alpha}(q,t)$ as a
function of $\pi$. It is convenient to first introduce a simpler object where we drop the assumption $w\in\mathcal{WP}'_\pi$ and instead sum over all labellings.
\begin{defn}\label{defn.chi}
For $\pi\in \DD_n$, the characteristic function of $\pi$ is defined by
$$
\chi(\pi): = \sum_{w \in \Z_{>0}^{n}} q^{\inv(\pi, w)} x_w,
$$
where $\inv(\pi,w)=|\Inv(\pi,w)|$ with $\Inv\{(i,j): w_i>w_j, (i,j)\in \Area(\pi) \}$.
\end{defn}
If $i<j$ and $(i,j)$ is under $\pi$, i.e. $(i,j)\in\Area(\pi)$ we say that $i$ and $j$ \emph{attack} each other.
A visual description is to put $w_i$ into the cell $\ce{(i,i)}$. Then two labels contribute a dinv if and only if i) lower label
is greater than the higher label, ii) the cell in the column of the lower label and in the row of the higher label lies under $\pi$.

It is clear that $\chi(\pi)=\bar\chi(\zeta^{-1}(\pi))$ (defined in \eqref{e-bar-chi}) under the $\zeta$ map. Thus
\begin{prop}\label{prop.sym}
The expression for $\chi(\pi)$ above is symmetric in the variables $x_1, x_2, x_3,\ldots$, so that Definition \ref{defn.chi} correctly defines an element of $\Sym[X]$.
\end{prop}
Since $\chi(\pi)$ is the fundamental element in our construction, we
present here a proof, which is similar to the proof of Lemma 10.2 from \cite{haglund2005combinatorial}.
\begin{proof}[Proof by reduction]
We first reduce the proof of the proposition to the case $X=x_1+x_2$.
It is sufficient to prove the symmetry in $x_i$ and $x_{i+1}$ for all $i$.

Let $\psi(w)$ be obtained from $w$ by replacing $i+1$ by $i$. Then we have
\begin{align*}
  \chi(\pi)&=\sum_{\bar w} \sum_{\psi(w)=\bar w} q^{\inv(\pi, w)} x_w,
\end{align*}
where the first sum ranges over all $\bar w$ without label $i+1$. Thus it is sufficient to show the symmetry of the inner sum.

For given $w$ in $\psi^{-1}(\bar w)$, let $(\pi', w')$ be obtained by removing all labels
not equal to $i$ and $i+1$, together with the corresponding rows and columns. It should be clear that
$\inv(\pi,w)=\inv(\pi, \bar w)+ \inv(\pi',w')$. This implies it is sufficient to prove the symmetry of
$$ \sum_{w'} q^{\inv(\pi',w')} x_{w'}, $$
where $w'$ contains only $i$ and $i+1$. This reduce the proposition to the $X=x_1+x_2$ case.

For the two variables case, we prove the proposition by induction on $n$ and area of $\pi$. The base case is when $\area(\pi)=0$. Clearly we have
$\chi(\pi)=(x_1+x_2)^n$ since there are no attack relations. To show the symmetry of $\chi(\pi)$ for $\area(\pi)>0$, consider Dyck path $\pi'$ obtained from $\pi$ by removing a peak area cell $\ce{(a,b)}$.
Then we have
$$ \chi(\pi)-\chi(\pi')= \sum_{w, w_a=2, w_b=1} (q-1) q^{\inv(\pi',w)} x_w,$$
where the sum ranges over all $w$ with $w_a=2, w_b=1$, since all the other terms cancel.
Now the inversion involves $w_a$ and $w_b$ is $b-a-1$, all comes from $w_i$ for $a<i<b$. Let $(\bar \pi, \bar w)$ be obtained from $(\pi',w)$ by removing rows $a,b$, columns $a,b$. Then we have
$$ \chi(\pi)-\chi(\pi')= \sum_{\bar w} (q-1) q^{b-a-1}x_1x_2  q^{\inv(\bar{\pi},\bar{w})} x_{\bar w} =(q-1) q^{b-a-1}x_1x_2 \chi(\bar \pi).$$
By the  induction hypothesis $\chi(\pi')$ and $\chi(\bar \pi)$ are both symmetric in $x_1$ and $x_2$. The symmetry of $\chi(\pi)$ then follows.
\end{proof}

(The above induction proof suggests a direct combinatorial proof. It might be interesting to describe it clearly.)

Another way to formulate this property is as follows: for a composition $c_1+c_2+\cdots+c_k=n$ consider the multiset $M_c = 1^{c_1} 2^{c_2} \ldots k^{c_k}$. Consider the sum
$$
\sum_{w \text{\; a permutation of $M_c$}} q^{\inv(\pi, w)}.
$$
Proposition \ref{prop.sym} simply says that this sum does not depend on the order of the numbers $c_1, c_2, \ldots, c_k$, or equivalently on the linear order on the set of labels. If $\lambda$ is the partition with components $c_1,c_2,\ldots,c_k$, then this sum computes the coefficient of the monomial symmetric function $m_\lambda$ in $\chi(\pi)$,
so we have (set $h_c=h_{c_1}\cdots h_{c_k}$)
\begin{equation}\label{eq.sym}
(\chi(\pi), h_c) = \sum_{w \text{\; a permutation of $M_c$}} q^{\inv(\pi, w)}.
\end{equation}

We list here a few properties of $\chi$ so that the reader has a feeling of what kind of object it is.

For a Dyck path $\pi$ denote by $\pi^{op}$ the reversed Dyck path, i.e. the path obtained by replacing each North step by East step and each East step by North step and reversing the order of steps. By reversing also the order of the components of $c$ in (\ref{eq.sym}) we see

\begin{prop}
$$
\chi(\pi) = \chi(\pi^{op}).
$$
\end{prop}

By applying
Theorem 6.10 of \cite{haglund2008catalan} to our case, we have
\begin{prop}\label{p-superization}
Let $\pi\in \DD_n$, $X$ and $Y$ be two alphabets. Then
$$\chi(\pi)[X-Y] = \sum_{w \in \mathcal{A}^n} q^{\inv'(\pi,w)}  z_w ,  $$
where $\mathcal{A}=\Z_{>0}\bigcup \Z_{<0}$, $z_i=x_i$, $z_{\bar i}=-y_i$ for $i>0$ with $\bar i=-i$,
and for any fixed total ordering on $\mathcal{A}$ (still denoted ``$<$"),
$$\inv'(\pi,w)=|\Inv'(\pi,w)|, \quad \Inv'(\pi,w)=\{(a,b)\in \Area(\pi) : w_a> w_b \text{ or } w_a=w_b<0 \}.$$
\end{prop}
The total ordering we are going to use is $1<\bar 1<2<\bar 2<\cdots$.

For example, if $\pi$ is given by $NNENEE$, or equivalently $\Area(\pi)=\{(1,2),(2,3)\}$.
We can compute the coefficient of $m_\mu$ in $\chi(\pi)$ directly:
i) $\mu=(3)$, we need the coefficient of $x_1^3$, so the $w$ has to be $111$, i.e., $w_1=1,w_2=1,$ and $w_3=1$. We use the notation
$w\to q^{\inv(\pi,w)}$ so that $111 \to 1$. Thus $\langle \chi(\pi), h_3 \rangle =1$;
ii) $\mu=(2,1)$, we need the coefficient of $x_1^2x_2$, so $w$ has three choices: $112\to 1$, $121\to q$, $211\to q$ (recall $(1,3)$ is not an attack relation). Thus $\langle \chi(\pi), h_{2}h_1 \rangle =1+2q$;
iii) $\mu=(1,1,1)$, we need the coefficient of $x_1x_2x_3$. Then $w$ consists of all $6$ permutations:
$123 \to 1$, $132\to q$, $213\to q$, $231\to q$, $312\to q$, and $321\to q^2$. Thus $\langle \chi(\pi), h_1^3 \rangle =1+4q+q^2$.
Taking the sum gives
\begin{multline*}
  \chi(\pi)=m_3+(1+2q)m_{21}+(1+4q+q^2)m_{111} \\
= \left( {q}^{2}-2q+1\right) p_3/3+ \left( -{q}^{2}
+1 \right) p_1p_2/2+ \left( {q}^{2}+4q+1
 \right) p_1^{3}/6.
\end{multline*}
The formula of $\chi(\pi)[X-Y]$ is too lengthy to be put here. Here we only evaluate its coefficient in $x_1y_1^2$.
Clearly the coefficient of $x_1y_1^2$ in $p_3, p_1p_2, p_1^2$ are respectively $0, -1, 3$. Thus
$$[x_1 y_1^2]\; \chi(\pi)[X-Y]=-\left( -{q}^{2}
+1 \right)/2+ \left( {q}^{2}+4q+1
 \right)/2=q^2+2q .   $$
On the other hand, we have three possible $w$ consisting of $1, \bar 1, \bar 1$: $1\bar 1 \bar1\to q$, $\bar 1 1, \bar 1 \to q$, $\bar 1 \bar 1 1\to q^2$. This agree with our computation.


Applying Proposition \ref{p-superization} to the case $X=0$ gives a summation over only negative labels.
Clearly, the result is equivalent to the following.
$$
(\chi(\pi)[-X], h_c[X]) = (-1)^{|\pi|} \sum_{w \text{\; a permutation of $M_c$}} q^{\inv'(\pi, w)},
$$
where $\inv'(\pi, w)$ is the number of non-strict inversions of $w$ under the path,
\[\inv'(\pi,w):=\#\left\{(i,j) \in \Area(\pi),\ w_i\geq w_j\right\}.\]
Equivalently we have
$$
(\chi(\pi), e_c) = \sum_{w \text{\; a permutation of $M_c$}} q^{\inv'(\pi, w)}.
$$

If we reverse the order of labels, we have
$$
(\chi(\pi)[-X], h_c[X]) = (-1)^{|\pi|} \sum_{w \text{\; a permutation of $M_c$}} q^{\area(\pi) - \inv(\pi, w)},
$$
which implies
\begin{prop}
$$
\bar\omega \chi(\pi) = (-1)^{|\pi|} q^{-\area(\pi)} \chi(\pi).
$$
\end{prop}

The following is also an application of Proposition \ref{p-superization}.
\begin{prop}\label{prop.q1}
$$
\chi(\pi)[(q-1)X] = (q-1)^{|\pi|} \sum_{w \in \Z_{>0}^{|\pi|} \text{no attack}} q^{\inv(\pi, w)} x_w,
$$
where ``no attack'' means that the summation is only over vectors $w$ such that $w_i\neq w_j$ for $(i,j)\in\Area(\pi)$.
\end{prop}
\begin{proof}
Apply Proposition \ref{p-superization} to the case $X=qX, Y=-X$. Then the weight of positive label $i$ is $q x_i$ and that of $\bar i$ is $-x_i$.

Let us divide the words into the disjoint union $W_1 \bigsqcup W_2$, where $W_1$ consists of all words $w$ such that there is no pair $(w_a,w_b)$ $(a<b)$ such that $|w_a|=|w_b|$.

For words $w\in W_1$, consider $\tilde w$ obtained from $w$ by replacing each $w_a$ by $|w_a|$, so $\tilde w$ has only positive labels and is ``no attack". There are $2^{|\pi|}$ possible $w$ with the same $\tilde w$. All of these words have the same inversion number, and hence
$$ \sum_{w \in W_1} q^{\inv'(\pi, w)} x_w =(q-1)^{|\pi|} \sum_{w \in \Z_{>0}^{|\pi|} \text{no attack}} q^{\inv(\pi, w)} x_w.
$$

It remains to construct a sign reversing involution on $W_2$. The very simple case reveals how to construct the involution.
Assume $w_a$ attacks $w_b$ ($a<b$) with $|w_a|=|w_b|$, say equal to $i$, and ignore their attack relations with other labels. We group four possible such pairs: $(i,i)\to q^2 x_i^2$; $(\bar i, i)\to -q^2 x_i^2$, where one $q$ comes from $i$ and the other comes from the inversion, and the minus sign comes from $\bar i$; $(i,\bar i)\to -q x_i^2$; and $(\bar i, \bar i)\to q x_i^2$. This special case leads to the following involution on $W_2$:

Among the pairs $(w_a, w_b),\; a<b$ with $|w_a|=|w_b|$, chose such a pair so that $b$ is the largest, and then chose $a$ to be the largest. Change $w_a$ to $-w_a$. This operation may only change the inversion involving $a$: i) For $a<c$ and $|w_a|=|w_c|$, we must have $c=b$ by our choice of $b$ and $a$. The inversion changes as desired. ii) For $c<a$ and $|w_a|=|w_c|=i$, there will be no change of inversion number: either $w_c=i$ then $(i,i)$ and $(i, \bar i)$ are both not inversions, or $w_c=\bar i$ then $(\bar i, i)$ and $(\bar i, \bar i)$ are both inversions.
\end{proof}

For an elementary proof, see Appendix.

\subsection{Weighted characteristic function}
To study the summand of $D_{\alpha}(q,t)$ in (\ref{eq:Dalpha}) as a
function of $\pi$ we introduce a more general characteristic function.
Given a function $\mathrm{wt}:c(\pi)\to R$ on the set of corners of
some Dyck path $\pi$ of size $n$, let
\begin{equation}
\label{chiwtdef}
\chi(\pi, \mathrm{wt}) := \sum_{w \in \Z_{>0}^{n}} q^{\inv(\pi,w)}
\left(\prod_{(i,j)\in c(w),\;w_i\leq w_j} \mathrm{wt}(i,j) \right)x_w,
\end{equation}
so in particular $(\ref{eq:Dalpha})$ becomes
\[D_{\alpha}(q,t)=\sum_{\touch'(\pi)=\alpha} t^{\bounce(\pi)} \chi(\pi,0).\]
For a constant function $\mathrm{wt}=1$ we recover the simpler characteristic function
\begin{equation}
\label{chidef}
\chi(\pi,1)=\chi(\pi).
\end{equation}
It turns out that we can express the weighted characteristic function $\chi(\pi, \mathrm{wt})$ in terms the unweighted one evaluated at different paths. In particular this implies that $\chi(\pi, \mathrm{wt})$ is symmetric too.

\begin{prop}
We have that $\chi(\pi,{\mathrm{wt}})$ is symmetric in the $x_i$
variables, and so defines an element of $\Sym[X]$.
\end{prop}

\begin{proof}
Let $\pi$ be a Dyck path, and let $(i,j)\in c(\pi)$ be one of its
corners. We denote by $\mathrm{wt}_1$ the weight on $\pi$ which is
obtained from $\mathrm{wt}$ by setting the weight of $(i,j)$ to
$1$. Let $\pi'$ be the Dyck path obtained from $\pi$ by turning the
corner inside out, in other words the Dyck path of smallest area
which is both above $\pi$ and above $(i,j)$. Let $\mathrm{wt}_2$ be
the weight on $\pi'$ which coincides with $\mathrm{wt}$ on all corners
of $\pi'$ which are also corners of $\pi$ and is $1$ on other
corners. We claim that
\begin{equation}
\label{cornereq}
\chi(\pi, \mathrm{wt}) = \frac{q \mathrm{wt}(i,j) - 1}{q-1} \chi(\pi, \mathrm{wt}_1) + \frac{1 - \mathrm{wt}(i,j)}{q-1} \chi(\pi', \mathrm{wt}_2).
\end{equation}
The formula is easily seen to hold for the contribution from any particular $w\in\mathbb{Z}_{>0}$:
if $w_{i}>w_{j}$ then $\inv(\pi', w)=\inv(\pi,w)+1$;
if $w_i\leq w_j$ then $\inv(\pi',w)=\inv(\pi,w)$.

The result now follows because we may recursively express any
$\chi(\pi,{\mathrm{wt}})$ in terms of $\chi(\pi)$, which we have
already proved to be symmetric.
\end{proof}



\begin{example}
\label{easydyck}
In particular, we can use this to extract $\chi(\pi,0)$ from
$\chi(\pi',1)$ for all $\pi'$.
If $S\subset c(\pi)$ is any
subset of the set of corners, let $\pi_S\in \mathbb{D}$ denote the path
obtained by flipping the corners that are in $S$. Then
equation \eqref{cornereq} implies that
\begin{equation}
\label{chi02chi1}
\chi(\pi,0)=(1-q)^{-|c(\pi)|} \sum_{S \subset c(\pi)} (-1)^{|S|}\chi(\pi_S,1).
\end{equation}
For instance, let $\pi$ be the Dyck path in Figure \ref{easydyckfig}.
\begin{figure}
\begin{tikzpicture}
\draw[help lines] (0,0) grid (3,3);
\draw[dashed, color=gray] (0,0)--(3,3);
\draw[->,very thick] (0,0)--(0,1);
\draw[->,very thick] (0,1)--(0,2);
\draw[->,very thick] (0,2)--(1,2);
\draw[->,very thick] (1,2)--(2,2);
\draw[->,very thick] (2,2)--(2,3);
\draw[->,very thick] (2,3)--(3,3);
\end{tikzpicture}
\caption{}
\label{easydyckfig}
\end{figure}

We can compute $\chi(\pi)$ directly:
1) for $m_3$, there is only one word $111\to 1$;
2) for $m_{2,1}$, there are 3 words $112 \to 1$, $1\underline{21} \to 1$, $211 \to q$;
3) for $m_{1,1,1}$, there are 6 words $123 \to 1$, $1\underline{32} \to 1$, $213 \to q$, $2\underline{31} \to 1$, $312 \to q$, $3\underline{21} \to q$. Where we have underlined those words counted by $\chi(\pi,0)$. Thus
\[\chi(\pi)=m_3+(2+q)m_{21}+(3+3q)m_{111}=s_3+(q+1)s_{21}+qs_{111}.\]
and
\[\chi(\pi,0)=m_{21}+(2+q)m_{111}=s_{21}+qs_{111}.\]
Similarly, if $\pi'=\pi_{\{(2,3)\}}$ we have
\[\chi(\pi')=m_3+(1+2q)m_{21}+(1+4q+q^2)m_{111}=s_3+2qs_{21}+q^2s_{111}.\]
By formula \eqref{chi02chi1}, we obtain
\[\chi(\pi,0)=(1-q)^{-1}\left(\chi(\pi)-\chi(\pi')\right)=s_{21}+qs_{111}.\]
This agree with our direct computation.


\end{example}

\begin{example}
We can check that the Dyck path from Example \ref{easydyck} is the
unique one satisfying $\touch'(\pi)=(1,2)$, and that
$\bounce(\pi)=1$.
Therefore, using the calculation that followed we have that
\[D_{(1,2)}(q,t)=t\chi(\pi,0)=ts_{21}+qts_{111}\]
which can be seen to agree with $\nabla C_1 C_2(1)$.
\end{example}

\begin{example}

Though we will not need it, this weighted characteristic function
can be used to describe an interesting reformulation
of the formula for the modified Macdonald polynomial given in \cite{haglund2005combinatorial}.
This is because one of the statistic is defined by inversion according to attack relations, which can be transformed to
attack relation for a Dyck path.

Let $\mu=(\mu_1\geq \mu_2 \geq \cdots \geq \mu_l)$ be a partition of size $n$. Let us list the cells of $\mu$ in the reading order:
$$
(l,1), (l,2), \ldots, (l, \mu_l), (l-1,1), \ldots, (l-1, \mu_{l-1}) ,\ldots, (1,1),\ldots, (1,\mu_1).
$$
Denote the $m$-th cell in this list by $(i_m, j_m)$.

We say that a cell $(i,j)$ attacks all cells which are after $(i,j)$ and before $(i-1,j)$. Thus $(i,j)$ attacks precisely $\mu_i-1$ following cells if $i>1$ and all following cells if $i=1$. Next construct a Dyck path $\pi_\mu$ of length $n$ in such a way that $(m_1, m_2)$ with $m_1<m_2$ is under the path if and only if $(i_{m_1}, j_{m_1})$ attacks $(i_{m_2}, j_{m_2})$. (Check that this really defines a Dyck path).
More specifically, the path begins with $\mu_l$ North steps, then it has $\mu_l$ pairs of steps East-North, then $\mu_{l-1}-\mu_l$ North steps followed by $\mu_{l-1}$ East-North pairs and so on until we reach the point $(n-\mu_1, n)$. We complete the path by performing $\mu_1$ East steps.

Note that the corners of $\pi_\mu$ precisely correspond to the pairs of
cells $(i,j), (i-1, j)$. We set the weight of such a corner to
$q^{\mathrm{arm}(i,j)} t^{-1-\mathrm{leg}(i,j)}$ and denote the weight
function thus obtained by $\mathrm{wt}_\mu$. Note that in our
convention for $\chi(\pi,\mathrm{wt})$ we should count non-inversions in
the corners, while in \cite{haglund2005combinatorial} they count
``descents,'' which translates to counting inversions in the
corners. Taking this into account, we obtain a translation of their
 Theorem 2.2:
$$
H_\mu = q^{-n(\mu')+\binom{\mu_1}{2}} t^{n(\mu)} \chi(\pi_\mu, \mathrm{wt}_\mu).
$$
\end{example}




\section{Raising and lowering operators}

Now let $\DD_{k,n}$ be the set of Dyck paths from $(0,k)$ to
$(n,n)$, which we will call partial Dyck paths, and let $\DD_k$ be their union
over all $n$. For
$\pi\in\DD_{k,n}$ let $|\pi|=n-k$ denote the number of North steps.
Unlike $\DD=\DD_0$, the union of the sets $\DD_k$ over all $k$ is closed under the operation of
adding a North or East step to the beginning of the path (we do not allow adding a North step to a path in $\DD_0$), and
any Dyck path may be created in such a way starting with the empty path
in $\DD_0$. This is the set of paths that we will develop a recursion for.
More precisely, we will define an extension of the function
$\chi$ to a map from $\DD_k$ to a new vector space $V_k$,
and prove that certain operators on these vector spaces commute with
adding North and East steps.

\subsection{The $\Delta_{uv}$ and $\Delta_{uv}^*$ operators}
We shall introduce two important linear operators $\Delta_{uv}$ and $\Delta_{uv}^*$ as follows.
Given a polynomial $P$ depending on variables $u, v$ define
\[(\Delta_{uv} P)(u,v) = \frac{(q-1) v P(u,v) + (v - q u) P(v,u)}{v - u},\]
\[(\Delta_{uv}^* P)(u,v) = \frac{(q-1) u P(u,v) + (v - q u) P(v,u)}{v - u}.\]
We shall also establish some properties of these operators for later use.

\begin{lem}\label{Dellem}
If $Q(u,v)$ is symmetric in $u,v$ then
\[(\Delta_{uv} Q)(u,v) =q Q(u,v), \qquad  (\Delta^*_{uv} Q)(u,v) = Q(u,v).   \]
In other words, $\Delta -q$ and $\Delta^*_{uv}-1$ acts by $0$ on polynomials symmetric in $u,v$.

The two operators commutes with multiplication by polynomials symmetric in $u,v$. To be precise,
for any $P(u,v)$ we have
\[(\Delta_{uv} QP)(u,v) =Q(u,v) (\Delta_{uv} P)(u,v),\qquad  (\Delta^*_{uv} QP)(u,v) =Q(u,v) (\Delta^*_{uv} P)(u,v).   \]
\end{lem}
\begin{proof}
Obvious by definition.
\end{proof}

We can recognize these operators as a simple modification of
Demazure-Lusztig operators. 

\begin{prop} We have the following relations:
\label{Delprop}
$$\Delta_{uv}^* = q \Delta_{uv}^{-1}, \qquad \Delta_{uv}^* = \Delta_{uv}+1-q.$$
$$
(\Delta_{uv} - q) (\Delta_{uv} + 1) = 0,\qquad (\Delta_{uv}^* - 1) (\Delta_{uv}^* + q) = 0,
$$
$$
\Delta_{uv}\Delta_{vw}\Delta_{uv}=\Delta_{vw}\Delta_{uv}\Delta_{vw},\qquad
\Delta_{uv}^*\Delta_{vw}^*\Delta_{uv}^*=\Delta_{vw}^*\Delta_{uv}^*\Delta_{vw}^*.
$$

\end{prop}

\begin{proof}
We first show the second equation, which is equivalent to $\Delta_{uv}+1=\Delta^*_{uv}+q$. Indeed, we have
\begin{align*}
(\Delta_{uv}+1) P(u,v) &= \left(\frac{(q-1)vP(u,v)+(v-qu)P(v,u) }{v-u}+ P(u,v)\right)\\
  &=\left(\frac{(qv-u)P(u,v)+(v-qu)P(v,u) }{v-u}\right)\\
\end{align*}
and
\begin{align*}
(\Delta^*_{uv}+q) P(u,v) &=  \frac{(q-1) u P(u,v) + (v - q u) P(v,u)}{v - u}+ qP(u,v)\\
  &=\frac{(qv-u)P(u,v)+(v-qu)P(v,u) }{v-u}.
\end{align*}
The second equation thus follows and the resulting polynomial is a symmetric polynomial in $u,v$.
The third and the fourth equation then follows by Lemma \ref{Dellem}.

Now the third equation can be written as
\begin{align*}
  q=\Delta_{uv}(\Delta_{uv}+1-q) = \Delta_{uv} \Delta^*_{uv}.
\end{align*}
This is equivalent to the first equation.

The equivalence of the fifth and the sixth equation follows from the first equation. %
%

It remains to show the fifth equation. It seems that the best way is by direct computation.
\end{proof}

We need the following representations.
\begin{lem}
For nonnegative integers $a,b$, we have
\begin{align}
\Delta_{uv}^* u^a v^b =\left\{
                         \begin{array}{ll}
                           u^av^a \left(qu h_{b-a-1}[u+v]-uv h_{b-a-2}[u+v]\right), & \hbox{if $a\le b$;} \\
                           u^bv^b \left(h_{a-b}[u+v]-qu h_{a-b-1}[u+v]\right), & \hbox{if $a\ge b$.} \rule{0pt}{16pt}
                         \end{array}
                       \right.
\label{e-Delta*s}
\end{align}
In particular we have
\begin{align}
\Delta_{uv}^* 1&=1, \qquad  \Delta_{uv}^* u =(1-q)u+v, \qquad \Delta_{uv}^* v =qu \tag{\ref{e-Delta*s}a }\\
\Delta_{uv}^* u^a &=v^a+(1-q)\sum_{i=0}^{a-1}v^i u^{a-i},
\tag{\ref{e-Delta*s}b }\\
\Delta_{uv}^* v^a &=u^a -(1-q)\sum_{i=0}^{a-1}v^i u^{a-i}.\tag{\ref{e-Delta*s}c }
\end{align}
\end{lem}
\begin{proof}

We first prove \eqref{e-Delta*s}.
If $a\le b$, then
\begin{align*}
 \Delta_{uv}^* u^a v^b &= u^a v^a \Delta_{uv}^* v^{b-a} =u^av^a \frac{(q-1) u v^{b-a} + (v - q u) u^{b-a}}{v - u}\\
  &= u^av^a \frac{(vu^{b-a}-uv^{b-a})+qu(v^{b-a}-u^{b-a}}{v-u} \\
  &=u^av^a \left(qu h_{b-a-1}[u+v]-uv h_{b-a-2}[u+v]\right).
\end{align*}
If $a\ge b$, then
\begin{align*}
 \Delta_{uv}^* u^a v^b &= u^b v^b \Delta_{uv}^* u^{a-b} =u^bv^b \frac{(q-1) u u^{a-b} + (v - q u) v^{a-b}}{v - u}\\
  &= u^bv^b \frac{(v^{a-b+1}-u^{a-b+1})+qu(u^{a-b}-v^{a-b}}{v-u} \\
  &=u^bv^b \left(h_{a-b}[u+v]-qu h_{a-b-1}[u+v]\right).
\end{align*}
Direct computation shows:
\begin{align*}
 \Delta_{uv}^* 1 &=\frac{(q-1) u  + (v - q u) }{v - u} =1,\\
\Delta_{uv}^* u &= \frac{(q-1) u\cdot u + (v - q u) v}{v - u} =(1-q)u+v,\\
\Delta_{uv}^* v &= \frac{(q-1) u v + (v - q u) u}{v - u} =qu.
\end{align*}
Note that once we obtain the third equality, we can also obtain the second one by
$$\Delta_{uv}^* u =\Delta_{uv}^* (u+v) - \Delta_{uv}^* v =u+v-qu.$$
Similar by using $\Delta_{uv}^* (u^a+v^a)=u^a+v^a$, we only need to prove (\ref{e-Delta*s}b):
\begin{align*}
 \Delta_{uv}^* u^a &= \left(h_{a}[u+v]-qu h_{a-1}[u+v]\right) \\
 &= \sum_{i=0}^a u^{i} v^{a-i} -qu \sum_{i=0}^{a-1} u^i v^{a-1-i} \\
&=v^a+ \sum_{i=1}^a u^{i} v^{a-i} -q \sum_{i=0}^{a-1} u^{i+1} v^{a-1-i} \\
&=v^a +(1-q)\sum_{i=1}^a u^{i} v^{a-i}\\
&=v^a+(1-q)\sum_{i=0}^{a-1}v^i u^{a-i}.
\end{align*}
\end{proof}

Let $m$ be a positive integer and $W^m$ be the linear space of homogeneous polynomials of total degree $m$ in $u,v$.
Since $\Delta_{uv}^* $ preserves the total degree in $u,v$, it can be treated as a linear transformation on
$W^m$. Its minimal polynomial is $(x+q)(x-1)$.
\begin{lem}\label{l-im-ker}
We have the direct sum $W^m=Im(\Delta_{uv}^*+q)\bigoplus ker(\Delta_{uv}^*+q)$, where
$$ Im(\Delta_{uv}^*+q)=ker(\Delta_{uv}^*-1)=\{Q(u,v)\in W^m : Q(u,v)=Q(v,u)\}, $$
$$ ker(\Delta_{uv}^*+q)=Im(\Delta_{uv}^*-1)=\{(qu-v)Q(u,v)\in W^m : Q(u,v)=Q(v,u)\}. $$
\end{lem}
\begin{proof}
Let $W_1$ and $W_2$ be the right hand side of the two formulas respectively. It is easy to check that
$Im(\Delta_{uv}^*+q)\supseteq W_1 $ and $ker(\Delta_{uv}^*+q) \supseteq W_2$.
Since $W_1+W_2$ is a direct sum, it is sufficient to show that $\dim W^m =\dim W_1+\dim W_2$, which is easy.

The other parts are easy consequence of linear algebra. It is worth mentioning that $Im(\Delta_{uv}^*-1)$ is easy to compute:
\begin{align*}
(\Delta^*_{uv}-1) P(u,v) &=  \frac{(q-1) u P(u,v) + (v - q u) P(v,u)}{v - u}-P(u,v)\\
  &=\frac{(qu-v)P(u,v)+(v-qu)P(v,u) }{v-u}\\
  &=(qu-v) \frac{P(u,v)-P(v,u)}{v-u}.
\end{align*}
This completes the proof.
\end{proof}

\subsection{The space $V_k$ and some operators}
\begin{defn}
Let $V_k=\Sym[X]\otimes \Q[y_1,y_2,\ldots,y_{k}]$,
and let
\[T_i = \Delta_{y_i y_{i+1}}^* : V_{k} \rightarrow
V_{k},\quad i=1,\ldots,k-1.\]
Define operators
\[d_{+} : V_k \rightarrow V_{k+1},
\quad d_- : V_{k} \rightarrow V_{k-1}\]
by
\begin{equation}
\label{dpdef}
(d_{+} F)[X] = T_1 T_2\cdots T_k \left(F[X+(q-1)y_{k+1}]\right),
\end{equation}
%
for $F\in V_k$,
 and
\begin{equation}
\label{dmdef}
d_-(y_k^i F)= -B_{i+1} F
\end{equation}
when $F$ does not depend on $y_k$.
\end{defn}

We have the following alternative description of $d_-$.
\begin{lem}\label{l-d-}
For $G\in V_k$, we have
$$ d_- G =- y_k G[X-(q-1)y_k] \pExp[-X/y_k]\Big|_{y_k^{0}} .$$
In other words, $d_-G$ is obtained by expanding $G[X-(q-1)y_k]$ as a
power series in $y_k$ and then replacing $y_k^i$ by $-h_{i+1}[-X]$ for all $i$.
\end{lem}
\begin{proof}
 By linearity, it is sufficient to assume $G=y_k^i F$ for $F$ free of $y_k$.
By direct computation, we have
\begin{align*}
  d_- y_k^i F &= - y_k^{i+1} F[X-(q-1)y_k] \pExp[-X/y_k]\Big|_{y_k^{0}} \\
              &= - F[X-(q-1)y_k] \pExp[-X/y_k]\Big|_{y_k^{-i-1}} \\
(\text{by } y_k=z^{-1}) &= -  F[X-(q-1)/z] \pExp[-Xz]\Big|_{z^{i+1}} \\
                   &=-B_{i+1} F
\end{align*}
This completes the proof.
\end{proof}

We now claim the following theorem:
\begin{thm}
\label{dpthm}
For any Dyck path $\pi$ of size $n$,
let $\eps_1\cdots \eps_{2n}$ denote the
corresponding sequence
of plus and minus symbols where a plus denotes an
east step, and a minus denotes a north step reading
$\pi$ from bottom left to top right. Then
\[\chi(\pi)=d_{\eps_1}\cdots d_{\eps_{2n}}(1)\]
as an element of $V_0=\Sym[X]$.
\end{thm}

\begin{example}
Let $\pi$ be the Dyck path from Example \ref{easydyck}.
We have that
\[d_- d_- d_+d_+d_-d_+(1)=
d_- d_- d_+d_+d_-(1)=d_- d_- d_+d_+(s_1)=\]
\[d_- d_- d_+\left(s_1+(q-1)y_1\right)=
d_- d_- \left(s_1+(q-1)(y_1+y_2)\right)=\]
\[d_- \left(s_2+s_{11}+(q-1)s_1y_1\right)=s_3+(1+q)s_{21}+qs_{111},\]
which agrees with the value calculated for $\chi(\pi)$.
\end{example}

Combining this result with equation \eqref{chi02chi1} implies the
following:
\begin{cor}
\label{chi0cor}
The following procedure computes $\chi(\pi,0)$: start with $1\in \Sym[X]=V_0$, follow the path from right to left applying $\frac{1}{q-1}[d_{-}, d_{+}]$ for each corner of $w$, and $d_{-}$ ($d_{+}$) for each North (resp. East) step which is not a side of a corner.
\end{cor}

\begin{rem}
Before we proceed to the proof of Theorem \ref{dpthm}, we would like to
explain why we expected such a result to hold and how we obtained it. First note that the number of Dyck paths of length $n$ is given by the Catalan number $C_n=\frac{1}{n+1} \binom{2n}{n}$ which grows exponentially with $n$. The dimension of the degree $n$ part of $\Sym[X]$ is the number of partitions of size $n$, which grows subexponentially. For instance, for $n=3$ we have $5$ Dyck paths, but only $3$ partitions. Thus there must be linear dependences between different $\chi(\pi)$. Now fix a partial Dyck path $\pi_1\in \DD_{k,n}$. For each partial Dyck path $\pi_2\in\DD_{k,n'}$ we can reflect $\pi_2$ and concatenate it with $\pi_1$ to obtain a full Dyck path $\pi_2^{op}\pi_1$ of length $n+n'-k$. Then we take its character $\chi(\pi_2^{op}\pi_1)$. We keep $n$, $\pi_1$ fixed and vary $n'$, $\pi_2$. Thus we obtain a map $\varphi_{\pi_1}:\DD_k\to\Sym[X]$. The map $\pi_1\to \varphi_{\pi_1}$ is a map from $\DD_k$ to the vector space of maps from $\DD_k$ to $\Sym[X]$, which is very high dimensional, because both the set $\DD_k$ is infinite and $\Sym[X]$ is infinite dimensional. A priori it could be the case that the images of the elements of $\DD_{k,n}$ in $\mathrm{Maps}(\DD_k,\Sym[X])$ are linearly independent. But computer experiments convinced us that it is not so, that there should be a vector space $V_{k,n}$ whose dimension is generally smaller than the size of $\DD_{k,n}$, so that we have a commutative diagram
\[
\begin{tikzcd}
\DD_{k,n} \arrow{r}{\chi_{k,n}} \arrow{d}{\varphi} & V_{k,n}\arrow{ld}\\
\mathrm{Maps}(\DD_k,\Sym[X]) &
\end{tikzcd}
\]
We then guessed that $V_{k,n}$ should be the degree $n$ part of $V_k=\Sym[X]\otimes\Q[y_1,\ldots,y_k]$, and from that conjectured a definition of $\chi_k:\DD_k\to V_k$ defined below, and verified on examples that partial Dyck paths that are linearly dependent after applying $\chi_{k,n}$ satisfy the same linear dependence after applying $\varphi$. Once this was established, the computation of the operators $d_{-}$, $d_{+}$ and proof of Theorem \ref{dpthm} turned out to be relatively straightforward.
\end{rem}


The proof of Theorem \ref{dpthm} will be divided in several parts.

\subsection{Characteristic functions of partial Dyck paths}

The following definition is motivated by Proposition \ref{prop.q1}. Let $\pi \in \DD_{k,n}$. Let $\sigma=(\sigma_1, \sigma_2,\ldots,\sigma_k)\in \Z_{>0}$ be a tuple of distinct numbers. The elements of $\Im(\sigma)\subset \Z_{>0}$ will be called \emph{special}. Let
\[
U_{\pi,\sigma}=\left\{w\in \mathbb{Z}^{n}_{>0} :
w_i=\sigma_i \mbox{ for } i \leq k,\
w_i \neq w_j \mbox{ for } (i,j) \in \Area(\pi)\right\}.
\]
The second condition on $w$ is the ``no attack'' condition as before, where $\Area(\pi)$ refers to $\Area(N^k \pi)$. The first condition says that we put the special labels in the positions $1,2,\ldots,k$ as prescribed by $\sigma$.
Let
\begin{equation}\label{eq:chiprime}
\chi_\sigma'(\pi)=\sum_{w \in U_{\pi,\sigma}} q^{\inv(\pi,w)} z_w.
\end{equation}
Here we use variables $z_1,z_2,\ldots$.

Suppose $\sigma$ is a permutation, i.e., $\sigma_i\leq k$ for all $i$. Set $z_i=y_i$ for $i\leq k$ and $z_i=x_{i-k}$ for $i>k$.
We denote
\[
\chi'_{k}(\pi)=\chi_{(1,2,\ldots,k)}'(\pi).
\]

Let us group the summands in \eqref{eq:chiprime} according to the positions of
special labels. More precisely, let $S\subset\{1,\ldots,n\}$ such that $\{1,\ldots,k\}\subset S$ and $w^S:S\to \{1,\ldots,k\}$ such that $w^S_i=\sigma_i$ for $i=1,2,\ldots,k$ and $w_i\neq w_j$ for $i,j\in S$, $(i,j)\in\Area(\pi)$. Set
\[
U_{\pi,\sigma}^{S,w^S}:=\left\{w\in U_{\pi,\sigma} :
w_i=w^S_i\mbox{ for } i \in S,\; w_i>k\mbox{ for } i \notin S\right\},
\]
\[
\Sigma_{\pi,\sigma}^{S,w^S}:=\sum_{w \in U_{\pi,\sigma}^{S,w^S}} q^{\inv(\pi,w)} z_w,\qquad \chi_\sigma'(\pi) = \sum_{S, w^S} \Sigma_{\pi,\sigma}^{S,w^S}.
\]
Let $m_1<m_2<\cdots<m_r$ be all the positions not in $S$.
Let $\pi_S$ be the unique Dyck path\footnote{The shape of $\pi_S$ is the Dyck path obtained by removing all rows and columns indexed by $S$.} of length $r$ such that $(i,j)\in\Area(\pi_S)$ if and only if $(m_i, m_j)\in \Area(\pi)$. We have
\[
\Sigma_{\pi,\sigma}^{S,w^S} = q^A \prod_{i\in S} y_{w_i} \sum_{w \in \Z_{>k}^{r} \text{no attack}} q^{\inv(\pi_S, w)} z_w,
\]
where
\[
A=\#\{(i,j)\in\Area(\pi) : (i\in S, j\in S, w^S_i>w^S_j) \mbox{ or } (i\notin S, j\in S)\}.
\]

By Proposition \ref{prop.q1} we have
\begin{equation}\label{eq:chitochik}
\Sigma_{\pi,\sigma}^{S,w^S} = q^A (q-1)^{|S|-n} \chi(\pi_S)\left[(q-1)X\right] \prod_{i\in S} y_{w_i}.
\end{equation}
In particular, $\chi_\sigma(\pi)$ is a symmetric function in $x_1, x_2,\ldots$ and it makes sense to define
\begin{equation}
\label{chikdef}
\chi_\sigma(\pi)[X] := \frac{1}{y_1 y_2\cdots y_k} (q-1)^{|\pi|}\chi_\sigma'(\pi)\left[\frac{X}{q-1}\right]\in V_k,\quad \chi_k(\pi):=\chi_{\Id_k}(\pi).
\end{equation}

\begin{rem}
The identity (\ref{eq:chitochik}) also implies that the coefficients of $\chi_\sigma(\pi)[X]$ are polynomials in $q$ and gives a way to express $\chi_\sigma$ in terms of the characteristic functions $\chi(\pi_S)$ for all $S$. Do we have a simple description of $\chi_k(\pi)$?
\end{rem}

For $k=0$ we recover $\chi(\pi)$:
\[
\chi_0(\pi)=\chi(\pi)\qquad (\pi\in \mathbb{D}_0=\mathbb{D}).
\]
Thus, it suffices to prove that
\begin{equation}\label{eq:recs1}
\chi_{k+1}(E\pi)=d_{+} \chi_{k}(\pi),\quad
\chi_{k-1}(N\pi)=d_{-} \chi_{k}(\pi)
\quad (\pi \in \DD_k).
\end{equation}

\subsection{Raising operator}
We begin with the first equality in \eqref{eq:recs1}. Let $\pi\in\DD_{k,n}$ so that $E\pi \in \DD_{k+1, n+1}$, and we need to express $\chi_{k+1}(E\pi)$ in terms of $\chi_k(\pi)$. Let $\sigma$ be the following sequence:
\[
\sigma=(k+1, 1, 2, \ldots, k).
\]
Then we have the natural bijection $f:U_{\pi,\Id_k} \to U_{E\pi,\sigma}$ obtained
by sending
\[
w=(1, 2, \ldots, k, w_{k+1},\ldots,w_n)
\]
to
\[
f(w):=(k+1, 1, 2, \ldots,k, w_{k+1},\ldots,w_n).
\]
To see this is well-defined, we observe that $\Area(E\pi)$ has some new cells in the first column. The ``no attack" condition is satisfied because $1$ does not attack $k+i$ in $E\pi$ for $i> 1$.
We clearly have $\inv({E\pi}, f(w)) = \inv(\pi, w) + k$, which implies
\begin{equation*}
\chi'_{\sigma}(E\pi) = z_{k+1} q^k \chi'_k(\pi),
\end{equation*}
where both sides are written in terms of the variables $z_i$. When we pass to the variables $x_i$, $y_i$ on the left, we have
\[
(z_1,z_2,\ldots)=(y_1,y_2,\ldots,y_{k+1},x_1,x_2,\ldots),
\]
but on the right we have
\[
(z_1,z_2,\ldots)=(y_1,y_2,\ldots,y_{k},x_1,x_2,\ldots),
\]
thus we need to perform the substitution $X=y_{k+1}+X$:
\[
\chi'_{\sigma}(E\pi)[X] = y_{k+1} q^k \chi'_k(\pi)[X+y_{k+1}],
\]
Performing the transformation \eqref{chikdef} we obtain
$$\chi_{\sigma}(E\pi)[(q-1)X] = q^k \chi_k(\pi)\left[(q-1)X+(q-1)y_{k+1}\right].
$$
By setting $X=X/(q-1)$, we obtain
\begin{equation}
\label{chi1dp}
\chi_{\sigma}(E\pi)[X] = q^k \chi_k(\pi)\left[X+(q-1)y_{k+1}\right].
\end{equation}
To finish the computation we need to relate $\chi_{k+1}=\chi_{\Id_{k+1}}$ and $\chi_\sigma$. We first note that $\sigma$ can be obtained from $\Id_{k+1}$ by successively swapping neighboring labels. Let $\sigma^{(1)}=\Id_{k+1}$ and
\[
\sigma^{(i)} = (i, 1, 2,\ldots,i-1,i+1,\ldots,k+1)\qquad (i=2,3,\ldots,k+1),
\]
so that $\sigma = \sigma^{(k+1)}$. It is clear that $\sigma^{(i+1)}$ can be obtained from $\sigma^{(i)}$ by interchanging the labels $i$ and $i+1$.

We show below (Proposition \ref{prop:swapping}) that this kind of interchange is controlled by the operator $\Delta_{y_i, y_{i+1}}$:
\begin{equation}
\label{chisigrec}
\chi_{\sigma^{(i+1)}}(E\pi)=\Delta_{y_i,y_{i+1}}
\chi_{\sigma^{(i)}}(E\pi).
\end{equation}

This implies
\[
\chi_\sigma(E\pi) = \Delta_{y_{k-1},y_{k}}\cdots \Delta_{y_{1},y_{2}} \chi_{k+1}(E\pi).
\]
When we insert this equation
into \eqref{chi1dp} and apply the first formula in Proposition \eqref{Delprop}, we arrive at
\[
\chi_{k+1}(E\pi) = T_{1}\cdots T_{k}
\left(\chi_k(\pi)\left[X+(q-1) y_{k+1}\right]\right) = d_+ \chi_k(\pi).
\]

\subsection{Swapping operators}
\begin{prop}\label{prop:swapping}
For any $\pi \in \DD_k$, $\sigma$ as above and $m$ special suppose that
$m+1$ is not special or $\sigma^{-1}(m)<\sigma^{-1}(m+1)$. Then we have
\[
\chi'_{\tau_m \sigma}(w) = \Delta_{z_{m}, z_{m+1}} \chi'_{\sigma},
\]
where $\tau_m$ is the transposition $m\leftrightarrow m+1$, $(\tau_m\sigma)_i = \tau_m(\sigma_i)$ for $i=1,\ldots,k$.
\end{prop}
\begin{proof}
We decompose both sides as follows. For any  $w \in U_{\pi, \sigma}$ let $S(w)$ be the set of indices $j$ where $w_j\in\{m,m+1\}$. For $w,w'\in U_{\pi, \sigma}$ write $w\sim w'$ if $S(w)=S(w')$ and $w_i=w_i'$ for all $i\notin S(w)$.
This defines an equivalence relation on $U_{\pi, \sigma}$.
The sum \eqref{eq:chiprime} is then decomposed as follows:
\begin{equation}\label{eq:rundecomp}
\chi'_\sigma(\pi) = \sum_{[w]\in U_{\pi, \sigma}/\sim} q^{\inv_1(\pi, w)} \prod_{i\notin S} z_{w_i} \sum_{w'\sim w} a(w'),
\end{equation}
where
\[
\inv_1(\pi, w) = \#\{(i,j)\in\Area(\pi): w_i>w_j, i\notin S(w) \mbox{ or } j\notin S(w)\},
\]
which does not depend on the choice of a representative $w$ in the equivalence class $[w]$, and
\[
a(w) = q^{\inv_2(\pi, w)} \prod_{i\in S} z_{w_i},
\]
\[
\inv_2(\pi, w) = \#\{(i,j)\in\Area(\pi): w_i>w_j, i,j \in S(w)\}.
\]

Let $f:U_{\pi, \sigma}\to U_{\pi, \tau_m\sigma}$ be the bijection defined by $f(w)_i = \tau_m(w_i)$. This bijection respects the equivalence relation $\sim$ and we have $S(f(w))=S(w)$. Moreover, we have $\inv_1(\pi, w) = \inv_1(\pi, f(w))$.
We now make the stronger claim that for any $w\in U_{\pi, \sigma}$
\begin{equation}
\label{chisigrecref}
\sum_{w'\sim f(w)} a(w') = \Delta_{z_m,z_{m+1}}
\sum_{w'\sim w} a(w')
\end{equation}
which would imply the statement by summing over all equivalence classes.

For each $w\in U_{\pi, \sigma}$ the set $S(w)$ is decomposed into a disjoint union of \emph{runs}, i.e. subsets
\[R=\{j_1,...,j_l\} \subset \{1,...,n\},\quad j_1<\cdots <j_l\]
such that in each run $j_a$ attacks $j_{a+1}$ for all $a$ and
elements of different runs do not attack each other.
Because of the ``no attack" condition, the labels $w_{j_a}$ must alternate between $m,m+1$ and $j_a$ does not attack $j_{a+2}$. Thus to fix $w$ in each equivalence class it is enough to fix $w_{j_1}$ for each run. Suppose the runs of $S(w)$ have lengths $l_1, l_2,\ldots, l_r$ and the first values of $w$ in each run are $c_1, c_2,\ldots, c_r$ respectively.

With this information $a(w)$ can be computed as follows:
\[
a(w) = \prod_{i=1}^r a(l_i, c_i),
\]
where
\[
a(l,c):=
\begin{cases}
q^{l'-1} z_{m}^{l'} z_{m+1}^{l'} & l=2l', c=m\\
q^{l'} z_{m}^{l'+1} z_{m+1}^{l'} & l=2l'+1, c=m \\
q^{l'} z_{m}^{l'} z_{m+1}^{l'} & l=2l', c=m+1 \\
q^{l'} z_{m}^{l'} z_{m+1}^{l'+1} & l=2l'+1, c=m+1.
\end{cases}
\]

For instance, let $k=3$ and $\pi$ be the Dyck path in Figure \ref{dyckproof},
and let
\[w=(1,3,2,7,1,7,1,2) \in U_{\pi,(132)}.\]
Let $m=1$. Then we have $S(w)=\{1,3,5,7,8\}$, which decomposes into
two runs
$\{1,3,5\}$ and $\{7,8\}$. So we have $r=2$, $(l_1, l_2)=(3, 2)$, $(c_1, c_2) = (1, 1)$ and we obtain
\[
a(w)=a(3, 1) a(2, 1) = q z_1^2 z_2 z_1 z_2 = q z_1^3 z_2^2.
\]
\begin{figure}
\begin{tikzpicture}
\draw[help lines] (0,0) grid (8,8);
\draw[dashed, color=gray] (0,0)--(8,8);
\draw[->,very thick] (0,3)--(0,4);
\draw[->,very thick] (0,4)--(1,4);
\draw[->,very thick] (1,4)--(2,4);
\draw[->,very thick] (2,4)--(2,5);
\draw[->,very thick] (2,5)--(3,5);
\draw[->,very thick] (3,5)--(4,5);
\draw[->,very thick] (4,5)--(4,6);
\draw[->,very thick] (4,6)--(5,6);
\draw[->,very thick] (5,6)--(5,7);
\draw[->,very thick] (5,7)--(5,8);
\draw[->,very thick] (5,8)--(6,8);
\draw[->,very thick] (6,8)--(7,8);
\draw[->,very thick] (7,8)--(8,8);
\node at (.7,.3) {1};
\node at (1.7,1.3) {3};
\node at (2.7,2.3) {2};
\node at (3.7,3.3) {7};
\node at (4.7,4.3) {1};
\node at (5.7,5.3) {7};
\node at (6.7,6.3) {1};
\node at (7.7,7.3) {2};
\end{tikzpicture}
\caption{}
\label{dyckproof}
\end{figure}

 Note that by the assumption on $\sigma$ we have $c_1=m$, while $c_i$ can take arbitrary values $\{m, m+1\}$ for $i>1$. This implies
\[
\sum_{w'\sim w} a(w') = a(l_1, m) \prod_{i=2}^r (a(l_i, m) + a(l_i, m+1)).
\]
On the other hand we have
\[
\sum_{w'\sim f(w)} a(w') = \sum_{w'\sim w} a(f(w')) =  a(l_1, m+1) \prod_{i=2}^r (a(l_i, m) + a(l_i, m+1)).
\]

Now notice that for all $l$ the sum $a(l, m) + a(l, m+1)$ is symmetric in $z_m, z_{m+1}$. The operator $\Delta_{z_m,z_{m+1}}$ commutes with multiplication by symmetric functions and satisfies
\[\Delta_{z_m,z_{m+1}}(a(l, m)) = a(l, m+1),\]
which is reduced to check the straightforward formulas $\Delta_{u,v} 1=q$ and $\Delta_{u,v} u=v$.
This establishes \eqref{chisigrecref} and the proof is complete.
\end{proof}
\begin{rem}
The arguments used in the proof can be used to show that in the case when $m, m+1$ are both not special the function $\chi'_\sigma(\pi)$ is symmetric in $z_m, z_{m+1}$. In particular, we can obtain a direct proof of the fact that $\chi'_\sigma$ is symmetric in the variables $z_m, z_{m+1}, z_{m+2},\ldots$ for $i=\max(\sigma)+1$, without use of Proposition \ref{prop.q1}.
\end{rem}

\subsection{Lowering operator}

We now turn to the remaining identity $\chi_{k-1}(N\pi)=d_-\chi_k(\pi)$. Assume $\pi\in\DD_{k,n}$, so that $N\pi\in\DD_{k-1,n}$. We observe that
\[\chi'_{k-1}(N\pi)[X + y_{k}] =
\sum_{r\geq 0} \chi'_{k,r}(\pi)[X],\]
where
\[\chi'_{k,r}(\pi)=\chi'_{\sigma}(\pi),\quad \sigma=(1,2,...,k-1,k+r)\]
and to get to the second equality of \eqref{eq:recs1} we have summed over all possible values
of $r=w_k-k$ that do not result in an attack. It is convenient to set $x_0=y_k$.
Using Proposition \ref{prop:swapping} we can characterize $\chi'_{k,r}(\pi)$ by
\begin{equation}
\label{chi1r}
\chi'_{k,0}(\pi) = \chi'_{k}(\pi),\quad
\chi'_{k,r+1}(\pi) = \Delta_{x_r,x_{r+1}} \chi'_{k,r}(\pi)\quad(r\geq 0).
\end{equation}

Now notice that there is a unique expansion
\[\chi'_{k}(\pi)[X]=\sum_{j \geq 1} y_k^j g_j(\pi)[X+y_k],\quad g_j(\pi)\in V_{k-1}.\]
The advantage over the more obvious expansion in powers of $y_k$ is that each
coefficient $g_j[X+y_k]$ is symmetric in the variables $y_k,x_1,...$.
As a result, we have that
\[\chi'_{k,r}(\pi)[X]=\Delta_{x_{r-1},x_{r}}\cdots \Delta_{x_2,x_1}\Delta_{y_k,x_1} \sum_{i\geq 1} y_k^i g_i(\pi)[X+y_k]=
\sum_{i\geq 1} f_{i,r} g_i(\pi)[X+y_k]\]
where by writing $x_0=y_k$, we have
\[f_{i,r}=\Delta_{x_{r-1},x_{r}}\cdots
\Delta_{x_1,x_2}\Delta_{x_0,x_1}(y_k^i)\quad(i\geq 1, r\geq 0),  \qquad f_{i,0}=y_k^i.
\]
The extra symmetry in the $y_k$ variable is used to
pass $\Delta_{y_k,x_1}$ by multiplication by $g_i(\pi)[X+y_k]$.

Now we need an explicit formula for $f_{i,r}$:
\begin{prop}
Denote $X_r=y_k+x_1+\cdots+x_r$, $X_{-1}=0$, $X_0=y_k:=x_0$. We have
\[
f_{i,r}=\frac{h_i[(1-q)X_r] - h_i[(1-q)X_{r-1}]}{1-q} \quad(i\geq 1, r\geq 0).
\]
\end{prop}
\begin{proof}
Denote the right hand side by $f_{i,r}'$. The proof goes by induction on $r$. For $r=0$ both sides are equal to $y_k^i$:
$h_i[(1-q)y_k]=(1-q) y_k^i$ for $i\ge 1$ by Lemma \ref{l-h1-u}.


Thus it is enough to show
that for $r\ge 0$ we have
\begin{equation}
\label{fjrrec}
\Delta_{x_{r},x_{r+1}}(f'_{i,r})=f'_{i,r+1}.
\end{equation}
Use $X_r=X_{r-1}+x_{r}$ to write
\begin{align*}
 h_i[(1-q) X_r] &= \sum_{a=0}^i h_a[(1-q)x_{r}]  h_{i-a}[(1-q) X_{r-1}] \\
&=h_{i}[(1-q) X_{r-1}]+ \sum_{a=1}^i (1-q) x_r^a h_{i-a}[(1-q) X_{r-1}],
\end{align*}
by which we can write $f_{i,r}'$ as follows:
\begin{align}
\label{fjreq}
f_{i,r}' &=\sum_{j=1}^i x_r^j h_{i-j}[(1-q) X_{r-1}]. 
\end{align}
Now $X_{r-1}$ does not contain the variables $x_r$, $x_{r+1}$, so we have
\[
\Delta_{x_{r},x_{r+1}}(f'_{i,r}) = \sum_{j=1}^i h_{i-j}[(1-q) X_{r-1}]  \Delta_{x_{r},x_{r+1}} x_r^j.
\]
We need the formula
\begin{align*}
 \Delta_{x_r, x_{r+1}} x_r^j &= \frac{(1-q) x_{r+1} x_r^j + (x_{r+1}-qx_r) x_{r+1}^j}{x_{r+1}-x_r} \\
                             &=x_{r+1} h_{j-1}[x_r+x_{r+1}] -q x_{r}x_{r+1} h_{j-2}[x_r+x_{r+1}] \\
                             &=x_{r+1} h_{j-1}[(1-q)x_r + x_{r+1}].
\end{align*}

Now we can evaluate
\begin{align*}
\Delta_{x_r x_{r+1}} f'_{i,r} &=
x_{r+1}\sum_{j=1}^i h_{j-1}[(1-q)x_r + x_{r+1}] h_{i-j}[(1-q)X_{r-1}]\\
&=
x_{r+1} h_{i-1}[(1-q) X_{r} + x_{r+1}]\\
&=\sum_{j=1}^i x_{r+1}^j h_{i-j}[(1-q) X_{r}],
\end{align*}
which matches $f_{i,r+1}'$ by \eqref{fjreq}.
\end{proof}

Now, if we sum over all $r$, we obtain
\begin{equation}
\label{sumfir}
\sum_{r \geq 0} f_{i,r}=
(1-q)^{-1}h_i\left[(1-q)(X+y_k)\right].
\end{equation}
Thus
\[\chi_{k-1}'(N \pi)[X + y_{k}] =
(1-q)^{-1} \sum_{i\geq 1} h_i[(1-q)(X+y_k)] g_i(\pi)[X+y_k].\]
This implies
\begin{equation}
\label{chistep}
\chi_{k-1} (N\pi)[X] = - \frac{(q-1)^{n-k}}{y_1\cdots y_{k-1}} \sum_{i\geq 1} h_i[-X] g_i(\pi)\left[\frac{X}{q-1}\right].
\end{equation}
On the other hand $g_i(\pi)$ were defined in such a way that
\[
\chi_k(\pi)[(q-1)X] = \frac{(q-1)^{n-k}}{y_1\cdots y_{k}} \sum_{i\geq 1} y_k^i g_i(\pi)[X+y_{k}].
\]
Substituting $\frac{1}{q-1}X-y_k$ for $X$ gives
\begin{equation}
\label{chistep2}
\chi_k(\pi)[X-(q-1)y_k] = \frac{(q-1)^{n-k}}{y_1\cdots y_{k}} \sum_{i\geq 1}  y_k^i g_i(\pi)\left[\frac{X}{q-1}\right].
\end{equation}
Comparing \eqref{chistep} and \eqref{chistep2} we obtain
\[
\chi_{k-1} (N\pi)[X] = \sum_{i\geq 0} -h_{i+1}[-X]\left(\chi_k(\pi)[X-(q-1)y_k]\big|_{y_k^i}\right).
\]

This is just $d_- \chi_k(\pi)$ by Lemma \ref{l-d-}. Hence we
established the second case of \eqref{eq:recs1}. Therefore the proof of Theorem \ref{dpthm} is complete.

\subsection{Main recursion}
We now show how to express all of $D_{\alpha}(q,t)$ using our operators:
\begin{thm}\label{thm:recN}
If $\alpha$ is a composition of length $l$, we have
\[D_{\alpha}(q,t)=d_-^l(N_\alpha).\]
where $N_\alpha\in V_l$ is defined by the recursion relations
\begin{equation}
\label{receqs}
N_{\emptyset}=1,\quad N_{1,\alpha}=d_+N_{\alpha},\quad
N_{a, \alpha} = \frac{t^{a-1}}{q-1} [d_{-},d_{+}] \sum_{\beta\models
  a-1} d_{-}^{l(\beta)-1} N_{\alpha \beta}.
\end{equation}
\end{thm}

\begin{example}
Using Theorem \ref{thm:recN}, we find that
\[N_{31}=\frac{t^3}{(q-1)^2}
\left(d_{-++-++}-d_{-+++-+}-d_{+-+-++}+d_{+-++-+}\right)+\]
\[\frac{t^2}{q-1}\left(d_{-+-+++}-d_{+--+++}\right)=
qt^3y_1^2-qt^2y_1 e_1\in V_2,\]
where $d_{\epsilon_1\cdots \epsilon_n}=d_{\epsilon_1}\cdots d_{\epsilon_n} (1).$
We may then check that
\[d_-^2 N_{31}=qt^3B_{3}B_1(1)+qt^2B_{2}B_1B_1(1)=\nabla C_{3}C_1(1).\]
\end{example}

\begin{proof}
For any $k> 0$ let $\DD_k^0\subset \DD_k$ denote the subset of partial Dyck paths that begin with an East step. For $k=0$ let $\DD_0^0=\{\emptyset\}$.
Define functions $\chi^0:\DD_k^0\rightarrow V_k$ by
\[\chi^0(\emptyset)=1,\quad \chi^0(EN^i\widetilde{\pi})=\frac{1}{q-1}[d_{-},
d_{+}] d_-^{i-1}\chi^0(\widetilde{\pi}),\]
\[\chi^0(E\widetilde{\pi})=d_+\chi^0(\widetilde{\pi}), \widetilde{\pi} \in \DD_k^0.\]

Given a composition $\alpha$ of length $l$, recall
\[
\DD'_{\alpha}=\left\{\pi \in \DD : \touch'(\pi)=\alpha\right\}.
\]
By the definition of $\touch'$ every element of $\DD'_{\alpha}$ is of the form $\pi=N^l \tilde{\pi}$ for a unique element $\tilde{\pi}\in \DD_l^0$ so that by Corollary \ref{chi0cor} we have
\[
\chi(\pi, 0) = d_-^l \chi^0_l(\tilde{\pi}).
\]
Let
\[
N'_\alpha = \sum_{\pi \in \DD'_\alpha}
t^{\bounce(\pi)}\chi_l^0(\tilde{\pi}) \in V_l,
\]
so that $D_{\alpha}(q,t)=d_-^l(N'_{\alpha})$.
It suffices to show that $N'_\alpha$ satisfies the relations
\eqref{receqs}, and so agrees with $N_\alpha$.

We have established the following recursion in Proposition \ref{p-D'-rec}:
$$\DD'_{a,\alpha} = \bigsqcup_{\beta\models a-1} \gamma_l(\DD'_{\alpha\beta}),    $$
where for $\tau \in \DD_{\alpha\beta}$ written as $N^{l+l(\beta)}\tilde{\tau}$, we have $\gamma_{l}(\tau) = N^{l+1} E N^{l(\beta)} \tilde{\tau}$ and
\[
\bounce(\gamma_{l}(\tau)) = \bounce(\tau) + |\beta|=\bounce(\tau)+a-1.
\]
The case $a=1$ has to be distinguished since $\gamma_{\ell}(\tau) = N^{l+1} E \tilde{\tau}$. Thus we have
$$N'_{1,\alpha}=\sum_{\touch( \tau) =\alpha } t^{\bounce(\gamma_l(\tau))}\chi_{l+1}^0(E\tilde\tau)=d_+ N'_\alpha.$$

For $a>1$ we have $\gamma_{\ell}(\tau) = N^{l+1} (E N) N^{l(\beta)-1} \tilde{\tau}$ with a corner $EN$.
Thus
\begin{align*}
 N'_{a,\alpha} &=\sum_{i\ge 1} \sum_{\beta\models a-1, l(\beta)=i } \sum_{\touch(\tau) =\alpha\beta} t^{\bounce(\gamma_l(\tau))}\chi_{l+1}^0(ENN^{i-1}\tilde{\tau}) \\
&=\sum_{i\ge 1} \sum_{\beta\models a-1, l(\beta)=i } t^{a-1} \frac{1}{q-1}[d_+,d_-] d_-^{i-1} N'_{\alpha\beta}.
\end{align*}
This agrees with the relations \eqref{receqs} for $N_\alpha'$.
\end{proof}





\section{Operator relations}

\subsection{Some Useful Formulas\label{s-twist}}

The following is a well-known formula for braid relations.
\begin{lem}\label{l-T1k-i}
  For $1\le i \le k-1$ we have
$$ T_1T_2\cdots T_k T_i =T_{i+1} T_1T_2\cdots T_k, \qquad  T_{i} T_k \cdots T_2 T_1= T_k \cdots T_2 T_1 T_{i+1}.   $$
\end{lem}
\begin{proof}
  \begin{align*}
  T_1T_2\cdots T_k T_i &=T_1T_2\cdots T_{i-1} T_i T_{i+1} T_i T_{i+2}\cdots T_k \\
&=T_1T_2\cdots T_{i-1}  T_{i+1} T_i T_{i+1} T_{i+2}\cdots T_k \\
&=T_{i+1} T_1T_2\cdots T_k.
\end{align*}
The second equation follows similarly.
\end{proof}

When working in $V_k =\Sym[X] \otimes \Q[y_1,\dots,y_k]$, it is convenient to write $Y_i=y_1+\cdots +y_i$, and
$y^a=y_1^{a_1}\cdots y_k^{a_k}$ for short.
We will use two bases of $V_k$ for convenience:
\begin{enumerate}
 \item $\{ s_\lambda[X] y^a : \lambda \text{ is a partition}, a \in \N^k \}$.


\item $\{ s_\lambda[X+(q-1)Y_k] y^a : \lambda \text{ is a partition}, a \in \N^k \}$.
\end{enumerate}
The first basis is natural. The expansion by the second basis is simple when using plethystic notation. For instance,
$$ s_{\lambda}[X]= s_\lambda[X+(q-1)Y_k -(q-1) Y_k]=\sum_{\mu\subseteq \lambda } s_\mu[X+(q-1)Y_k] s_{\lambda/\mu}[-(q-1)Y_k],$$
and $s_{\lambda/\mu}[-(q-1)Y_k]$ is clearly a polynomial in the $y$'s.

The advantage of the latter basis is that $s_\lambda[X+(q-1) Y_k]$ commutes with the $T_i$ operators.
Moreover, when working with the $d_+, d_-$ operators, and $d_+^*$ operators in later section, it is natural to define
a family of twisted multiplications:
For $F\in\Sym[X]$, $G\in V_k$, $m=0,1,2,\ldots,k$ put
\[(F \ast_m G)[X] = F\left[X + (q-1)\left(\sum_{i=1}^m t y_i + \sum_{i=m+1}^k y_i\right)\right] G. \]
In particular,
$$(F\ast G)[X] =(F\ast_0 G)[X] = F\left[X + (q-1)Y_k\right] G.   $$
Thus our basis elements of $V_k$ are simply $s_\lambda \ast y^a$.


We have the following twisted commuting properties, where we only need the $m=0$ case in this section.
\begin{lem}\label{l-twistedeq}
For $F\in \Sym[X]$ and $G\in V_k$, we have
\begin{equation}
\label{twistedeq}
d_{+} (F\ast G) = F\ast d_{+} G, \quad d_{-} (F\ast_m G) = F\ast_m d_{-} G\ (0\le m <k).
\end{equation}
\end{lem}
\begin{proof}
For the first equality, we have
\begin{align*}
  d_+ (F\ast G) [X] &= d_+ \left(F\left[X+(q-1)Y_k \right] G[X] \right) \\
  &= T_1T_2\cdots T_k  F\left[X+(q-1)Y_{k+1}\right] G[X+(q-1)y_{k+1}] \\
  &= F\left[X+(q-1)Y_{k+1}\right]\cdot \left( T_1T_2\cdots T_k   G[X+(q-1)y_{k+1}]\right) \\
  &=F \ast d_+ G[X].
\end{align*}

The second equality follows by direct application of Lemma \ref{l-d-}:
\begin{align*}
d_-(F\ast G) &= y_k F\left[X -(q-1)y_k+ (q-1)\left(\sum_{i=1}^m t y_i + \sum_{i=m+1}^k y_i\right)\right] G[X-(q-1)y_k]\pExp[-X/y_k] \Big|_{y_k^0} \\
             &= F\left[X + (q-1)\left(\sum_{i=1}^m t y_i + \sum_{i=m+1}^{k-1} y_i\right)\right] \left( y_k G[X-(q-1)y_k]\pExp[-X/y_k] \Big|_{y_k^0}  \right) \\
&=F\ast_m (d_- G),
\end{align*}
as desired.
\end{proof}

\subsection{The Operator Algebra}

We have operators 
\begin{equation}
\label{operators}
e_k,d_{\pm},T_i \acts V_*=V_0\oplus V_1\oplus \cdots
\end{equation}
where $e_k$ is the projection onto $V_k$, the others are
defined as above. More precisely, an element $v\in V^*$ can be uniquely written as
$v=\sum_{k\ge 0} v_k$, where $v_k \in V_k$ for all $k$. Then
$e_k v=v_k$, and $d_\pm, T_i$ act componentwise.
It is natural to ask for a complete set of relations between them.
They are formalized in the following algebra:
\begin{defn}\label{defn:dpa}
The Dyck path algebra $\AA=\AA_q$ (over $R$) is the path algebra of the quiver with vertex set $\Z_{\geq 0}$, arrows $d_+$ from $i$ to $i+1$, arrows $d_-$ from $i+1$ to $i$ for $i\in\Z_{\geq 0}$, and loops $T_1, T_2, \ldots, T_{k-1}$ from $k$ to $k$ subject to the following relations:
$$
(T_i-1)(T_i+q)=0,\quad T_i T_{i+1} T_i = T_{i+1} T_i T_{i+1},
\quad T_i T_j = T_j T_i\quad(|i-j|>1),
$$
$$
T_i d_- = d_- T_i\ (i<k-1),\quad d_+ T_i = T_{i+1} d_+, \quad T_1 d_+^2 =
d_+^2,\quad d_-^2 T_{k-1} = d_-^2,
$$
$$
d_-(d_+ d_- - d_- d_+) T_{k-1} = q (d_+ d_- - d_- d_+) d_- \quad(k\geq2),
$$
$$
T_1 (d_+ d_- - d_- d_+) d_+ = q d_+(d_+ d_- - d_- d_+),
$$
where in each identity $k$ denotes the index of the vertex where the respective paths begin.
We have used the same letters $T_i$ to label the $i$-th loop at every
node $k>i$ to match with the previous notation. To distinguish between
different nodes, we will use $T_i e_k$ where $e_k$ is the idempotent
associated with node $k$.
\end{defn}

Observe that this is almost a complete set of relations already. The first line of equations gives the braid relations. The second line gives commuting rules for $T_i$ and $d_\pm$, except for: i) $T_{k-1}$ and $d_-$, but $T_{k-1} d_-$ is not valid on $V_k$; ii) $T_1$ and $d_+$. The last two lines gives commuting rules for $[d_+,d_-]$ and $d_\pm$.

Indeed, we will prove
\begin{thm}\label{lem:mainlemma}
The operators \eqref{operators} define
a representation of $\AA$ on $V_*$. Furthermore, we have an isomorphism
of representations
\[\varphi:\AA e_0 
\xrightarrow{\sim}
V_*\]
which sends $e_0$ to $1\in V_0$, and maps $e_k \AA e_0$ isomorphically
onto $V_k$.
\end{thm}

The proof will occupy the rest of this section.

We begin by establishing that we have a defined a representation of the algebra.
\begin{lem}
\label{replemma}
The operators $T_i$ and $d_{\pm}$ satisfy the relations of Definition \ref{defn:dpa}.
\end{lem}
Recall
\begin{align*}
  d_+ F[X] &=T_1 T_2\cdots T_k F[X+(q-1) y_{k+1}],\quad  F\in V_k,\\
  d_- F[X] &\text{ sends } y_k^i F  \to -B_{i+1} F[X], \\
  T_i &= \Delta_{y_i y_{i+1}}^*, \qquad (\Delta_{uv}^* P)(u,v) = \frac{(q-1) u P(u,v) + (v - q u) P(v,u)}{v - u}.
\end{align*}
\begin{proof}
The first line is just Proposition \ref{Delprop}.

For the second line,
the first identity $d_-T_i=T_i d_- (i<k-1)$ is easy, since $T_i$ does not affect $y_k$.

The second identity follows by applying Lemma \ref{l-T1k-i}:
\begin{align*}
  d_+ T_i(F)&=T_1\cdots T_k\left((T_i F)[X+(q-1)y_{k+1}]\right) \\
 &=T_{i+1}T_1\cdots T_k F[X+(q-1)y_{k+1}] =T_{i+1} d_+(F).
\end{align*}

For the third identity,
by iterative application of the second identity, we have
\begin{align*}
  d_+^2 F &= d_+ T_1 T_2 \cdots T_k(F[X+(q-1) y_{k+1}+(q-1) y_{k+2}])\\
&=T_2 T_3\cdots T_{k+1} d_+F[X+(q-1) y_{k+1}]\\
&=T_2 T_3\cdots T_{k+1}
T_1 T_2\cdots T_{k+1} (F[X+(q-1) y_{k+1}+(q-1) y_{k+2}]).
\end{align*}
The last $T_{k+1}$ can be removed because its argument is symmetric in
$y_{k+1}$ and $y_{k+2}$, and we obtain $T_1^{-1} d_+^2 F$.

The fourth identity is more technical.
By Lemma \ref{l-im-ker}, the operator image of $T_{k-1}-1$ consists of elements of the form $(q y_{k-1} - y_k) Q$, where $Q$ is symmetric in $y_{k-1}$ and $y_k$. Thus we need to check that $d_-^2$ vanishes on such elements. By linearity, it is sufficient to assume $Q=(y_{k-1}^a y_{k} ^b+y_{k}^a y_{k-1} ^b)F$, where $F$ does not contain the variables $y_{k-1}$ and $y_k$, and $a,b\in\Z_{\geq0}$.
We have
\begin{align*}
  d_-^2\left((q y_{k-1} - y_k) y_{k-1}^a y_{k} ^b F\right) =(q B_{a+2} B_{b+1} - B_{a+1} B_{b+2}) F.
\end{align*}
This expression is antisymmetric in $a$, $b$ by Corollary 3.4 of \cite{haglund2012compositional}, which implies our identity.

For the third line, by using the previous relations and Lemma \ref{lem:commutator} below we can write
\begin{align}
  d_- (d_+ d_- - d_- d_+) T_{k-1} &= (q-1) d_- T_1 T_2 \cdots
T_{k-1} y_k T_{k-1} \nonumber\\
&=q (q-1) d_- T_1 T_2 \cdots T_{k-2} y_{k-1} \nonumber\\
&=q (q-1) T_1 T_2 \cdots T_{k-2} y_{k-1} d_- = q (d_+ d_- - d_- d_+) d_-.
\label{eq:dminusrel}
\end{align}
Similarly for the fourth line, we have
\begin{align}
  T_1 (d_+ d_- - d_- d_+) d_+ 
&=(q-1) q^k T_1 y_1 T_1^{-1} T_2^{-1} \cdots T_k^{-1} d_+ \nonumber \\
&=
 (q-1) q^k T_1 y_1 T_1^{-1} d_+ T_1^{-1} \cdots T_{k-1}^{-1}\nonumber \\
\label{eq:dplusrel}
&= (q-1) q^k d_+ y_1 T_1^{-1} \cdots T_{k-1}^{-1}
= q d_+(d_+ d_- - d_- d_+).
\end{align}

Here we used the fact
\begin{align*}
  ( y_1 T_1^{-1} d_+) F[X] &= y_1 T_2\cdots T_{k}F[X+(q-1)y_{k+1}]\\
                         &=T_2\cdots T_{k} y_1 F[X+(q-1)y_{k+1}]\\
                         &=(T_1^{-1} d_+y_1) F[X].
\end{align*}
\end{proof}

To establish the isomorphism, we first show that we can produce the
operators of multiplication by $y_i$ from $\AA$.
\begin{lem}\label{lem:commutator}
For $F\in V_k$ we have
\begin{equation}
\label{addyeq}
(d_{+} d_{-}-d_{-} d_{+} ) F = (q-1) T_1 T_2\cdots T_{k-1} (y_k
F),\quad
y_i = \frac{1}{q} T_i y_{i+1} T_i.
\end{equation}
Consequently,
\begin{equation}\label{e-d+--y1}
 (d_{+} d_{-}-d_{-} d_{+} ) F=(q-1)q^{k-1} y_1 T_1^{-1}\cdots T_{k-1}^{-1} F.
\end{equation}
\end{lem}
\begin{proof}
The second relation is equivalent to $\Delta_{y_{i},y_{i+1}} y_i = \Delta_{y_{i},y_{i+1}}^* y_{i+1} $, which follows easily by definition.
Equation \eqref{e-d+--y1} is a consequence of the first two relations.

Now we prove the first relation. Since the operators on both sides are linear and commute with the twisted action of
$\Sym[X]$ (see Section \ref{s-twist}, Equation \ref{twistedeq}),
we may assume without loss of generality that $F=y^a$ for some $a\in \mathbb{N}^k$.

If we write $a'=(a_1,\dots, a_{k-1},0)$, then  $d_- F = -B_{a_k+1} y^{a'}=-y^{a'}h_{a_k+1}[-X].$
Now the left hand side of the first identity becomes
\[
LHS=T_1\cdots T_{k-1} \left(-y^{a'}h_{a_k+1} [-(X+(q-1) y_k)]\right)-d_{-} T_1 \cdots T_{k-1} T_k y^a.\]
The operator $d_{-}$ in the second summand involves only the variable $y_{k+1}$. Thus we have
$$
LHS=T_1 \cdots T_{k-1} y^{a'} (- h_{a_k+1} [-X-(q-1) y_k)]-d_{-} T_k y_k^{a_k} ).
$$
Hence it is enough to prove
$$
d_{-} T_k y_k^{a_k} + h_{a_k+1} [-X-(q-1) y_k)] = (1-q) y_k^{a_k+1}.
$$
Write $a_k=i$. By Equation (\ref{e-Delta*s}c), the left hand side of the above equation equals
\begin{align*}
  -h_{i+1}[-X] &- (1-q)\sum_{j=1}^i y_k^j h_{i-j+1}[-X] + h_{i+1}[-X-(q-1)y_k]\\
&=- (1-q)\sum_{j=1}^i y_k^j h_{i-j+1}[-X] + (1-q)\sum_{j=1}^{i+1} y_k^j h_{i-j+1}[-X]
= (1-q) y_k^{i+1}.
\end{align*}
\end{proof}

The operators of multiplication by $y_i$ are characterized by these
relations, and therefore come from elements of $\AA$.
We next establish the relations that these operators satisfy within $\AA$:
\begin{lem}\label{lem:addingback}
For $k\in\Z_{>0}$ define elements $y_1,\ldots,y_k\in e_k \AA e_0$ by
solving for $y_iF$ in the identities \eqref{addyeq}, so that
$$y_k=\frac{1}{q-1} T_{k-1}^{-1}\cdots T_1^{-1} (d_+d_--d_-d_+) \quad \text{ and } y_i=\frac{1}{q} T_i y_{i+1} T_i.    $$
Then the following identities hold in $\AA$:
$$
y_i T_j = T_j y_i\qquad\text{for $i\notin\{j,j+1\}$,}
$$
$$
y_i d_- = d_- y_i\ (i<k),\qquad d_+ y_i = T_1 T_2 \cdots T_i y_i (T_1 T_2 \cdots T_i)^{-1} d_+,
$$
$$
y_i y_j = y_j y_i\quad\text{for any $i,j$.}
$$
\end{lem}
\begin{proof}
The second defining relation allows us to reduce some cases by applying the $T$-operators.

We explain in detail for the first relation. We claim that for $i>j+2$ if $T_j^{-1} y_i T_j = y_i$  then $T_j^{-1} y_{i-1}  T_j=y_{i-1}$. This reduces the $i>j+1$ case to the $i=k$ case.
The reason is simply due to the commutation of $T_j$ and $T_{i-1}$. We have
$$  T_j^{-1} y_{i-1}  T_j =\frac{1}{q} T_j^{-1} T_{i-1}y_{i} T_{i-1} T_j =\frac{1}{q} T_{i-1} T_j^{-1} y_i T_j T_{i-1}=\frac{1}{q} T_{i-1} y_i  T_{i-1}=y_{i-1}.$$
Similarly the $i<j$ case reduce to the $i=1$ case.

For $j>1$ (and $j\le k-1$ since $T_k$ does not act on $V_k$), by \eqref{e-d+--y1} we have
\begin{align*}
  y_1 T_j &= \frac{1}{q^{k-1}(q-1)} (d_+ d_- - d_- d_+) T_{k-1} \cdots T_1 T_j\\
(\text{by Lemma \ref{l-T1k-i}})  &=\frac{1}{q^{k-1}(q-1)} (d_+ d_- - d_- d_+) T_{j-1} T_{k-1} \cdots T_1 \\
by (*)  &=T_j \frac{1}{q^{k-1}(q-1)} (d_+ d_- - d_- d_+) T_{k-1} \cdots T_1 \\
&= T_j y_1.
\end{align*}
where at $(*)$ we have used the fact $T_i d_-=d_-T_i (i<k-1)$ and $d_+ T_i =T_{i+1}d_+(i\le k)$, so that
$$[d_+,d_-] T_j =T_{j+1}[d_+,d_-], \qquad j\le k-2. \eqno{(*)}.$$

The $i=k$ case (hence $j<k-1$) follows in a similar way:
\begin{align*}
  y_k T_j^{-1} &= \frac{1}{q-1} T_{k-1}^{-1}\cdots T_1^{-1}[d_+,d_-] T_j^{-1} =\frac{1}{q-1} T_{k-1}^{-1}\cdots T_1^{-1} T_{j+1}^{-1} [d_+,d_-]\\
&=T_{j}^{-1}\frac{1}{q-1} T_{k-1}^{-1}\cdots T_1^{-1}  [d_+,d_-] = T_j^{-1} y_k.
\end{align*}

A similar reasoning shows that it is enough to check the second identity for $i=k-1$, the third one for $i=1$ and the last one for $i=1$, $j=k$. The other cases can be deduced from these by applying the $T$-operators.

The second identity is similar to reversing the arguments in (\ref{eq:dminusrel}):
\begin{align*}
  d_-y_{k-1} & = d_- \frac{q^{-1}}{q-1} T_{k-1} T_{k-1}^{-1}\cdots T_1^{-1} (d_+d_--d_-d_+) T_{k-1} \\
              & =\frac{q^{-1}}{q-1} T_{k-2}^{-1}\cdots T_1^{-1} (d_+d_--d_-d_+) T_{k-1} \\
& = \frac{1}{q-1} T_{k-2}^{-1}\cdots T_1^{-1}(d_+d_--d_-d_+)d_-  =y_{k-1} d_-.
\end{align*}

The third identity is similar to reversing the argument in (\ref{eq:dplusrel}):
\begin{align*}
  T_1 y_1 T_1^{-1} d_+ &= \frac{1}{q^{k}(q-1)} T_1 (d_+ d_- - d_- d_+) T_{k} \cdots T_1 T_1^{-1} d_+ \\
&= \frac{1}{q^{k}(q-1)} T_1 (d_+ d_- - d_- d_+) d_+ T_{k-1} \cdots T_1  \\
&= \frac{1}{q^{k-1}(q-1)} d_+ (d_+ d_- - d_- d_+)  T_{k-1} \cdots T_1  =d_+ y_1
\end{align*}

Thus it is left to check that $y_k y_1 = y_1 y_k$ for $k\geq 2$. Write the left hand side as
\begin{align*}
y_k y_1 &=\frac{1}{q-1} T_{k-1}^{-1} \cdots T_1^{-1} (d_+ d_- - d_- d_+) y_1 \\
&= \frac{1}{q-1} T_{k-1}^{-1} \cdots T_1^{-1} (T_1 y_1 T_1^{-1}) (d_+ d_- - d_- d_+)  \\
(\text{by the first relation})\quad  &= y_1 \frac{1}{q-1} T_{k-1}^{-1} \cdots   T_1^{-1}) (d_+ d_- - d_- d_+)=y_1y_k.
\end{align*}
\end{proof}

The following lemma completes the proof of Theorem \ref{lem:mainlemma}:
\begin{lem}
\label{lem:algbasis}
The elements of the form
\begin{equation}
\label{basiseq}
d_-^m y_1^{a_1} \cdots y_{k+m}^{a_{k+m}} d_+^{k+m} e_0
\end{equation}
with $a_{k+1}\geq a_{k+2}\geq\cdots \geq a_{k+m}$ form a basis of $\AA e_0$.
Furthermore, the representation $\varphi$ maps these elements to a basis of $V_*$.
\end{lem}
\begin{proof}
We first show that elements of the form \eqref{basiseq}, with no
condition
on the $a_i$ span $\AA$.
It suffices to check that the span of these elements is
invariant under $d_-$, $T_i$ and $d_+$ when the action is well-defined. This can be done by applying
the following reduction rules that follow
from the definition of $\AA$ and Lemma \ref{lem:addingback}:
\[T_i d_- \to d_- T_i,\quad T_j y_i \to y_i T_j \quad(i\notin \{j,j+1\}),\]
\[T_i y_i \to y_{i+1} T_i + (1-q) y_i, \quad T_i y_{i+1} \to y_{i} T_i + (q-1) y_i,\]
\[T_i d_+^{k+m} e_0 \to d_+^{k+m} e_0,\]
\[d_+ d_- \to d_- d_+ + (q-1) T_1 T_2 \cdots T_{k-1} y_k,\quad y_i d_- \to d_- y_i. \]
Note that $T_i d_- \neq d_-T_i$ when apply to $F\in V_{i+1}$, but we will never need to deal with this situation, because then $d_-F \in V_i$ and $T_i$ does not act on $V_i$.

The second row of identities are easy if we use the operator definition of $T_i$. To prove in $\AA$, we use the relation $0=(T_i-1)(T_i+q)=T_i^2+(q-1)T_i-q$, which implies that
$T_i(T_i+q-1)=q$. Thus we have $T_i^{-1}={q}^{-1}(T_i+q-1)$.

Now the relation $y_i=q^{-1} T_i y_{i+1} T_i$ can be rewritten as
\begin{align*}
  qT_i^{-1} y_i &=y_{i+1} T_i \Leftrightarrow T_i y_i = y_{i+1} T_i -(q-1)y_i,\\
  T_i y_{i+1} &= y_i qT_{i}^{-1}=y_i(T_i+q-1)=y_i T_i +(q-1)y_i.
\end{align*}
These are the identities in row two.

For the third row to be well-defined, we need $i< k+m$. The rules we are using is $T_jd_+=d_+T_{j-1}$ and $T_1 d_+^2=d_+^2$.
Then we have
$$T_i d_+^{k+m} e_0= d_+^{i-1}T_1 d_+^{k+m-i+1}=d_+^{k+m}.$$

The next step is to reduce the spanning set:
We can use the following identity, which follows from $d_-^2 T_{k-1}=d_-^2$ when acting on $V_k$:
\[d_-^m (1-T_j) y_1^{a_1}\cdots y_{k+m}^{a_{k+m}} d_+^{k+m} e_0 = 0 \qquad (k<j<k+m,\ m\ge 2).\]

Indeed, we are proving the claim that $d_-^m(1-T_j)$ acts by $0$ on $V_{k+m}$ by induction on $m$.

The base case is $m=2$. Then $j=k+1$. Since we are acting on $V_{k+2}$, the claim holds true.

Now assume the claim holds for $m-1$, we want to show that it holds for $m$.

Since $d_-^m(1-T_{j})$ acts on $V_{k+m}$, it is already $0$ if $j=k+m-1$. Otherwise
$j<k+m-1$ and we have
$$ d_-^m (1-T_{j}) =d_-^{m-1} (1-T_{j}) d_- .$$
So $d_-^{m-1} (1-T_{j})$ acts on $V_{k+m-1}$ which is thus acts by $0$ by the induction hypothesis when $k<j<k+m-1$.

Note that $T_j$ commutes with $y_j y_{j+1}$. Suppose $a_j<a_{j+1}$. Then we can rewrite the above identity as
\[0=d_-^m y_1^{a_1}\cdots y_j^{a_j} y_{j+1}^{a_j}
(1-T_j) y_{j+1}^{a_{j+1}-a_j} y_{j+2}^{a_{j+2}} \cdots y_{k+m}^{a_{k+m}} d_+^{k+m} e_0.\]
Using $T_j y_{j+1} = y_j (T_j+(q-1))$, $T_j y_r = y_r T_j$ for
$r>j+1$, and $T_j d_+^{k+m} e_0=d_+^{k+m} e_0$ we can rewrite the
identity as vanishing of a linear combination of terms of the form
\eqref{basiseq},
and the lexicographically smallest term is precisely
\[d_-^m y_1^{a_1}\cdots y_{k+m}^{a_{k+m}} d_+^{k+m} e_0.\]
Thus we can always reduce terms of the form \eqref{basiseq} which
violate the condition $a_{k+1}\geq a_{k+2}\geq\cdots \geq a_{k+m}$ to
a linear combination of lexicographically greater terms, showing that
the subspace in the lemma at least spans $\AA e_0$.

We now show that they map to a basis of $V_*$, which also establishes
that they are independent, completing the proof.
Consider the image of the
elements of our spanning set
\[d_-^m y_1^{a_1} \cdots y_{k+m}^{a_{k+m}} d_+^{k+m} (1)
= d_-^m (y_1^{a_1} \cdots y_{k+m}^{a_{k+m}}) \]
which is equal to
\begin{equation}
\label{eq:elements}
(-1)^m y_1^{a_1} y_2^{a_2} \cdots y_k^{a_k} B_{a_{k+1}+1}
B_{a_{k+2}+1} \cdots B_{a_{k+m}+1} (1).
\end{equation}
Notice that $\lambda:=(a_{k+1}+1,a_{k+2}+1,\ldots,a_{k+m}+1)$ is a partition, so
\[B_{a_{k+1}+1} B_{a_{k+2}+1} \cdots B_{a_{k+m}+1} (1) \]
is a multiple of the Hall-Littlewood polynomial
$H_{\lambda'}[-X;1/q,0]$. These polynomials form a basis of the space of
symmetric functions, thus the elements \eqref{eq:elements} form a
basis of $\bigoplus_{k\in\Z_{\geq0}} V_k$.

\end{proof}

\section{Conjugate structure}
It is natural to ask if there is a way to extend
$\nabla$ to the spaces $V_k$, recovering the original operator at $k=0$.
What we have found is that it is better to extend the composition
\begin{equation}
\label{omega1}
\mathcal{N}(F)=\nabla \bar\omega F=\nabla \omega\overline{F}
\end{equation}
where the conjugation simply makes the substitution
$(q,t)=(q^{-1},t^{-1})$, and $\omega(F)=F[-X]$ is the Weyl involution
up to a sign, $\bar\omega$ denotes the composition of these.
This is a very interesting operator, which in fact is an antilinear
involution on $\Sym[X]$ corresponding to dualizing vector bundles
in the Haiman-Bridgeland-King-Reid picture, which identifies $\Sym[X]$
with the equivariant $K$-theory of the Hilbert scheme of points in the
complex plane \cite{BKR}. The key to our proof is to extend this operator
to an antilinear involution on every $V_k$, suggesting that $V_k$
should have some undiscovered geometric interpretation as well.

We will define the operator, which was discovered experimentally to
have nice properties, by explicitly constructing the action of
$\AA$ conjugated by the conjectural involution $\mathcal{N}$.
Let $\AA^*=\AA_{q^{-1}}$, and label the corresponding generators by
$d^*_{\pm},T_i^*,e_i^*$. Denote by $z_i$ the image of $y_i$ under the
isomorphism from $\AA$ to $\AA^*$ which sends generators to
generators, and is antilinear with respect to $q\mapsto q^{-1}$.

%

\begin{thm}
\label{Astarthm}
There is an action of $\AA^*$ on $V_*$ given by the assignment
\begin{equation}
\label{eq:dplusstar}
T_i^*=T^{-1}_i,\quad d^*_{-}=d_-,\quad e_i^*=e_i,\quad
(d_+^* F)[X] = \gamma F[X+(q-1)y_{k+1}],
\end{equation}
where $F\in V_k$ and $\gamma$ is the operator which sends $y_i$ to $y_{i+1}$ for $i=1,\ldots,k$ and $y_{k+1}$ to $t y_1$.
Furthermore, it satisfies the additional relations
\begin{equation}
\label{addreleq}
z_{i+1} d_+ = d_+ z_i,\quad y_{i+1} d_+^* = d_+^* y_i,\quad z_1 d_+ = -y_1 d_+^* t q^{k+1},\quad
{d_+^{*}}^m (1) = d_+^{m}(1)
\end{equation}
for any $m\geq 1$.
\end{thm}
Note the appearence of $t$ in $\gamma$.

The first part of the statement is equivalent to say that the set of operators satisfy the set of relations
in Definition \ref{defn:dpa} but with $q$ replaced by $1/q$. We list them here for reader's convenience.
$$
(T_i^*-1)(T_i^*+1/q)=0,\quad T_i^* T^*_{i+1} T^*_i = T^*_{i+1} T^*_i T^*_{i+1},
\quad T_i^* T_j^* = T_j^* T_i^*\quad(|i-j|>1),
$$
$$
T_i^* d_- = d_- T_i^*\ (i<k-1),\quad d_+^* T_i^* = T^*_{i+1} d^*_+, \quad T^*_1 {d^*_+}^2 =
{d^*_+}^2,\quad d_-^2 T^*_{k-1} = d_-^2,
$$
$$
d_-(d_+^* d_- - d_- d^*_+) T^*_{k-1} = \frac1q (d_+^* d_- - d_- d_+^*) d_- \quad(k\geq2),
$$
$$
T_1^* (d_+^* d_- - d_- d_+^*) d_+^* = \frac1q d_+^*(d_+^* d_- - d_- d_+^*),
$$

These will be verified in the following propositions.

\begin{prop}
We have the following identities on $V_k$:
$$
(T_i^*-1)(T_i^*+1/q)=0,\quad T_i^* T^*_{i+1} T^*_i = T^*_{i+1} T^*_i T^*_{i+1},
\quad T_i^* T_j^* = T_j^* T_i^*\quad(|i-j|>1).
$$
\end{prop}
\begin{proof}
  The braid relations are obvious since $T_i^*=T_i^{-1}$.
The first equality can be verified as follows.

\begin{align*}
  (T_i^{-1}-1)(T_i^{-1}+1/q) =q^{-1}T_i^{-2}(1-T_i)(T_i+q) =0.
\end{align*}
\end{proof}

\begin{prop} We have to following identities on $V_k$.
$$
d_+^* T_i^{-1} = T_{i+1}^{-1} d_+^*, \qquad T_1^{-1} d_+^{*2} = d_+^{*2},
\qquad d_+^* y_i = y_{i+1} d_+^*.
$$
$$
T_i^* d_- = d_- T_i^*\ (i<k-1), \quad d_-^2 T^*_{k-1} = d_-^2,
$$
\end{prop}
\begin{proof}
The relations in the first line are easy from the definition. For instance,
\begin{align*}
  d_+^*y_i F &= \gamma (y_i F[X+(q-1)y_{k+1}])\\
           &=y_{i+1} \gamma F[X+(q-1)y_{k+1}] =y_{i+1} d_+^* F.
\end{align*}
The second line is just a rewrite of the original one.
\end{proof}

To verify the rest of the relations, we need to define a family of twisted multiplications:
For $F\in\Sym[X]$, $G\in V_k$, $m=0,1,2,\ldots,k$ put
\[(F \ast_m G)[X] = F\left[X + (q-1)\left(\sum_{i=1}^m t y_i + \sum_{i=m+1}^k y_i\right)\right] G. \]
Then $F\ast G$ is just $F\ast_0 G$. We have the following properties.

\begin{lem}\label{l-astm}
For $F\in \Sym[X]$, and $G\in V_k$, we have
\[d_+^*(F\ast_m G) = F \ast_{m+1} d_+^* G \ (0\leq m\leq k) ,\quad
d_-(F\ast_m G) = F\ast_m d_- G \ (0\leq m<k).\]
\end{lem}
\begin{proof}
The formula for $d_-$ appeared in Lemma \ref{l-twistedeq}.

We verify the formula for $d_+^*$ as follows.
\begin{align*}
  d_+^*(F\ast_m G) &=d_+^* F\left[X + (q-1)\left(\sum_{i=1}^m t y_i + \sum_{i=m+1}^k y_i\right)\right] G\\
  &=\gamma\left( F\left[X+(q-1)y_{k+1} + (q-1)\left(\sum_{i=1}^m t y_i + \sum_{i=m+1}^k y_i\right)\right] G[X+(q-1)y_{k+1}]\right)\\
  &= F\left[X+(q-1)ty_{1} + (q-1)\left(\sum_{i=2}^{m+1} t y_i + \sum_{i=m+2}^{k+1} y_i\right)\right] \gamma G[X+(q-1)y_{k+1}]\\
  &=F \ast_{m+1} d_+^* G.
\end{align*}
\end{proof}

Let us verify the following property.
\begin{prop}
\begin{equation}\label{eq:conjrel1}
d_-(d_+^* d_- - d_- d_+^*) T_{k-1}^{-1} = q^{-1} (d_+^* d_- - d_- d_+^*) d_-\quad(k\geq 2).
\end{equation}
\end{prop}
\begin{proof}
Rewrite it as
\begin{align*}
q d_-(d_+^* d_- - d_- d_+^*) - (d_+^* d_- - d_- d_+^*) d_-T_{k-1}&=0\\
qd_-d_+^*d_- -q d_-^2 d_+^* -d_+^*d_-^2 -d_-d_+^*d_-T_{k-1} &=0\quad (\text{by }d_-^2T_{k-1}=d_-^2) \\
d_+^* d_-^2 - d_- d_+^* d_- (T_{k-1} + q) + q d_-^2 d_+^* &= 0\quad (\text{by }\times -1).
\end{align*}

Multiplying both sides by $q+1$ from right we get
\begin{align*}
d_+^* d_-^2 (q+1)- d_- d_+^* d_- (T_{k-1} + q)(q+1) + q d_-^2 d_+^*(q+1)&=0\\
\end{align*}
Now the $q+1$ in the first and the last term can be replaced by $q+T_{k-1}$. This is because
$$d_-^2=d_-^2 T_{k-1}, \qquad \text{and }\qquad d_-^2 d_+^* T_{k-1} =d_-^2T_k d_+^*=d_-^2d_+^*.$$

Thus the equality we need to prove becomes
$$A (T_{k-1}+q)=0,$$
where we have set
$$A =d_+^* d_-^2 - (1 + q)d_-d_+^* d_- + q d_-^2 d_+^*.
$$

By Lemma \ref{l-im-ker} the image of $T_{k-1}+q$ is symmetric in $y_{k-1}$, $y_k$.
It is enough to show that $A$ vanishes on elements of $V_k$ that are symmetric in $y_{k-1}$, $y_k$.
By linearity, we need only show that $A$ vanishes on $s_\lambda\ast y^a$.

We have (recall that $k\geq 2$), by Lemma \ref{l-astm},
$$
A(F\ast G) = F \ast_1 A (G),\quad A y_i  = y_{i+1} A \quad (F\in\Sym[X],\, G\in V_k,\, i<k-1).
$$

Thus it is enough to verify the vanishing of $A$ on symmetric polynomials of $y_{k-1}, y_k$. We evaluate $A$ on $y_{k-1}^a y_k^b$:
\[
A (y_{k-1}^a y_k^b) = (\Gamma_+(t(q-1)y_1) B_{a+1} B_{b+1} - (q+1) B_{a+1} \Gamma_+(t(q-1)y_1) B_{b+1}
\]
\[
+ q B_{a+1} B_{b+1} \Gamma_+(t(q-1)y_1) )\ (1),
\]
where $\Gamma_+(Z)$ is the operator $F[X]\to F[X+Z]$. For any monomial $u$ and integer $i$ we have operator identities
$$
\Gamma_+(u) B_i = (B_i - u B_{i-1}) \Gamma_+(u),\quad
B_i \Gamma_+(-u) = \Gamma_+(-u)(B_i - u B_{i-1}),
$$
which can be checked as follows.
\begin{align*}
  \Gamma_+(u) B_i F[X] &= \Gamma_+(u) F[X-(q-1)/z] \pExp[-zX]\Big|_{z^i} \\
                       &= F[X+u-(q-1)/z] \pExp[-z(X+u)] \Big|_{z^i} \\
                       &= F[X+u-(q-1)/z] \pExp[-zX] \pExp[-zu] \Big|_{z^i} \\
                       &= F[X+u-(q-1)/z] \pExp[-zX] (1-zu) \Big|_{z^i} \\
                       &=(B_i-uB_{i-1}) \Gamma_+(u) F[X]
\end{align*}
\begin{align*}
  B_i \Gamma_+(-u) F[X] &= B_i F[X-u] = F[X-u-(q-1)/z] \pExp[-zX] \Big|_{z^i} \\
  &=\Gamma_+(-u) \left(F[X-(q-1)/z] \pExp[-zX]\right) \pExp(-zu)  \Big|_{z^i}\\
  &=\Gamma_+(-u) \left(F[X-(q-1)/z] \pExp[-zX] (1-zu)  \Big|_{z^i} \right)\\
  &=\Gamma_+(-u) (B_i-uB_{i-1}) F[X].
\end{align*}

Thus by $\Gamma_+(t(q-1)y_1)=\Gamma_+(-ty_1) \Gamma_+(tqy_1)$, we have
\begin{align*}
\Gamma_+(t(q-1)y_1) B_{a+1} B_{b+1} &= \Gamma_+(-ty_1) (B_{a+1} - qty_1 B_a) (B_{b+1} - qty_1 B_b) \Gamma_+(qty_1), \\
B_{a+1} \Gamma_+(t(q-1)y_1) B_{b+1} &= \Gamma_+(-ty_1) (B_{a+1} - ty_1 B_a) (B_{b+1} - qty_1 B_b) \Gamma_+(qty_1),\\
B_{a+1} B_{b+1} \Gamma_+(t(q-1)y_1) &= \Gamma_+(-ty_1) (B_{a+1} - ty_1 B_a) (B_{b+1} - ty_1 B_b) \Gamma_+(qty_1).
\end{align*}
Performing the cancellations we arrive at
$$
A (y_{k-1}^a y_k^b) = \Gamma_+(-ty_1) (t y_1 (1-q) (B_{a} B_{b+1} - q B_{a+1} B_b)) (1).
$$
This expression is antisymmetric in $a$, $b$ by Corollary 3.4, \cite{haglund2012compositional}. Thus \eqref{eq:conjrel1} is true.
\end{proof}

Next we have to check that
\begin{prop}
\[
T_1^{-1} (d_+^* d_- - d_- d_+^*) d_+^* = q^{-1} d_+^*(d_+^* d_- - d_- d_+^*).
\]
\end{prop}
\begin{proof}
Multiplying both sides by $q T_1$ from the left and use $T_1d_+^{*2}=d_+^{*2}$, we can rewrite it as
$$d_+^{*2} d_- - (T_1+q) d_+^* d_- d_+^* + q d_- d_+^{*2} = 0.$$
Denote the left hand side by $A$. By linearity, we only need to show that
$A$ vanishes on $F\ast y^a$ for $F\in \Sym[X]$.

We claim that $A (F\ast G)= F\ast_2 AG$ for $F\in \Sym[X]$. We
must be careful about the $T_1+q$ in the middle term, which is only known to commute with polynomials symmetric in $y_1,y_2$.
The argument pass through naturally:
\begin{align*}
  (T_1+q) d_+^* d_-d_+^* F\ast G &=(T_1+q) F\ast_2 d_+^* d_-d_+^* G \\
  &=(T_1+q) F\left[X+(q-1)(ty_1+ty_2) +(q-1)\sum_{i=3}^{k+1} y_i\right] d_+^* d_-d_+^* G\\
&= F\left[X+(q-1)(ty_1+ty_2) +(q-1)\sum_{i=3}^{k+1} y_i\right] (T_1+q)d_+^* d_-d_+^* G\\
&=F\ast_2 (T_1+q)d_+^* d_-d_+^* G.
\end{align*}
It is also easy to check that
$A y_i =y_{i+2} A$ for $i\le k-1$ (be careful about the $T_1+q$ in the middle term).

These twisted commute relation reduce the vanishing of $A$ to $A y_k^a=0$ for all $a\in\Z_{\geq0}$. We obtain
\begin{align*}
A y_k^a &=- h_{a+1}[-X - t(q-1)(y_1+y_2)] + (T_1+q) h_{a+1}[-X-t(q-1)y_1] - q h_{a+1}[-X]\\
&=- h_{a+1}[-X]+(1+q)h_{a+1}[-X]-q h_{a+1}[-X]\\
&\quad +\sum_{b=1}^{a+1} h_{a+1-b}[-X] \left( - h_{b}[t(1-q)(y_1+y_2)] + (T_1+q) h_{b}[t(1-q)y_1] \right),
\end{align*}
where we have used the identity $h_n[X+Y]=\sum_{i+j=n} h_i[X] h_j[Y]$ and collected terms. We need to show each summand for $b>0$ vanishes.

By a direct computation (using (\ref{e-Delta*s}b) and Lemma \ref{l-h1-u}):
\begin{align*}
  (T_1+q) h_{b}[t(1-q)y_1] &= (T_1+q) (1-q) t^b y_1^b
=(1-q) t^b  \left(y_2^b + (1-q)\sum_{i=0}^{b-1} y_2^i y_1^{b-i} + qy_1^b\right)\\
&=(1-q) t^b  (y_2^b + (1-q)\sum_{i=1}^{b-1} y_2^i y_1^{b-i} + y_1^b),
\end{align*}
while
\begin{align*}
h_{b}[t(1-q)(y_1+y_2)]&=t^b\sum_{i=0}^b h_i[(1-q)y_2] h_{b-i}[(1-q)y_1]\\
&=t^b(1-q)\left(y_1^b+y_2^b +(1-q)\sum_{i=1}^{b-1} y^i y_1^{b-i}  \right).
\end{align*}
This completes the proof.
\end{proof}

At this point, we have established the fact that the operators given by \eqref{eq:dplusstar} define an action of $\AA^*$ on $V_*$.  Also we have established the second relation in \eqref{addreleq}. The last relation is obvious. The first and the third are verified below:

\begin{prop}
$$
z_1 d_+ = -t q^{k+1}y_1 d_+^*,\qquad z_{i+1} d_+ = d_+ z_i.
$$
\end{prop}
\begin{proof}
By definition (on $V_{k}$)
$$
z_1 = \frac{q^{k-1}}{q^{-1}-1} (d_+^* d_- - d_- d_+^*) T_{k-1}^{-1} \cdots T_1^{-1},
$$
thus for $G\in V_k$ ($z_1$ is acting on $V_{k+1}$ below)
\begin{align*}
z_1 d_+ G&=  \frac{q^{k}}{q^{-1}-1} (d_+^* d_- - d_- d_+^*) T_{k}^{-1} \cdots T_1^{-1} d_+ G \\
& = \frac{q^{k}}{q^{-1}-1} (d_+^* d_- - d_- d_+^*) G[X+(q-1)y_{k+1}].
\end{align*}
From this expression the following two properties of $z_1 d_+$ are evident, so the operator commute with $y_i$ for $i\le k$:
$$
z_1 d_+ y_i = y_{i+1} z_1 d_+,\qquad z_1 d_+ (F\ast G) = F \ast_1 z_1 d_+(G)
$$
for $F\in \Sym[X]$, $G\in V_k$, $i=1,\ldots,k$.
The following similar properties for $y_1 d_+^*$ can be easily checked.
$$
y_1 d_+^* y_i = y_{i+1} y_1 d_+^*,\qquad y_1 d_+^* (F\ast G) = F \ast_1 y_1 d_+^*(G).
$$
Thus it is enough to verify the first identity when acting on $1\in V_k$. The right hand side is $-t q^{k+1} y_1$. The left hand side is
\begin{align*}
\frac{q^k}{q^{-1}-1}(d_+^*-1)d_-(1)
&= \frac{q^k}{q^{-1}-1}(d_+^*-1) e_1[X] \\
&=\frac{q^k}{q^{-1}-1}(e_1[X]+t (q-1)y_1-e_1[X]) = -t q^{k+1} y_1,
\end{align*}
so the first identity holds.

It is enough to verify the second identity for $i=1$ because the general case can be deduced  from this one by applying the $T$-operators. For the identity $z_2 d_+ = d_+ z_1$, expressing $z_1$, $z_2$ in terms of $d_-$, $d_+^*$ and the $T$-operators, we arrive at the following equivalent identity:
$$
T_1^{-1} d_+(d_+^* d_- - d_- d_+^*) = (d_+^* d_- - d_- d_+^*) d_+.
$$
If we denote by $A$ either of the two sides, we can check that
$$
A (F\ast G) = F\ast_1 A(G),\qquad A y_i = T_2 T_3\cdots T_{i+1} y_{i+1} (T_2 T_3 \cdots T_{i+1})^{-1} A
$$
for $F\in \Sym[X]$, $G\in V_k$, $i=1,\ldots,k-1$. (Note that $d_+$ does not intertwine with the $\ast_1$ but $T_1^{-1} d_+$ does.)
Thus it is enough to verify the identity on $y_k^a\in V_k$ ($a\in\Z_{\geq0}$). Applying $T_k^{-1} T_{k-1}^{-1} \cdots T_2^{-1}$ to both sides, the identity to be verified is
$$
T_k^{-1} T_{k-1}^{-1} \cdots T_1^{-1} d_+ (d_+^* d_- - d_- d_+^*)(y_k^a)
=
(d_+^* d_- - d_- d_+^*) T_{k-1}^{-1} \cdots T_1^{-1} d_+(y_k^a).
$$
Denote by $LHS$ (resp. $RHS$) the left (resp right) hand side of the above identity. Then
$$
LHS=-h_{a+1}[-X-t(q-1)y_1-(q-1)y_{k+1}] + h_{a+1}[-X-(q-1)y_{k+1}].
$$
Detail computation is given as follows.
\begin{align*}
  d_+^* d_-(y_k^a) &=-d_+^* h_{a+1}[-X] =-h_{a+1}[-(X+(q-1)ty_1)],\\
  T_k^{-1} T_{k-1}^{-1} \cdots T_1^{-1} d_+  d_+^* d_- (y_k^a)&=-h_{a+1}[-X-(q-1) ty_1-(q-1)y_{k+1}];\\
  - d_- d_+^*(y_k^a) &= -d_- y_{k+1}^a = h_{a+1}[-X],\\
-T_k^{-1} T_{k-1}^{-1} \cdots T_1^{-1} d_+ d_- d_+^*(y_k^a) &=h_{a+1}[-X-(q-1)y_{k+1}].
\end{align*}

Next we evaluate the right hand side as follows.

\begin{align*}
RHS&= (d_+^* d_- - d_- d_+^*) T_k(y_{k}^a) (\text{ use (\ref{e-Delta*s}b)})\\
   &= (d_+^* d_- - d_- d_+^*) \left( y_{k+1}^a +(1-q)\sum_{i=0}^{a-1} y_k^{a-i} y_{k+1}^{i} \right)\\
   &= d_+^* \left( -h_{a+1}[-X]   -(1-q)\sum_{i=0}^{a-1} y_k^{a-i} h_{i+1}[-X] \right)\\
&\qquad      -d_-\left( y_{k+2}^a +(1-q)\sum_{i=0}^{a-1} y_{k+1}^{a-i} y_{k+2}^{i} \right) \\
    &=  \left( -h_{a+1}[-X-(q-1)ty_1]   -(1-q)\sum_{i=0}^{a-1} y_{k+1}^{a-i} h_{i+1}[-X-(q-1)ty_1] \right)\\
&\qquad     - \left( -h_{a+1}[-X] -(1-q)\sum_{i=0}^{a-1} y_{k+1}^{a-i} h_{i+1}[-X] \right) \\
&=F[X+t(q-1)y_1] - F[X],
\end{align*}
where
\begin{align*}
F[X] &= -h_{a+1}[-X] - (1-q)\sum_{i=0}^{a-1} y_{k+1}^{a-i} h_{i+1}[-X]  \\
&= (1-q)y_{k+1}^{a+1}-h_{a+1}[-X] - (1-q)\sum_{j=0}^{a} y_{k+1}^{a+1-j} h_{j}[-X]\\
&=-h_{a+1}[-X+(1-q)y_{k+1}] +(1-q)y_{k+1}^{a+1},
\end{align*}
and the identity follows.
\end{proof}
This completes our proof of Theorem \ref{Astarthm}.

We also have the following Proposition, which we will use to connect the conjugate
action with $N_{\alpha}$.
\begin{prop}\label{prop:recy}
For a composition $\alpha$ of length $k$ let
$$
y_\alpha = y_1^{\alpha_1-1}\cdots y_k^{\alpha_k-1} \in V_k.
$$
Then the following recursions hold:
$$
y_{1\alpha} = d_+^* y_\alpha,\quad y_{a \alpha} = \frac{t^{1-a}}{q-1}(d_+^* d_- - d_- d_+^*) \sum_{\beta\models a-1} q^{1-l(\beta)} d_-^{l(\beta)-1}(y_{\alpha\beta})\quad(a>1).
$$
\end{prop}
\begin{proof}
The first identity easily follows from the explicit formula for $d_+^*$.

For the second equality, we rename the parameter $a$ by $a-1$ and let $l(\alpha)=k-1$ so that both sides are in $V_k$. Observe that for $i=1,2,\ldots,k-1$ we have
$$
(d_- d_+^* - d_+^* d_-) y_i = y_{i+1} (d_- d_+^* - d_+^* d_-).
$$
Therefore it is enough to verify the following identity for any $a\in\Z_{\geq1}$:
\begin{equation}\label{eq:y1a}
(q-1) t^a y_1^a = (d_+^* d_- - d_- d_+^*) \sum_{\beta\models a}
q^{1-l(\beta)} d_-^{l(\beta)-1}(y_k^{\beta_1-1} \cdots y_{k+l(\beta)-1}^{\beta_{l(\beta)}-1}) \in V_k.
\end{equation}
We group the terms on the right hand side by $b=\beta_1-1$ and the sum becomes
\begin{align*}
\sum_{b=0}^{a-1} &y_k^b \sum_{\beta\models a-b-1} q^{-l(\beta)} d_-^{l(\beta)}\left(y_{k+1}^{\beta_1-1}\cdots y_{k+l(\beta)}^{\beta_{l(\beta)}-1}\right)\\
&=\sum_{b=0}^{a-1} y_k^b \sum_{\beta\models a-b-1} q^{-l(\beta)} (-1)^{l(\beta)} B_{\beta_1} \cdots B_{\beta_{l(\beta)}}(1)\\
&=\sum_{b=0}^{a-1} y_k^b h_{a-b-1}[q^{-1}X],
\end{align*}
where we have used the identity
\begin{equation}\label{eq:hnBn}
h_n[q^{-1}X] = \sum_{\alpha\models n} q^{-l(\alpha)} (-1)^{l(\alpha)} B_\alpha(1),
\end{equation}
which can be obtained by applying $\bar\omega$ to Proposition 5.2 of \cite{haglund2012compositional}:
$$
h_n[-X] = \sum_{\alpha\models n} C_\alpha(1).
$$
Thus the right hand side of (\ref{eq:y1a}) is evaluated to the following expression:
\begin{align*}
(d_+^* d_- -& d_- d_+^*) \sum_{b=0}^{a-1} y_k^b q^{-(a-b-1)} h_{a-b-1}[X]  \\
&
= -\sum_{b=0}^{a-1}\left(\Gamma_{+}(t(q-1)y_1) B_{b+1} - B_{b+1} \Gamma_{+}(t(q-1)y_1)\right) h_{a-b-1}[q^{-1} X]\\
&= - \sum_{b=0}^{a-1} \Gamma_+(-t y_1) \left((B_{b+1} - qt y_1 B_{b}) - (B_{b+1} - t y_1 B_{b})\right) (h_{a-b-1}[q^{-1}X + ty_1])\\
&= (q-1) t y_1 \Gamma_+(-t y_1) \sum_{b=0}^{a-1} B_{b} ( h_{a-b-1}[q^{-1}X+t y_1]).
\end{align*}
Thus we need to prove
$$
\sum_{b=0}^{a-1} B_b (h_{a-b-1}[q^{-1} X+ty_1]) = t^{a-1} y_1^{a-1}.
$$
Then the left hand side as a polynomial in $y_1$ indeed has the right coefficient of $y_1^{a-1}$. The coefficient of $y_1^i$ for $i<a-1$ is
$$
t^i \sum_{b=0}^{a-1-i} B_b(h_{a-b-1-i}[q^{-1} X]).
$$
So it is enough to show:
$$
\sum_{b=0}^m B_b(h_{m-b}[q^{-1}X]) = 0 \quad (m\in\Z_{>0}).
$$
Using (\ref{eq:hnBn}) again we see that the left hand side equals
\begin{align*}
LHS&= B_0(h_m[q^{-1}X]) +\sum_{b=1}^m B_b \sum_{\alpha \models m-b} q^{-l(\alpha)}(-1)^{l(\alpha)} B_\alpha (1)\\
&= B_0(h_m[q^{-1}X])-q  \sum_{\beta \models m} q^{-l(\beta)}(-1)^{l(\beta)} B_\beta(1) \\
&=B_0(h_m[q^{-1}X]) - q h_m[q^{-1}X].
\end{align*}
Finally
$$
B_0(h_m[q^{-1}X]) = B_0 (-q^{-1} C_m(1)) = q (-q^{-1} C_m(B_0(1))) =q h_m[q^{-1} X],
$$
because $B_0 C_m = q C_m B_0$ by Proposition 3.5 of \cite{haglund2012compositional} and $B_0(1)=1$.
\end{proof}

\section{The Involution}
\begin{defn}
Consider $\AA$ and $\AA^*$ as algebras over $\Q(q,t)$, and let
$\tilde{\AA}=\tilde{\AA}_{q,t}$ be the quotient of the free product of
$\AA$ and $\AA^*$ by the relations
\[d_-^* = d_-,\quad T_i^*=T_i^{-1},\quad e_i^*=e_i,\]
\[z_{i+1} d_+ = d_+ z_i,\quad y_{i+1} d_+^* = d_+^* y_i,\quad z_1 d_+ = -t q^{k+1}y_1 d_+^*.\]
\end{defn}
We now prove
\begin{thm}
\label{Atildethm}
The operations $T_i$, $d_-$, $d_+$, $d_+^*$, $e_i$ define an action of $\tilde{\AA}$ on
$V_*$. Furthermore, the kernel of the natural map $\tilde{\AA} e_0 \rightarrow V_*$
that sends $f e_0$ to $f(1)$ is given by $I e_0$ where $I \subset
\tilde{\AA}$ is the left ideal generated by
\begin{equation}
\label{defI}
I=\langle {d_+^*}^m-d_+^{m}|\quad m \geq 1\rangle.
\end{equation}
In particular, we have an isomorphism $V_* \cong \tilde{\AA} e_0/Ie_0$.
\end{thm}

\begin{proof}
Theorem \ref{Astarthm} shows that we have a map of modules
$\tilde{\AA} e_0 \rightarrow V_*$, that restricts to the isomorphism
of Theorem \ref{lem:mainlemma} on the subspace $\AA e_0$, so in
particular is surjective. Furthermore, the last relation
of \eqref{addreleq} shows that it descends to a map
$\tilde{\AA} e_0/Ie_0 \mapsto V_*$, which must still be surjective. We have the following commutative diagram:

\[
\begin{tikzcd}
\tilde{\AA} e_0/Ie_0 \arrow{r} & V_*\\
\AA e_0 \arrow{u} \arrow{ur}{\sim} &
\end{tikzcd}
\]

Thus we have an inclusion $\AA e_0 \subset \tilde{\AA}e_0/Ie_0$ and it remains to show that
the image of $\AA e_0$ in $\tilde{\AA}e_0/Ie_0$ is the entire space.
We do so by induction: notice that both $\AA e_0$ and $\tilde{\AA}e_0/Ie_0$ have a grading by the total degree in $d_+$, $d_+^*$ and $d_-$, as all the relations are homogeneous.
For instance, $y_i$ and $z_i$ have degree $2$, and $T_i$ has degree $0$ for all
$i$. Denote the space of elements of degree $m$ in $\AA e_0$, $\tilde{\AA}e_0/Ie_0$ by $V^{(m)}$, $W^{(m)}$ respectively. We need to prove $V^{(m)}=W^{(m)}$. The base cases $m=0$, $m=1$ are clear (since $I$ contains $d_+^*-d_+$).

For the induction step, suppose $m>0$, $V^{(i)}=W^{(i)}$ for $i\leq m$ and let $F\in V^{(m)}$. It is enough to show that $d_+^* F \in V^{(m+1)}$.
By Lemma \ref{lem:algbasis}, we can assume that
$F$ is in the canonical form \eqref{basiseq}.
We therefore must check three cases: i) $F=d_+^m (1)$ for $1\in V_0$;
ii) $F=y_i G$ for $G\in V^{(m-2)}$; and iii) $F=d_- (G)$ for $G\in V^{(m-1)}$.

Case i) is easy: we have $d_+^* F =d_+^*d_+^{m}(1)=d_+^*{d_{+}^*}^m(1) =  d_+^{m+1} (1) \in V^{(m+1)}$.
For Case ii) we have $d_+^* (F) = y_{i+1} d_+^* (G)$, which is also in $V^{(m+1)}$ by induction hypothesis.

For Case iii) we assume $G \in V_k$. Then we have
\[d_+^* F = d_+^* d_- G = d_- d_+^* G + (q^{-1} - 1) T_1^{-1}\cdots T_{k-1}^{-1} z_k G.\]
The first term is in $V^{(m+1)}$ by induction hypothesis. Since $T_i^{-1}= q^{-1}(T_i+q-1)$, it is sufficient to show $z_j G \in V^{(m+1)}$ for all $j\le k$ (here we only need the $j=k$ case but $z_j$ may appear when commuting with $T_i$).

Now we use expansion of $G\in V_k$ in terms of the generators $T_i$, $d_+$ and $d_-$. Because of the commutation
relations between $T_i$ and $z_j$ it is enough to consider two cases: $G=d_+ G'$ and $G=d_- G'$ for $G'\in V^{(m-2)}$. In the first case we have
$z_j G = d_+ z_{j-1} G'$ if $j>1$ and $z_1 G = -t q^{k} y_1 d_+^*  G'$ if $j=1$ (since $G' \in V_{k-1}$).
In the second case we have $z_j G = d_- z_j G'$. In all cases the claim is reduced to the induction hypothesis.
\end{proof}

%
%

Now by looking at the defining relations of $\tilde{\AA}$, we make
the remarkable observation that there exists an
involution $\iota$ of $\tilde{\AA}$ that permutes $\AA$ and $\AA^*$
and is antilinear with respect to
the conjugation $(q,t)\mapsto (q^{-1},t^{-1})$ on the ground field
$\Q(q,t)$! Furthermore, this involution preserves the ideal $I$\footnote{In the original paper, $I$ was defined to be the left ideal $I'=\langle d_+^* d_+^m-d_+^{m+1}|\quad m \geq 0\rangle$.
It was unclear why $I'$ is preserved under the involution, though
it is not hard to show that
$I'=I$.
},
and therefore induces an involution on $V_*$ via the isomorphism of Theorem
\ref{Atildethm}.

\begin{thm}
\label{mainthm}
There exists a unique antilinear degree-preserving automorphism $\mathcal{N} : V_*
\rightarrow V_*$ satisfying
\[\mathcal{N}(1)=1,\quad \mathcal{N} T_i = T_i^{-1} \mathcal{N},
\quad \mathcal{N} d_- = d_- \mathcal{N},\quad \mathcal{N} d_+ = d_+^*
\mathcal{N},
\quad \mathcal{N} y_i = z_i \mathcal{N}.\]
Moreover, we have
\begin{enumerate}[label=(\roman*)]
\item \label{invpart} $\mathcal{N}$ is an involution, i.e. $\mathcal{N}^2=\Id$.
\item \label{Nalphapart} For any composition $\alpha$ we have
$$
\mathcal{N}(y_\alpha) = q^{\sum(\alpha_i-1)} N_\alpha.
$$
\item \label{nabpart} On $V_0=\Sym[X]$, we have $\mathcal{N} = \nabla
  \bar\omega$, where $\bar\omega$ is the involution sending $q$, $t$, $X$ to $q^{-1}$, $t^{-1}$, $-X$ resp. (see \eqref{omega1}).
\end{enumerate}
\end{thm}
\begin{proof}

The automorphism is induced from the involution of $\tilde{\AA}$, from which
part \ref{invpart} follows immediately. Part \ref{Nalphapart} follows
from applying $\mathcal{N}$ to the relations of Proposition \ref{prop:recy}.
The resulting recursion is the same as the one in Theorem \ref{thm:recN}.

Finally, let $D_1, D_1^*: V_0\to V_0$ be the operators
\[(D_1 F)[X] = F[X+(1-q)(1-t) u^{-1}] \pExp[-uX] |_{u^1},\]
\[(D_1^* F)[X] = F[X-(1-q^{-1})(1-t^{-1}) u^{-1}] \pExp[uX] |_{u^1},\]
and let $\underline{e}_1: V_0\to V_0$ be the operator of multiplication by $e_1[X]=X$.
It is easy to verify that
\[D_1 = -d_- d_+^*,\quad \underline{e}_1 = d_- d_+,\quad \bar\omega D_1 = D_1^*\bar\omega.\]

We check as follows.
\begin{align*}
  -d_-d_+^* F[X] &= -d_- F[X+(q-1)ty_1] \\
                         &= y_1 F[X-(q-1)y_1+(q-1)ty_1] \pExp[-X/y_1] \Big|_{y_1^0} \\
                         &=F[X-(q-1)(t-1) u^{-1}]\pExp[-uX]  \Big|_{u^1} \\
                         &=D_1 F[X].
\end{align*}

\begin{align*}
  d_-d_+ F[X] &= d_- F[X+(q-1)y_1] \\
              &= y_1 F[X] \pExp[-X/y_1] \Big|_{y_1^0} \\
              &=e_1[X] F[X].
\end{align*}

\begin{align*}
  (D_1^* \bar \omega F)[X;q,t] &= (D_1^*  F)[-X;q^{-1},t^{-1}]\\
&=F[-X+(1-q^{-1})(1-t^{-1}) u^{-1};q^{-1},t^{-1}] \pExp[-uX] |_{u^1}\\
&=\bar\omega F[X+(1-q)(1-t) u^{-1};q,t] \pExp[-uX] |_{u^1}\\
&=\bar\omega D_1 F[X].
\end{align*}

Thus it follows that
\[\mathcal{N} D_1 = -\underline{e}_1 \mathcal{N},\quad \mathcal{N} \underline{e}_1 = -D_1 \mathcal{N}.\]
Let $\nabla'=\mathcal{N}\bar\omega$ on $V_0$.
Then
\[\nabla'(1) = 1,\quad \nabla' \underline{e}_1 =
D_1 \nabla',\quad \nabla' D_1^* = -\underline{e}_1 \nabla'.\]
It was shown in \cite{garsia1998explicit} that $\nabla$ satisfies the
same commutation relations,
and that one can obtain all symmetric functions starting from $1$ and
successively applying $\underline{e}_1$ and $D_1^*$. Thus
$\nabla=\nabla'$,
proving part \ref{nabpart}.
\end{proof}

%
%
%

The compositional shuffle conjecture now follows easily:
\begin{thm}
\label{shuffthm}
For a composition $\alpha$ of length $k$, we have
\[\nabla C_{\alpha_1}\cdots C_{\alpha_k}(1)= D_{\alpha}(X;q,t).\]
\end{thm}
\begin{proof}
Using Theorems \ref{thm:recN} and \ref{mainthm}, we have
\[D_{\alpha}(q,t)=d_-^{k} (N_\alpha) = d_-^k(\mathcal{N}(q^{|\alpha|-k} y_\alpha))=\mathcal{N}(q^{|\alpha|-k} d_-^k(y_\alpha))=\]
\[\mathcal{N}\left(q^{|\alpha|-k} (-1)^{k} B_{\alpha}(1)\right) =
\mathcal{N}\bar\omega C_{\alpha}(1)=\nabla
C_{\alpha}(1).\]
Here we used the following fact.
\begin{align*}
  \bar\omega C_a \bar\omega F[X] &=\bar\omega C_a \bar F[-X]\\
 &=\bar\omega -q^{1-a} \bar F[-(X+(q^{-1}-1)z^{-1}] \pExp[zX]\big|_{z^a}\\
&= -q^{a-1} F[X-(q-1)z^{-1}] \pExp[-zX]\big|_{z^a}\\
&=-q^{a-1} B_a F[X].
\end{align*}
\end{proof}

\newcommand\qbinom[2]{\left[#1 \atop #2  \right]_q}

\section{Appendix: An elementary approach to $\chi(\pi)[(q-1)X]$}
We first review some concepts.
Let $\pi$ be a Dyck path of length $n$. Denote by $\Area(\pi)$ the set of area cells $(i,j)$ between $\pi$ and the diagonal.
Define
$$ \chi(\pi) = \sum_{w\in \Z_{>0}^n} q^{\inv(\pi,w)} x_w, \qquad  \chi'(\pi) = \sum_{w\in \Z_{>0}^n ``no\ attack"} q^{\inv(\pi,w)} x_w,$$
where $\inv(\pi,w)$ counts the number of inversion $\{(i,j): w_i>w_j, i<j\}$ with respect to $\pi$, i.e.,
we need the extra condition that $(i,j)\in \Area(\pi)$, or equivalently, the $(i,j)$ cell is under $\pi$; the ``no attack" condition means that for such $(i,j)$ cell we have $w_i\ne w_j$. For instance, if $\pi$ is given by $NNENEE$,
then $w_1>w_3$ is an inversion but is not an inversion with respect to $\pi$: the column 1, row 3 cell is not under $\pi$. See Section \ref{s-example} for examples.

Our goal here is to give an elementary proof of Proposition \ref{prop.q1}, i.e., the following identity:
\begin{align} \label{e-main0}
  \chi(\pi)[(q-1)X] =(q-1)^n \chi'(\pi)[X].
\end{align}

\begin{rem}\label{r-prod}
If $\pi=\pi_1\pi_2$, i.e., a Dyck path $\pi_1$ following by Dyck path $\pi_2$,
then it is easy to see that $\chi(\pi)=\chi(\pi_1)\chi(\pi_2)$ and similarly $\chi'(\pi)=\chi'(\pi_1)\chi'(\pi_2)$. Thus it is sufficient to prove
\eqref{e-main0} for prime Dyck path $\pi$, i.e., $\pi=N \pi' E$ where $\pi'$ is also a Dyck path.
\end{rem}

\subsection{Main idea}
For a word $w$ of length $n$, we denote by $\alpha(w)$ the weak composition $(\alpha_1,\alpha_2,\dots)$ where $\alpha_i$ denotes the number of $i$'s in $w$. We say $w$ is compact if $1,2,\dots, \max(w)$ each appears at least once. Alternatively, $w$ is compact when $\alpha(w)$ is a composition of $n$. Let $CW_n$ be the set of compact words of length $n$.

Let
$$c_\beta(\pi) = \sum_{\alpha(w)=\beta} q^{\inv(\pi, w)},\qquad  c'_\beta(\pi) = \sum_{\alpha(w)=\beta, ``no \ attack"} q^{\inv(\pi, w)}.$$
Then
$$\chi(\pi)[X]= \sum_{\lambda \vdash n} c_\lambda(\pi) m_\lambda[X], \qquad \chi'(\pi)[X]= \sum_{\lambda \vdash n} c'_\lambda(\pi) m_\lambda[X].$$

It is equivalent to show that for each $\mu \vdash n$ we have
\begin{align}
  \label{e-main-mu}
\sum_{\lambda \vdash n} c_\lambda(\pi) m_\lambda[(q-1)X] \big|_{m_\mu}=(q-1)^n c'_\mu(\pi).
\end{align}

Our proof is based on the following known expansion (see Section \ref{s-exp-mmu}).
\begin{align}\label{e-m-lambda}
  m_\lambda[(q-1)X]= \sum_{\mu} m_{\mu}[X] \sum_{\alpha\in R(\lambda), \alpha \text{ is compatible with } \mu} \prod_i (-1)^{l(\alpha^i)}(1-q^{\alpha^i_1}),
\end{align}
where the sum ranges over all rearrangements $\alpha$ of $\lambda$ that is \emph{compatible} with $\mu$, which means that $\alpha$ is the concatenation of $\alpha^1,\dots, \alpha^{l(\mu)}$ such that $|\alpha^i|=\mu_i$ for all $i$.

We present a proof by assuming the truth of Lemma \ref{l-singleton}, which corresponds to the $\mu=(n)$ case and will be proved in the next subsection.

\begin{proof}[Proof of Equation \eqref{e-main-mu}]
By \eqref{e-m-lambda}, the left hand side of equation \eqref{e-main-mu}
becomes
\begin{align*}
LHS &=  \sum_\lambda c_\lambda(\pi) \sum_{\alpha\in R(\lambda), \text{ compatible with } \mu} \prod_i (-1)^{l(\alpha^i)} (1-q^{\alpha^i_1}) \\
&=
  \sum_{\alpha\models n, \text{ compatible with } \mu} c_\alpha(\pi) \prod_i (-1)^{l(\alpha^i)} (1-q^{\alpha^i_1})
\end{align*}
This can be rewritten as
\begin{align}\label{e-mu-n-g}
LHS &= \sum_{w\in CW_n, \alpha(w) \text{ is compatible with } \mu} q^{\inv(\pi,w)}  (-1)^{\max(w)} \prod_i (1-q^{\alpha^i_1(w)}).
\end{align}
Since the sum is over $w$ compatible with $\mu$, we can group $w$
by tuple of sets $\mathcal{S}=(S_1,\dots, S_{l(\mu)})$, where $S_i$ consists of positions $j$ with $w_j$ corresponding to the $i$-th block $\alpha^i$.
More precisely, $S_1$ are positions of $1,2,\dots, l(\alpha^1)$, $S_2$ are positions of $l(\alpha^1)+1,\dots, l(\alpha^1)+l(\alpha^2)$, and so on.
Thus $|S_i|=\mu_i$ and we simply write $|\mathcal S|=\mu$.

Let $(\pi_S, w_S)$ be obtained from $(\pi,w)$ by removing all rows and columns with indices not in $S$.
Observe that if we denote by $\eta(\mathcal{S})$ the word with value $i$ at positions in $S_i$, then
$$ \inv(\pi,w)=\inv(\pi, \eta(\mathcal{S}))+   \sum_{i} \inv(\pi_{S_i},w_{S_i}).$$
We have, by applications of Lemma \ref{l-singleton},
\begin{align*}
  LHS &= \sum_{|\mathcal S|=\mu}  q^{\inv(\pi, \eta(\mathcal{S}))} \prod_i \sum_{w_{S_i}\in CW_{\mu_i}} q^{\inv(\pi_{S_i},w_{S_i})} (-1)^{l(\alpha(w_{S_i}))} (1-q^{\alpha_1(w_{S_i})})\\
&=\sum_{|\mathcal S|=\mu}  q^{\inv(\pi, \eta(\mathcal{S}))} \prod_i (q-1)^{\mu_i} \chi(\area(\pi_{S_i})=0)\\
&=(q-1)^n \sum_{\alpha(w)=\mu, \text{``no attack"}} q^{\inv(\pi,w)}.
\end{align*}
Here we used the standard notation that $\chi(S)$ equals $1$ if the statement $S$ is true and equals $0$ if otherwise. This will cause no confusion from the context. This completes the proof.
\end{proof}

\subsection{The case when $\mu$ is the singleton partition $(n)$}

When $\mu=(n)$ the right hand side of \eqref{e-mu-n-g} evaluates as
\begin{lem} \label{l-singleton}
We have
\begin{align}\label{e-main-g}
  \sum_{w\in CW_n} q^{\inv(\pi,w)}  (-1)^{\max(w)} (1-q^{\alpha_1(w)}) = (q-1)^n \chi(\area(\pi)=0) .
\end{align}
\end{lem}
\begin{proof}
We prove by induction on the length of $\pi$. The base case $n=1$ is clear. Assume Lemma \ref{l-singleton} holds for all smaller $n$, we show that it holds for $n$.

Denote by $F(\pi)$ the left hand side of \eqref{e-main-g}. It is equivalent to show
that
$$F(\pi)=\sum_{w\in CW_n} q^{\inv(\pi,w)}  (-1)^{\max(w)} (1-q^{\alpha_1(w)}) = (q-1)^n \sum_{\alpha(w)=(n), \text{``no attack"}} q^{\inv(\pi,w)}.$$
Observe that this equality holds when $\area(\pi)=0$. Because this is equivalent to the truth of the lemma, and is equivalent to the truth of \eqref{e-main-mu} when $\mu=(n)$. But when $\area(\pi)=0$, we clearly have $\chi(\pi)=\chi'(\pi)=p_1[X]^n$ and thus
$\chi(\pi)[(q-1)X]=(q-1)^n \chi'(\pi)[X]$.

We group the sum according to $S_2=S_2(w)=\{i: w_i=\max(w)\}$. Denote by $S_1=[n] \setminus S_2$ the complement of $S_2$.
We use the same notation as in the proof of \eqref{e-main-mu}.
Observe that the inversion $\inv(\pi,w)$ is decomposed as $\inv(\pi_{S_2},w_{S_2})=0$, $\inv(\pi_{S_1},  w_{S_1})$, and $\inv(\pi, \eta(S_1,S_2))$.

By the induction hypothesis for each $\pi_{S_1}$, we have
\begin{align*}
F(\pi)&=-(1-q^n)+\sum_{k=1}^{n-1} \sum_{|S_2|=k } \sum_{S_2(w)=S_2}
 q^{\inv(\pi_{S_1}, w_{S_1})} q^{\inv(\pi,\eta(S_1,S_2))} (-1)^{\max(w_{S_1})+1} (1-q^{\alpha_1(w_{S_1})}) \\
&=-(1-q^n)-\sum_{k=1}^{n-1} \sum_{|S_2|=k } q^{\inv(\pi, \eta(S_1,S_2))}
 (q-1)^{n-k} \chi(\area(\pi_{S_1})=0).
\end{align*}
This can be written as
\begin{align*}
F(\pi)=  -(1-q^n)-\sum_{\max(w)=2, \text{``no attack on 1"}} q^{\inv(\pi,w)}
 (q-1)^{\alpha_1(w)},
\end{align*}
where ``no attack on 1" means if  $w_i=w_j=1$ then $(i,j)$ is not an area cell of $\pi$.

We prove by induction on $\area(\pi)$ that
$$F(\pi)=(q-1)^n \chi(\area(\pi)=0)$$
using this new formula. The base case $\area(\pi)=0$ has been justified.
Let $\pi'$ be obtained from $\pi$ by removing a peak cell $(a,b)$. Then
$\inv(\pi,w)=\inv(\pi',w)+\chi(w_a>w_b).$
Thus
\begin{align*}
 F(\pi)-F(\pi')&=(1-q)\!\!\!\!\!\!\!\!\!\!\!\! \sum_{w_a=2,w_b=1 \atop \max(w)=2,\text{``no attack on 1"}}\!\!\!\!\!\!\!\!\!\!\!\! q^{\inv(\pi',w)}
 (q-1)^{\alpha_1(w)}
 +\!\!\!\!\!\!\!\!\!\!\!\! \sum_{ w_a=1,w_b=1\atop \max(w)=2,\text{``no attack on 1"}} \!\!\!\!\!\!\!\!\!\!\!\! q^{\inv(\pi',w)}
 (q-1)^{\alpha_1(w)},
\end{align*}
where the second sum is over all $w$ that satisfy the ``no attack" condition in $\pi'$ but does not satisfy the ``no attack" condition in $\pi$.

Now we can pair $w$ in the first summand with $w'$ in the second summand where $w'$ is obtained from
$w$ by replacing $w_a$ by $1$. Then $\alpha_1(w)=\alpha_1(w')-1$ and
$$(1-q) q^{\inv(\pi',w)} (q-1)^{\alpha_1(w)} + q^{\inv(\pi',w')}
 (q-1)^{\alpha_1(w')}=0$$
for each such pair.

For the element $w$ with all $1$'s except for $w_a=2$, $w'$ does not appear in the second sum since $\max(w')=1$, and
$w$ is no attack only when $\area(\pi')=0$. Therefor we obtain
$$F(\pi)-F(\pi') =-(q-1)^n \chi(\area(\pi'=0)),$$
as desired.
\end{proof}

\subsection{Examples\label{s-example}}

We give detailed computation for the case $\pi=NNNEEE$. We need to check the left hand side of \eqref{e-main-g} reduces to $0$. We classify by compositions $\alpha(w)$ of $3$:
a) only $w=111$ has $\alpha(w)=3$ and it contributes $-(1-q^3)$;
b) $\alpha(w)=21$ for $w$ equals to $112, 121, 211$, so they contribute $(1-q^2)(1+q+q^2)$;
c) $\alpha(w)=12$ for $w$ equals to $122, 212, 221$, so they contribute $(1-q)(1+q+q^2)$;
d) $\alpha(w)=111$ for all 6 permutations of $123$, they contribute $-(1-q)(1+q)(1+q+q^2)$.
Case a) cancels with case c)  and case b) cancels with case d). Thus the total sum reduces to $0$.

Two special cases can be verified directly: i) If $\pi$ has area $0$, then $\inv(\pi,w)$ is always $0$, and the number of $w$ with
$\alpha(w)=\beta \models n$ is equal to $\binom{n}{\beta}=\frac{n!}{\beta_1!\beta_2!\cdots}$.
ii) If $\pi$ has the max area $\binom{n}{2}$, then $\inv(\pi,w)$ is the usual inversion on $w$, and it is well-known that
$$\sum_{\alpha(w)=\beta} q^{\inv(w)} = \qbinom{n}{\beta}=\frac{[n]!}{[\beta_1]![\beta_2]!\cdots},$$
where $[a]=1+q+\cdots +q^{a-1}$ and $[a]!=[a][a-1]\cdots [1]$.

Thus Lemma \ref{l-singleton} reduces to the following two identities, which can be checked directly.
\begin{align}\label{e-area-0}
  \sum_{\alpha \models n} (-1)^{l(\alpha)} (1-q^{\alpha_1})\binom{n}{\alpha} &=(q-1)^n.\\
  \sum_{\alpha \models n} (-1)^{l(\alpha)} (1-q^{\alpha_1})\qbinom{n}{\alpha} &=0. \label{e-area-max}
\end{align}
\begin{proof}
We prove \eqref{e-area-0} by using exponential generating function technique.
\begin{align*}
  \sum_{\alpha \models n} (-1)^{l(\alpha)} (1-q^{\alpha_1})\binom{n}{\alpha}
&=-(1-q^n)+\sum_{a=1}^{n-1} \sum_{\beta \models n-a} (-1)^{l(\beta)+1} (1-q^a)  \frac{n!}{a! \beta !} \\
&=-(1-q^n)- \sum_{a=0}^{n-1}(1-q^a) \frac{n!}{a!}  \sum_{\beta \models n-a} \prod_{i} \frac{-1}{\beta_i!} \\
&=-(1-q^n)-\sum_{a=0}^{n-1}(1-q^a) \frac{n!}{a!}  [x^{n-a}] \sum_{k\ge 1} (1-\exp(x))^k \\
&=-(1-q^n)-\sum_{a=0}^{n-1}(1-q^a) \frac{n!}{a!}  [x^{n-a}] \frac{1-\exp(x)}{1-(1-\exp(x))} \\
&=-(1-q^n)-\sum_{a=0}^{n-1}(1-q^a) (-1)^{n-a} \binom{n}{a} \\
&=\sum_{a=0}^n q^a (-1)^{n-a}\binom{n}{a} - \sum_{a=0}^n (-1)^{n-a}\binom{n}{a} \\
&=(q-1)^n.
\end{align*}

The proof of \eqref{e-area-max} is simpler.
\begin{align*}
  \sum_{\alpha \models n} (-1)^{l(\alpha)}& (1-q^{\alpha_1})\binom{n}{\alpha} = \sum_{\alpha \models n} (-1)^{l(\alpha)} (1-q) \frac{[n]!}{[\alpha_1-1]![\alpha_2]!\cdots} \\
&= (1-q) [n] \cdot \left(\sum_{\alpha\models n} (-1)^{l(\alpha)} \qbinom{n-1}{\alpha_1-1,\alpha_2,\alpha_3,\dots}  \right) \\
&= (1-q) [n] \cdot \left(\sum_{\beta\models n-1} (-1)^{l(\beta +1)} \qbinom{n-1}{\beta} + \sum_{\beta\models n-1} (-1)^{l(\beta)} \qbinom{n-1}{\beta} \right) ,
\end{align*}
where we have to split the sum according to whether $\alpha_1$ is equal to 1 or not.
Now the two sum clearly cancels with each other and \eqref{e-area-max} follows.
\end{proof}

\subsection{Expansion of $m_\lambda[(q-1)X]$\label{s-exp-mmu}}
We need some formulas.
\begin{align}
 h_m[-X] =\sum_{\nu \vdash m} (-1)^{l(\nu)} |R(\nu)|  h_\nu=\sum_{\alpha \models m} (-1)^{l(\alpha)} h_{\alpha},
\end{align}
where $R(\nu)$ denotes the set of compositions $\alpha$ with $h_\alpha=h_\nu$, or equivalently, the set of rearrangements of $\nu$.

\begin{proof}
There are many proofs. The following one might be the simplest one.
\begin{align*}
  h_m[-X]=\pExp[-tX] \Big|_{t^{m}}
         &=\frac{1}{1+\sum_{i\ge 1} h_i[X] t^i} \Big|_{t^{m}} \\
         &=\sum_{\alpha \models m} (-1)^{l(\alpha)} h_\alpha\\
         &=\sum_{\nu \vdash m} (-1)^{l(\nu)} |R(\nu)|  h_\nu.
\end{align*}
\end{proof}

\begin{lem}
  We have
\begin{align}
  m_\nu[q-1] =(-1)^{l(\nu)} \sum_{\alpha \in R(\nu)} (1-q^{\alpha_1}),
\end{align}
where $R(\nu)$ denotes the set of rearrangements of the parts of $\nu$.
\end{lem}
\begin{proof}
we have
\begin{align*}
 \sum_{\nu \vdash n} h_\nu [X] m_{\nu}[q-1]&= h_n[(q-1)X] \\
&= \sum_{i=0}^n h_i[qX] h_{n-i}[-X] \\
&=\sum_{i=0}^n q^i h_i[X] \sum_{\beta \models n-i} (-1)^{l(\beta)} h_\beta.
\end{align*}
By comparing coefficients of $h_\nu[X]$ we obtain
\begin{align*}
  m_\nu[q-1] &= (-1)^{l(\alpha)} \sum_{\alpha \in R(\nu)} 1 +\sum_{(i,\beta)\in R(\nu)} (-1)^{l(\beta)-1} q^i\\
             &= (-1)^{l(\nu)} \sum_{\alpha \in R(\nu)} 1 +\sum_{\alpha \in R(\nu)} (-1)^{l(\nu)-1} q^{\alpha_1}\\
             &= (-1)^{l(\nu)} \sum_{\alpha \in R(\nu)} (1-q^{\alpha_1}).
\end{align*}
This completes the proof.
\end{proof}

To find an expansion of $m_\lambda[(q-1)X]$, we compute $h_n[(q-1)XY]$ in two ways.
On one hand, we have
\begin{align}
  h_n[(q-1)XY] = \sum_{\lambda} m_{\lambda}[(q-1)X] h_\lambda[Y].
\end{align}

On the other hand, we have
\begin{align*}
 h_n[(q-1)XY] &= \sum_{\mu} m_{\mu}[X] h_\mu[(q-1)Y] 
              = \sum_{\mu} m_{\mu}[X] \prod_{i} h_{\mu_i}[(q-1)Y] \\
              &= \sum_{\mu} m_{\mu}[X] \prod_{i} \left(\sum_{\nu^i \vdash \mu_i} h_{\nu^i}[Y] m_{\nu^i}[q-1] \right).
\end{align*}
Thus by equating coefficients, we have
\begin{align*}
 m_\lambda[(q-1)X] &= \sum_{\mu} m_{\mu}[X]  \sum_{\nu^i \vdash \mu_i,  (\nu^1, \nu^2,\dots) \in R( \lambda)}  \prod_{i} m_{\nu^i}[q-1] \\
&= \sum_{\mu} m_{\mu}[X] \sum_{\nu^i \vdash \mu_i,  (\nu^1, \nu^2,\dots) \in R( \lambda)} \prod_i m_{\nu^i}[q-1] \\
&= \sum_{\mu} m_{\mu}[X] \sum_{\nu^i \vdash \mu_i,  (\nu^1, \nu^2,\dots) \in R( \lambda)} \prod_i (-1)^{l(\nu^i)}\sum_{\alpha^i \in R(\nu^i)}(1-q^{\alpha^i_1}) \\
&= \sum_{\mu} m_{\mu}[X] \sum_{\alpha^i \models \mu_i,  (\alpha^1, \alpha^2,\dots) \in R( \lambda)} \prod_i (-1)^{l(\alpha^i)}(1-q^{\alpha^i_1}).
\end{align*}

There is at most one way to decompose $\alpha$ as $(\alpha^1,\alpha^2,\dots)$ such that $|\alpha^i|=\mu_i$ for all $i$. When such a decomposition exists, $\alpha$ is said to be compatible with $\mu$, and we denote by $\alpha^i(\mu)=\alpha^i$. Thus we obtain \eqref{e-m-lambda}.

\printbibliography


\end{document}